\pdfoutput=1 
\documentclass[11pt, a4paper, english]{amsart}
\usepackage{amsmath} % utilities for mathematics
\usepackage{amsthm} % utilities for theorem environment
\usepackage{amssymb} % loads mathematical fonts
\usepackage[dvipsnames]{xcolor}
\usepackage{amscd} % utilities for commutative diagrams
\usepackage{mathtools}

\usepackage{subcaption}% <-- added

\usepackage{array} % core package
\usepackage{newtxtext,newtxmath} %fuente

\usepackage[cal=boondoxo,scr=euler]{mathalfa}
\usepackage[backref=page,linktocpage]{hyperref} % permits navigating on the pdf
\usepackage{cleveref} % smart referencing
\usepackage{caption} %
\usepackage{graphics,graphicx} % permits putting images
\usepackage{tikz,tikz-cd} % tool for diagrams

\usetikzlibrary{graphs,graphs.standard,calc}

\usetikzlibrary{decorations.pathreplacing,}%angles,quotes

\usepackage{enumerate} % permits enumerating using letters or other symbols

\DeclareMathAlphabet{\mathsf}{OT1}{\sfdefault}{m}{n}

\newcommand{\nocontentsline}[3]{}
\newcommand{\tocless}[2]{\bgroup\let\addcontentsline=\nocontentsline#1{#2}\egroup}

\usepackage[margin=1.20in]{geometry}
\usepackage{verbatim}

\usepackage{scalerel}

\makeatletter
\def\dual#1{\expandafter\dual@aux#1\@nil}
\def\dual@aux#1/#2\@nil{\begin{tabular}{@{}c@{}}#1\\#2\end{tabular}}
\makeatother

\makeatletter
\@namedef{subjclassname@2020}{\textup{2020} Mathematics Subject Classification}
\makeatother

\newcommand{\stirlingtwo}[2]{\biggl\{\genfrac{}{}{0pt}{}{#1}{#2}\biggr\}}
\newcommand{\tstirlingtwo}[2]{\left\{\genfrac{}{}{0pt}{}{#1}{#2}\right\}}

\DeclareMathAlphabet{\amathbb}{U}{bbold}{m}{n}

\hypersetup{
    colorlinks = true,
    linkbordercolor = {white},
    linkcolor = {BrickRed},
    anchorcolor = {black},
    citecolor = {BrickRed},
    filecolor = {cyan},
    menucolor = {BrickRed},
    runcolor = {cyan},
    urlcolor = {black}
}

\usetikzlibrary{automata}

\newtheoremstyle{teoremas}% <name>
{11pt}% <Space above>
{11pt}% <Space below>
{\itshape}% <Body font>
{}% <Indent amount>
{\bfseries}% <Theorem head font>
{}% <Punctuation after theorem head>
{.5em}% <Space after theorem headi>
{}% <Theorem head spec (can be left empty, meaning `normal')>

\theoremstyle{teoremas}
\newtheorem{theorem}{Theorem}[section]
\newtheorem{corollary}[theorem]{Corollary}
\newtheorem{lemma}[theorem]{Lemma}
\newtheorem{proposition}[theorem]{Proposition}

\newtheoremstyle{definition}% <name>
{11pt}% <Space above>
{11pt}% <Space below>
{}% <Body font>
{}% <Indent amount>
{\bfseries}% <Theorem head font>
{}% <Punctuation after theorem head>
{.5em}% <Space after theorem headi>
{}% <Theorem head spec (can be left empty, meaning `normal')>

\theoremstyle{definition}
\newtheorem{definition}[theorem]{Definition}
\newtheorem{conjecture}[theorem]{Conjecture}

\newtheorem{problem}[theorem]{Problem}

\newtheorem{example}[theorem]{Example}
\newtheorem{remark}[theorem]{Remark}

\crefname{theorem}{theorem}{theorems}
\Crefname{theorem}{Theorem}{Theorems}
\crefname{lemma}{lemma}{lemmas}
\Crefname{lemma}{Lemma}{Lemmas}
\crefname{proposition}{proposition}{propositions}
\Crefname{proposition}{Proposition}{Propositions}

\DeclareMathOperator{\rk}{rk}

\DeclareMathOperator{\Poin}{Poin}
\DeclareMathOperator{\Tor}{Tor}
\DeclareMathOperator{\Ext}{Ext}

\newcommand{\M}{\mathsf{M}}
\newcommand{\N}{\mathsf{N}}
\newcommand{\U}{\mathsf{U}}
\newcommand{\cI}{\mathcal{I}}

\newcommand{\Q}{\mathbb{Q}}

\newcommand{\Z}{\mathbb{Z}}

\newcommand{\Hilb}{\operatorname{Hilb}}

\renewcommand{\H}{\mathrm{H}}
\newcommand{\CH}{\mathrm{CH}}
\newcommand{\aug}{\operatorname{aug}}
\newcommand{\IH}{\mathrm{IH}}
\newcommand{\uH}{\underline{\mathrm{H}}}
\newcommand{\uCH}{\underline{\mathrm{CH}}}

\newcommand{\cL}{\mathcal{L}}
\newcommand{\Rep}{\operatorname{Rep}}
\newcommand{\gr}{\operatorname{gr}}
\newcommand{\VRep}{\operatorname{VRep}}
\newcommand{\Ind}{\operatorname{Ind}}
\newcommand{\Res}{\operatorname{Res}}
\newcommand{\Aut}{\operatorname{Aut}}
\newcommand{\col}{\operatorname{col}}
\newcommand{\asc}{\operatorname{asc}}
\newcommand{\bad}{\operatorname{bad}}
\newcommand{\des}{\operatorname{des}}

\AtBeginDocument{%
   \def\MR#1{}
}

\title[Hilbert series of matroids]{Hilbert--Poincar\'e series of matroid\\ Chow rings and intersection cohomology}

\author[L. Ferroni]{Luis Ferroni}

\address{(L. Ferroni)
  Department of Mathematics, KTH Royal Institute of Technology, Stockholm, Sweden
}
\email{ferroni@kth.se}

\thanks{LF is supported by the Swedish
Research Council, Grant 2018-03968. }

\author[J. Matherne]{Jacob P. Matherne}

\address{(J. Matherne)
Department of Mathematics, North Carolina State University, Raleigh, NC, USA.
}
\email{jpmather@ncsu.edu}

\thanks{JM was supported by the Max Planck Institute for Mathematics in Bonn and the Deutsche Forschungsgemeinschaft (DFG) under Germany's Excellence Strategy - GZ 2047/1, Projekt-ID 390685813.}

\author[M. Stevens]{Matthew Stevens}
\address{(M. Stevens) Department of Mathematics, University of Pennsylvania, Philadelphia, PA, USA}
\email{mcs0042@sas.upenn.edu}

\author[L. Vecchi]{Lorenzo Vecchi}

\address{(L. Vecchi)
  Department of Mathematics, KTH Royal Institute of Technology, Stockholm, Sweden
}
\email{lvecchi@kth.se}

\subjclass[2020]{Primary: 05B35, 14C15, 52B40, 13D40, 05A15. Secondary: 52B45, 05E18, 05E45, 30C15.}

\allowdisplaybreaks

\begin{document}

\begin{abstract}
    We study the Hilbert series of four objects arising in the Chow-theoretic and Kazhdan--Lusztig framework of matroids. These are, respectively, the Hilbert series of the Chow ring, the augmented Chow ring, the intersection cohomology module, and its stalk at the empty flat. (The last two are known as the $Z$-polynomial and the Kazhdan--Lusztig polynomial, respectively.)
    We develop an explicit parallelism between the Kazhdan--Lusztig polynomial of a matroid and the Hilbert--Poincar\'e series of its Chow ring. This extends to a parallelism between the $Z$-polynomial of a matroid and the Hilbert--Poincar\'e series of its augmented Chow ring. This suggests to bring ideas from one framework to the other. Our two main motivations are the real-rootedness conjecture for all of these polynomials, and the problem of computing them.
    
    We provide several intrinsic definitions of these invariants via recursions they satisfy.
    Uniform matroids are a case of combinatorial interest; we link the resulting polynomials with certain real-rooted families appearing in combinatorics such as the Eulerian and the binomial Eulerian polynomials, and we settle a conjecture of Hameister, Rao, and Simpson.  Furthermore, we prove the real-rootedness of the Hilbert series of the augmented Chow rings of uniform matroids via a technique introduced by Haglund and Zhang; and in addition, we prove a version of a conjecture of Gedeon in the Chow setting: uniform matroids maximize coefficient-wise these polynomials for matroids with fixed rank and cardinality.
    By relying on the nonnegativity of the coefficients of the Kazhdan--Lusztig polynomials and the semi-small decompositions of Braden, Huh, Matherne, Proudfoot, and Wang, we strengthen the unimodality of the Hilbert series of Chow rings, augmented Chow rings, and intersection cohomologies to $\gamma$-positivity, a property for palindromic polynomials that lies between unimodality and real-rootedness; this also settles a conjecture of Ferroni, Nasr, and Vecchi. 
\end{abstract}

\maketitle

%\setcounter{tocdepth}{2}
%\tableofcontents %This is only to facilitate navigating (in my opinion the article is not long enough to deserve a table of contents on its first page)

\section{Introduction}\label{sec:introduction}

\subsection{Overview} Starting from a loopless matroid $\M$, one can construct various algebraic objects having good properties that in turn can be used to answer purely combinatorial questions. Notable examples of such objects are the \textbf{Chow ring} $\uCH(\M)$, the \textbf{augmented Chow ring} $\CH(\M)$, and the \textbf{intersection cohomology module} $\IH(\M)$. These three algebraic structures possess a number of remarkable features and play an instrumental role in the proofs of the log-concavity of the Whitney numbers of the first kind \cite{adiprasito-huh-katz}, and the top-heaviness of the Whitney numbers of the second kind of the lattice of flats $\mathcal{L}(\M)$ \cite{braden-huh-matherne-proudfoot-wang}. Our main objects of study in this article are the coefficients of the respective Hilbert--Poincar\'e series of each of these structures. 

The intersection cohomology module $\IH(\M)$ is particularly relevant in the Kazhdan--Lusztig theory of matroids. Its Hilbert--Poincar\'e series is known in the literature as the \textbf{$Z$-polynomial} of the matroid $\M$ and is denoted by $Z_{\M}(x)$. In \cite{proudfoot-xu-young} Proudfoot, Xu, and Young introduced this polynomial using a purely combinatorial language, i.e. without making reference to any (at that point) meaningful notion of intersection cohomology for general matroids. Later, in the work of Braden, Huh, Matherne, Proudfoot, and Wang a precise description of the intersection cohomology was obtained. In fact, the $Z$-polynomial is strongly related to the \textbf{Kazhdan--Lusztig polynomial} of $\M$, denoted $P_\M(x)$ and first studied by Elias, Proudfoot, and Wakefield in \cite{elias-proudfoot-wakefield}. The following can be used as a simultaneous recursive definition for both the Kazhdan--Lusztig and $Z$-polynomials of matroids.

\begin{theorem}[{\cite{proudfoot-xu-young,braden-vysogorets}}]\label{thm:definition-kl-and-zeta}
    There is a unique way to assign to each loopless matroid $\M$ a polynomial $P_{\M}(x) \in \Z[x]$ such that the following properties hold:
    \begin{enumerate}[\normalfont(i)]
        \item If $\rk(\M) = 0$, then $P_{\M}(x) = 1$.
        \item If $\rk(\M) > 0$, then $\deg P_{\M}(x) < \frac{1}{2} \rk(\M)$.
        \item For every matroid $\M$, the polynomial
            \[ Z_{\M}(x) := \sum_{F\in \mathcal{L}(\M)} x^{\rk(F)}\, P_{\M/F}(x)\]
        is palindromic.\footnote{In this context, we say that $p(x)$ is palindromic if $p(x) = x^d p(x^{-1})$ where $d = \deg p(x)$. For example, the polynomial $q(x) = x^2+x$ is \emph{not} palindromic although it satisfies $x^3q(x^{-1})=q(x)$.}
    \end{enumerate}
\end{theorem}

The original definition of the Kazhdan--Lusztig polynomials of Elias, Proudfoot, and Wakefield is given by the following result, which provides a recursion that defines them uniquely in terms of characteristic polynomials of restrictions $\M|_F$ and Kazhdan--Lusztig polynomials of contractions $\M/F$ for flats $F\in\mathcal{L}(\M)$. In contrast to Theorem~\ref{thm:definition-kl-and-zeta}, this does not make any reference to the $Z$-polynomial.

\begin{theorem}[\cite{elias-proudfoot-wakefield}]\label{thm:kazhdan-lusztig}
    There is a unique way to assign to each loopless matroid $\M$ a polynomial $P_{\M}(x)\in \mathbb{Z}[x]$ such that the following conditions hold:
    \begin{enumerate}[\normalfont (i)]
        \item If $\rk(\M) = 0$, then $P_{\M}(x) = 1$.
        \item If $\rk(\M) > 0$, then $\deg P_{\M}(x) < \frac{1}{2} \rk(\M)$.
        \item For every matroid $\M$, the following recursion holds:
            \[ x^{\rk(\M)} P_{\M}(x^{-1}) = \sum_{F\in\mathcal{L}(\M)} \chi_{\M|_F}(x)\, P_{\M/F}(x).\]
    \end{enumerate}
\end{theorem}

These results provide compact and purely combinatorial definitions of $P_{\M}(x)$ and $Z_{\M}(x)$. However, many properties of these families of polynomials are not easy to deduce from such statements. For instance, 
a non-obvious property that these polynomials possess is the nonnegativity of their coefficients. In fact, the second main result of \cite{braden-huh-matherne-proudfoot-wang} establishes this by proving that the coefficients of these polynomials are given by graded dimensions, more precisely, $Z_{\M}(x) = \Hilb(\IH(\M), x)$ and $P_{\M}(x) = \Hilb(\IH(\M)_{\varnothing},x) $ (see \cite[Theorem~1.9]{braden-huh-matherne-proudfoot-wang}).

We now turn history on its head. Recall that in the matroid Kazhdan--Lusztig setting, the combinatorially-defined Kazhdan--Lusztig and $Z$-polynomials came first, and their descriptions as Hilbert--Poincar\'e series of $\IH(\M)_\varnothing$ and $\IH(\M)$ came later.  For the Chow-theoretic setting, the Chow ring $\uCH(\M)$ and augmented Chow ring $\CH(\M)$ were first introduced in \cite{feichtner-yuzvinsky} and \cite{semismall}, respectively.  In what follows, we give an ``intrinsic'' combinatorial definition of their Hilbert--Poincar\'e series by mirroring Theorem~\ref{thm:definition-kl-and-zeta} and Theorem~\ref{thm:kazhdan-lusztig}.  (Theorem~\ref{thm:correcthpseries} below asserts that these polynomials are indeed the correct Hilbert--Poincar\'e series.)  These combinatorial definitions can be seen as the starting point of our study.

\begin{theorem}\label{thm:main-recursion-defi-H-and-uH}
    There is a unique way to assign to each loopless matroid $\M$ a palindromic polynomial $\uH_{\M}(x) \in \Z[x]$ such that the following properties hold:
    \begin{enumerate}[\normalfont(i)]
        \item If $\rk(\M) = 0$, then $\uH_{\M}(x) = 1$.\label{it:mainfirst}\
        \item If $\rk(\M) > 0$, then $\deg \uH_{\M}(x) = \rk(\M) - 1$.\label{it:mainsecond}
        \item For every matroid $\M$, the polynomial
            \[ \H_{\M}(x) := \sum_{F\in \mathcal{L}(\M)} x^{\rk(F)}\, \uH_{\M/F}(x)\]
        is palindromic.\label{it:mainthird}
    \end{enumerate}
\end{theorem}

The polynomial $\uH_{\M}(x)$ will be sometimes referred to as the \textbf{Chow polynomial} of the matroid $\M$, whereas $\H_{\M}(x)$ will be called the \textbf{augmented Chow polynomial} of $\M$.\footnote{We thank Nick Proudfoot for proposing these names.} Similar to the case of the Kazhdan--Lusztig polynomials and the $Z$-polynomials, a non-obvious property that one may observe after working out some examples is that Chow polynomials and augmented Chow polynomials of matroids appear to have nonnegative coefficients. (The reason for the name of these polynomials and the nonnegativity of their coefficients is given by Theorem~\ref{thm:correcthpseries} below.) Additionally, given the resemblance between Theorem~\ref{thm:main-recursion-defi-H-and-uH} and Theorem~\ref{thm:definition-kl-and-zeta}, it is reasonable to ask for a counterpart for Theorem~\ref{thm:kazhdan-lusztig} in this alternative setting.

\begin{theorem}\label{thm:intro-main0}
    There is a unique way to assign to each loopless matroid $\M$ a polynomial $\uH_{\M}(x)\in \mathbb{Z}[x]$ such that the following conditions hold:
    \begin{enumerate}[\normalfont (i)]
        \item If $\rk(\M) = 0$, then $\uH_{\M}(x) = 1$.
        \item For every matroid $\M$, the following recursion holds:
            \[ \uH_{\M}(x) = \sum_{\substack{F\in\mathcal{L}(\M)\\ F\neq\varnothing}} \overline{\chi}_{\M|_F}(x)\, \uH_{\M/F}(x).\]
    \end{enumerate}
\end{theorem}

The proof of this fact is a straightforward induction. What is less evident is that the polynomials $\uH_{\M}(x)$ arising from this result coincide with those arising from Theorem~\ref{thm:main-recursion-defi-H-and-uH}. In particular, it is already an interesting conclusion that the polynomials as defined in Theorem~\ref{thm:intro-main0} are palindromic and have nonnegative coefficients, as this is not at all hinted by the recursion since the \emph{reduced} characteristic polynomial has coefficients that alternate in sign.

Somewhat against intuition, the polynomials $\uH_{\M}(x)$ and $\H_{\M}(x)$ are from some perspectives less manageable than $P_{\M}(x)$ and $Z_{\M}(x)$, especially from a computational point of view. For example, their coefficients are much larger, and neither $\uH_{\M}(x)$ nor $\H_{\M}(x)$ behave well under direct sums of matroids; one of the many possible explanations for this phenomenon is that the reduced characteristic polynomial of a direct sum is not the product of the reduced characteristic polynomials of the summands and, in particular, the arguments that Elias, Proudfoot, and Wakefield used in \cite[Proposition~2.7]{elias-proudfoot-wakefield} to prove that $P_{\M_1\oplus \M_2}(x) = P_{\M_1}(x)P_{\M_2}(x)$ do not hold. As a result, even simple operations within matroid theory (such as ``adding a coloop'') change the polynomials $\uH_{\M}(x)$ and $\H_{\M}(x)$ in a less evident way.

The reason that the coefficients of these polynomials are always nonnegative integers is given by the following connection with the Chow ring and the augmented Chow ring.

\begin{theorem}\label{thm:correcthpseries}
    Let $\M$ be a loopless matroid. The polynomial $\uH_{\M}(x)$ is the Hilbert--Poincar\'e series of the Chow ring $\uCH(\M)$. The polynomial $\H_{\M}(x)$ is the Hilbert--Poincar\'e series of the augmented Chow ring $\CH(\M)$. In particular, both of them have nonnegative coefficients. 
\end{theorem}

The proof relies on a construction by Feichtner and Yuzvinsky \cite{feichtner-yuzvinsky} of a certain Gr\"obner basis for the Chow ring of atomic lattices with respect to arbitrary building sets. The strategy is to start with a raw expression for the Hilbert--Poincar\'e series of both the Chow ring and the augmented Chow ring and prove that they satisfy the recursions of both Theorem~\ref{thm:main-recursion-defi-H-and-uH} and Theorem~\ref{thm:intro-main0}.

\subsection{Real-rootedness and \texorpdfstring{$\gamma$}{gamma}-positivity}

Naturally, one can view our previous two results as ``intrinsic'' combinatorial definitions of the Hilbert series of $\uCH(\M)$ and $\CH(\M)$ that avoid their explicit construction. Another consequence is that some statements (many of which are still conjectural) regarding the Kazhdan--Lusztig polynomial and the $Z$-polynomial of matroids now admit a ``natural'' counterpart for the Chow ring and the augmented Chow ring. In particular, two outstanding conjectures posed in \cite{gedeon-proudfoot-young-survey} and \cite{proudfoot-xu-young} assert the \textbf{real-rootedness} of both $P_{\M}(x)$ and $Z_{\M}(x)$. A well-known fact (see, e.g. \cite{branden}) is that the real-rootedness of a polynomial with positive coefficients implies that the coefficients form an \textbf{ultra log-concave sequence} and, in the case the polynomial is palindromic, that it is \textbf{$\gamma$-positive}. Each of these two properties, i.e. ultra log-concavity or $\gamma$-positivity, is also known to be stronger than the \textbf{unimodality} of the coefficients of the polynomial.

The similarity between the defining recursions of the polynomials arising in Theorem~\ref{thm:definition-kl-and-zeta} and Theorem~\ref{thm:main-recursion-defi-H-and-uH}, along with a significant amount of experiments\footnote{These experiments include all matroids on up to $9$ elements, all lattice path matroids on up to $14$ elements, sparse paving matroids on up to $40$ elements, and all the matroids of the catalog in \cite{oxley} that have cardinality at most $12$.}, make it plausible to postulate that these desirable properties of the coefficients might hold for $\uH_{\M}(x)$ and $\H_{\M}(x)$ as well. Indeed, these conjectural properties have already been observed in different sources.  %We collect them in the following conjecture.

\begin{conjecture}\label{conj:real-rootedness}
    For every matroid $\M$, the following polynomials have only real roots.
    \begin{itemize}
        \item (\cite[Conjecture~8.18]{ferroni-schroter}) The polynomial $\uH_{\M}(x)$, i.e. the Hilbert--Poincar\'e series of the Chow ring $\uCH(\M)$.
        \item (\cite[Conjecture~4.3.3]{stevens-bachelor}) The polynomial $\H_{\M}(x)$, i.e. the Hilbert--Poincar\'e series of the augmented Chow ring $\CH(\M)$.
        \item (\cite[Conjecture~5.1] {proudfoot-xu-young})  The polynomial $Z_{\M}(x)$, i.e. the Hilbert--Poincar\'e series of the intersection cohomology module $\IH(\M)$.
        \item (\cite[Conjecture~3.2]{gedeon-proudfoot-young-survey})  The polynomial $P_{\M}(x)$, i.e. the Hilbert--Poincar\'e series of the stalk $\IH(\M)_{\varnothing}$ at the empty flat of $\IH(\M)$.
    \end{itemize}
\end{conjecture}

The first two of the above four assertions were observed and conjectured by June Huh in a private communication with the third author in 2019.  

\begin{remark}
    We note that $\CH(\M)$ and $\IH(\M)$ are both modules over the graded M\"obius algebra $\H(\M)$ (see Section~\ref{subsec:ih} for its definition).  We make this point only to be able to define $\IH(\M)_\varnothing$ (see Definition~\ref{def:ih}), which depends on the $\H(\M)$-module structure.  As the purpose of this paper is to study various Hilbert--Poincar\'e series, we only need to view $\H(\M)$, $\uCH(\M)$, $\CH(\M)$, $\IH(\M)$, and $\IH(\M)_\varnothing$ as graded vector spaces.
\end{remark}

The main result of Adiprasito, Huh, and Katz \cite[Theorem~1.4]{adiprasito-huh-katz} asserts the validity of the \textbf{K\"ahler package} for the Chow rings of matroids. In particular, \textbf{Poincar\'e duality} and the \textbf{hard Lefschetz property} are valid and they respectively imply the palindromicity and the unimodality of the coefficients of the Hilbert series of the Chow ring. On the other hand, the K\"ahler package was proved for the augmented Chow ring by Braden, Huh, Matherne, Proudfoot, and Wang in \cite[Theorem~1.6]{semismall} and for the intersection cohomology module in \cite[Theorem~1.6]{braden-huh-matherne-proudfoot-wang}, so one can conclude analogous statements for the Hilbert--Poincar\'e series of $\CH(\M)$ and $\IH(\M)$. Less is known regarding the Hilbert--Poincar\'e series of $\IH(\M)_{\varnothing}$, i.e. the Kazhdan--Lusztig polynomial $P_{\M}(x)$. Although its coefficients are nonnegative, it is not known whether they are always unimodal in general.

Another of our contributions will be to provide a proof of the $\mathbf{\gamma}$-positivity for three of the families of polynomials above. (The notion of $\gamma$-positivity is only meaningful for palindromic polynomials, so we only ask this question for these families.) This gives further evidence to part of Conjecture~\ref{conj:real-rootedness} and resolves a problem posed in \cite[Conjecture 5.6]{ferroni-nasr-vecchi} by Ferroni, Nasr, and Vecchi that was previously known to hold only in certain cases.

\begin{theorem}\label{thm:main-gamma-positivity}
    For every matroid $\M$, the polynomials $\uH_{\M}(x)$, $\H_{\M}(x)$, and $Z_{\M}(x)$ are $\gamma$-positive.
\end{theorem}

The proofs of the $\gamma$-positivity of each of these families depend on a recursion witnessed by the validity of the semi-small decompositions of \cite{semismall} for the Chow ring and the augmented Chow ring. For the intersection cohomology, a recursion found by Braden and Vysogorets \cite{braden-vysogorets} does the job, but a crucial requirement is the nonnegativity of the coefficients of the Kazhdan--Lusztig polynomials.

\subsection{The case of uniform matroids}\label{subsec:uniform-matroids} Computing all of the polynomials mentioned so far is, in general, a very difficult challenge. Even for uniform matroids, the task of producing convenient formulas is of much interest. For the $Z$-polynomial and the Kazhdan--Lusztig polynomial of uniform matroids, some explicit formulas are known, see, e.g. \cite{gao-uniform}.\footnote{We point out that the real-rootedness of $P_{\M}(x)$ and $Z_{\M}(x)$ for $\M\cong \U_{k,n}$ are still open problems.} On the other hand, prior to this work, no clear formulas for $\uH_\M(x)$ and $\H_\M(x)$ of arbitrary uniform matroids were known. We mention, however, that the polynomials $\uH_{\M}(x)$ for arbitrary uniform matroids were addressed by Hameister, Rao, and Simpson in \cite{hameister-rao-simpson}. One of their main results \cite[Theorem~5.1]{hameister-rao-simpson} is useful to retrieve the fact that, for a Boolean matroid $\M\cong \U_{n,n}$, one has $\uH_{\U_{n,n}}(x)=A_n(x)$, where $A_n$ is the $n$-th \emph{Eulerian polynomial}; whereas for uniform matroids of corank $1$, it follows that  $\uH_{\U_{n-1,n}}(x) = \frac{1}{x}d_n(x)$, where $d_n(x)$ denotes the $n$-th \emph{derangement polynomial}. Unfortunately, deducing compact expressions for Chow polynomials of arbitrary uniform matroids via their result seems rather difficult; although their formula is nice in terms of statistics of permutations, it is intricate from a computational point of view.

Without making any references to augmented Chow rings of matroids, the polynomial $\H_{\M}(x)$ for Boolean matroids has been studied in detail recently in \cite{postnikov-reiner-williams,athanasiadis-eulerian,shareshian-wachs,haglund-zhang,binhan,branden-jochemko}. To be precise, one has $\H_{\U_{n,n}}(x)=\widetilde{A}_n(x)$ where $\widetilde{A}_n(x)$ denotes the $n$-th \emph{binomial Eulerian polynomial}. Notice that, when applied to a Boolean matroid the recursion in Theorem~\ref{thm:main-recursion-defi-H-and-uH}~\eqref{it:mainthird} asserts the known identity
    \[ \widetilde{A}_n(x) = \sum_{j=0}^n \binom{n}{j} x^j\, A_{n-j}(x).\]
A further motivation to study the Hilbert series of Chow rings and augmented Chow rings is that they may be viewed as vast generalizations of the Eulerian, binomial Eulerian, and derangement polynomials, and within this broader framework one can derive  interesting new identities relating them. The following result provides a compact expression for the Chow polynomials and augmented Chow polynomials of general uniform matroids.

\begin{theorem}\label{thm:uniform-formulas-main}
    The Hilbert series of the Chow ring and augmented Chow ring of arbitrary uniform matroids are given by:
    \begin{align*}
        \uH_{\U_{k,n}}(x) &= \enspace \sum_{j=0}^{k-1}\, \binom{n}{j} \, d_j(x) (1+x+\cdots+x^{k-1-j}),\\
        \H_{\U_{k,n}}(x) &= 1 + x\sum_{j=0}^{k-1} \binom{n}{j}\, A_j(x) (1+x+\cdots+x^{k-1-j}).
    \end{align*}
\end{theorem}

The first of the above two formulas, when $k=n-1$ yields a recursion for the derangement polynomials deduced in recent work by Juhnke-Kubitzke, Murai, and Sieg in \cite[Corollary~4.2]{juhnke-murai-sieg}.

On the other hand, we can leverage the parallelism between the Kazhdan--Lusztig and the Chow-theoretic settings to produce alternative formulas that allow one to compute $\uH_{\U_{k,n}}(x)$ and $\H_{\U_{k,n}}(x)$, by a different method. Indeed, we imitate one of the main tricks used in the literature to compute $P_{\U_{k,n}}(x)$ and $Z_{\U_{k,n}}(x)$, which consists of considering the incidence algebra of the lattice of flats and the \textbf{inverse Kazhdan--Lusztig polynomial} of Gao and Xie \cite{gao-xie}. In the Chow setting this elementary idea yields useful formulas that apply to matroids in general, and uniform matroids in particular. As an application of this, we prove a conjecture of Hameister, Rao, and Simpson \cite[Conjecture 6.2]{hameister-rao-simpson} and an analogue for the augmented case. For the precise statements of those formulas, we refer directly to Proposition~\ref{prop:hilbert-chow-inverse-incidence} and Proposition~\ref{prop:hilbert-aug-chow-inverse-incidence}. 

We lend support to the part of Conjecture~\ref{conj:real-rootedness} stating the real-rootedness of the Hilbert series of augmented Chow rings. 

\begin{theorem}\label{thm:intro-hz}
     For every uniform matroid $\U_{k,n}$, the augmented Chow polynomial $\H_{\U_{k,n}}(x)$ is real-rooted.
\end{theorem}

The proof of the preceding result relies on a combination of Theorem~\ref{thm:uniform-formulas-main} and a technical tool introduced by Haglund and Zhang \cite{haglund-zhang} on interlacing sequences of polynomials.

\subsection{Dominance of uniform matroids} 
An unpublished conjecture of Gedeon states that the Kazhdan--Lusztig and $Z$-polynomials of $\U_{k,n}$ are coefficient-wise maximal among all matroids of rank $k$ on $n$ elements. This conjecture is known to hold for all paving matroids in the Kazhdan--Lusztig setting \cite[Theorem~1.5]{ferroni-nasr-vecchi}. Here, we prove a Chow-theoretic counterpart for all matroids.
%Furthermore, we discuss the Chow-theoretic version of an unpublished conjecture posed by Gedeon. It asserts that the maximum of the coefficients of the Kazhdan--Lusztig and $Z$-polynomial of a matroid of rank $k$ and size $n$ is attained when the matroid is isomorphic to $\U_{k,n}$. In the Kazhdan--Lusztig framework, this conjecture is known to hold for all paving matroids \cite[Theorem~1.5]{ferroni-nasr-vecchi}. Here we prove that its counterpart for the polynomials $\uH_{\M}(x)$ and $\H_{\M}(x)$  holds for \emph{all} matroids.

\begin{theorem}\label{thm:dominance-uniform-main}
    Let $\M$ be a loopless matroid of rank $k$ on a ground set of size $n$. The following coefficient-wise inequalities hold:
    \begin{align*}
        \uH_{\M}(x) &\preceq \uH_{\U_{k,n}}(x),\\
        \H_{\M}(x) &\preceq \H_{\U_{k,n}}(x).
    \end{align*}
    In other words, uniform matroids maximize 
    the coefficients of Chow polynomials and augmented Chow polynomials among all matroids with fixed rank and size.
\end{theorem}

\subsection{Outline}

The paper is organized as follows. In Section~\ref{sec:two} we review all the necessary background on matroids, $\gamma$-positivity, and the definitions of the derangement polynomials and the (binomial) Eulerian polynomials. 
In Section~\ref{sec:three} we discuss the Chow ring and augmented Chow ring. The validity of Theorem~\ref{thm:main-recursion-defi-H-and-uH}, Theorem~\ref{thm:intro-main0}, and Theorem~\ref{thm:correcthpseries} is a consequence of Theorem~\ref{thm:intrinsic-main-body} and Theorem~\ref{thm:convolution-reduced-char-poly}.
We prove the part of Theorem~\ref{thm:main-gamma-positivity} asserting $\gamma$-positivity for $\uH_\M(x)$ and $\H_\M(x)$ in Theorem~\ref{thm:gamma-positivity-uH-H}. The discussion on incidence algebras and the formulas for uniform matroids happens in Section \ref{subsec:incidence-uniform}; this includes the verification of the conjecture of Hameister, Rao, and Simpson and the proof of Theorem~\ref{thm:uniform-formulas-main}. The proof of Theorem~\ref{thm:intro-hz} is briefly motivated in Section~\ref{subsec:haglund-zhang}, but mostly deferred to Appendix~\ref{appendixb}. Theorem~\ref{thm:dominance-uniform-main} is discussed in Section~\ref{subsec:dominance-uniform}.
In Section~\ref{sec:four} we put our attention on the Kazhdan--Lusztig invariants. We establish the $\gamma$-positivity of the $Z$-polynomial as Theorem~\ref{thm:z-is-gamma-positive}.
In Section~\ref{sec:five} we collect some concluding remarks revolving around Conjecture~\ref{conj:real-rootedness}. In particular, we prove the real-rootedness of $\H_{\M}(x)$ when $\M$ is a uniform matroid via a construction due to Haglund and Zhang; additionally, we pose some conjectures and problems, discuss some aspects of total positivity, and survey the consequences of Koszulness of the Chow ring and  augmented Chow ring of a matroid applied to the Hilbert series of our interest.

\section{Background}\label{sec:two}

\subsection{Matroids}

Although we will assume familiarity with the usual techniques and constructions in matroid theory, in this subsection we recall some basic notions to establish the notation we will use throughout the paper. For any undefined concept that appears, we refer to Oxley's book on matroid theory \cite{oxley}. 

\begin{definition}
    A \emph{matroid} $\M$ is a pair $(E,\mathscr{B})$, where $E$ is a finite set and $\mathscr{B}\subseteq 2^E$ is a family of subsets of $E$ that satisfies the following two conditions:
    \begin{enumerate}[(a)]
        \item $\mathscr{B}\neq \varnothing$.
        \item If $B_1\neq B_2$ are members of $\mathscr{B}$ and $a\in B_1\smallsetminus B_2$, then there exists an element $b\in B_2\smallsetminus B_1$ such that $(B_1\smallsetminus \{a\})\cup \{b\}\in \mathscr{B}$.
    \end{enumerate}
\end{definition}

The set $E$ is usually called the \emph{ground set} and the members of $\mathscr{B}$ the \emph{bases} of $\M$. The most basic example of a matroid is that of \emph{uniform matroids}. We write $\U_{k,n}$ for the uniform matroid of rank $k$ and cardinality $n$. The matroid $\U_{k,n}$ is defined by $E=[n]$ and $\mathscr{B} = \binom{[n]}{k}$. The uniform matroid $\U_{n,n}$ of rank $n$ with $n$ elements will be referred to as the \emph{Boolean matroid} of rank $n$. The matroid $\U_{0,0}$, the only matroid having the empty set as its ground set, will be called the \emph{empty matroid}.

An important operation in matroid theory is that of \emph{dualization}. If $\M = (E,\mathscr{B})$ is a matroid, then by considering the set $\mathscr{B}^*=\{E\smallsetminus B: B\in\mathscr{B}\}$, it can be shown that $\M^*:= (E,\mathscr{B}^*)$ is a matroid as well---this is called the \emph{dual matroid} of $\M$. For example, the dual of the uniform matroid $\U_{k,n}$ is precisely $\U_{n-k,n}$.

\begin{definition}
    Let $\M=(E,\mathscr{B})$ be a matroid.
    \begin{itemize}
        \item If $e\in E$ is in no basis of $\M$, then we say that $e$ is a \emph{loop}. If $e$ is in all the bases, then we say that $e$ is a \emph{coloop}.
        \item An \emph{independent set} is a set $I\subseteq E$ contained in some $B\in \mathscr{B}$. If a set is not independent, we will say that it is \emph{dependent}.
        \item A \emph{circuit} of $\M$ is a minimal dependent set. 
        \item The \emph{rank} of an arbitrary subset $A\subseteq E$ is given by
            \[ \rk_{\M}(A) := \max_{B\, \in \,\mathscr{B}}\, |A\cap B|.\]
        We define the rank of $\M$ to be $\rk(\M) := \rk_{\M}(E)$.
        \item A \emph{flat} of $\M$ is a set $F$ with the property that adjoining new elements to $F$ strictly increases its rank. The poset of all flats of $\M$ ordered by inclusion happens to be a lattice called the \emph{lattice of flats}, and we will  denote it by $\mathcal{L}(\M)$. A flat of rank $\rk(\M)-1$ is said to be a \emph{hyperplane}.
    \end{itemize}
\end{definition}

Two operations that we will use often are those of restriction and contraction. We briefly comment about the fact that these operations are usually written using different notations in different articles. We therefore stick to Oxley's notation \cite{oxley} to avoid any confusion. 

\begin{definition}
    Let $\M=(E,\mathscr{B})$ be a matroid and let $A\subseteq E$.
    \begin{itemize}
    \item The matroid $\M|_A$, called the \emph{restriction of $\M$ to $A$}, is the matroid on the ground set $A$ with bases $\mathscr{B}(\M|_A) := \{B\cap A: B\in \mathscr{B} \text{ and } |B\cap A| = \rk_{\M}(A)\}$.
    \item The matroid $\M/A$, called the \emph{contraction of $\M$ by $A$}, is the matroid $(\M^*|_{E\smallsetminus A})^*$.
    \end{itemize}
    Matroids of the form $(\M|_{A_1})/{A_2}$ for some subsets $A_2\subseteq A_1$ of $E$ are called \emph{minors} of $\M$.
\end{definition}

If $\M$ is a loopless matroid on $E$, the minimum element of the lattice of flats $\mathcal{L}(\M)$ corresponds to the empty set $\varnothing$. Moreover, if $A$ is a flat, the operations of restriction and contraction can be understood nicely in $\mathcal{L}(\M)$: the lattice of flats of $\M|_A$ is isomorphic (as a poset) to the interval $[\varnothing,A]$ in $\mathcal{L}(\M)$, whereas the lattice of flats of $\M/A$ is isomorphic to the interval $[A,E]$ in $\mathcal{L}(\M)$. We will use these identifications  often.

An invariant that one may associate to a bounded graded poset is its \emph{characteristic polynomial}. In the case of the lattice of flats $\mathcal{L}(\M)$ of a loopless matroid $\M$, it is defined by 
    \begin{equation}\label{eq:definition-char-poly}
        \chi_{\M}(x)=\chi_{\mathcal{L}(\M)}(x) := \sum_{F\in \mathcal{L}(\M)} \mu(\varnothing, F) x^{\rk(E)-\rk(F)},
    \end{equation}
where $\mu\colon\mathcal{L}(\M)\times\mathcal{L}(\M)\to \mathbb{Z}$ denotes the M\"obius function (defined below) of the lattice $\mathcal{L}(\M)$. Using the above displayed formula as the definition, we have $\chi_{\U_{0,0}}(x)=1$ for the empty matroid. On the other hand, it is customary to extend the definition to arbitrary matroids by setting $\chi_{\M}(x):=0$ whenever $\M$ has loops.

In general, we define the \emph{M\"obius function} of a finite poset $\mathcal{L}$ as the map $\mu\colon\mathcal{L}\times\mathcal{L}\to \mathbb{Z}$ satisfying the recursion
    \begin{equation} \label{eq:mobius}
    \mu(x, y) = 
    \begin{cases}
        0 & \text{if }x\not\leq y,\\
        1 & \text{if }x = y,\\
        -\displaystyle\sum_{x\leq z < y} \mu(x, z) & \text{if }x < y. \end{cases}\end{equation}

If one evaluates both sides of equation \eqref{eq:definition-char-poly} at $x=1$, the recursion that defines the M\"obius function yields
\[ \chi_{\M}(1) = \sum_{F\in\mathcal{L}(\M)} \mu(\varnothing, F) = \begin{cases} 0 & \text{ if } E\neq \varnothing,\\ 1 & \text{ if } E=\varnothing.
\end{cases}\]

In other words, for a nonempty loopless matroid $\M$, the polynomial $\chi_{\M}(x)$ is a multiple of $x-1$. In particular, when $\M$ is nonempty, the quotient $\frac{\chi_{\M}(x)}{x-1}$ is itself a polynomial.

\begin{definition}
    Let $\M$ be a loopless matroid. We define the \emph{reduced characteristic polynomial} of $\M$ as the polynomial
    \[ \overline{\chi}_{\M}(x) = 
    \begin{cases} -1 & \text{ if $\M$ is empty,}\\
    \dfrac{\chi_{\M}(x)}{x-1} & \text{ otherwise.}
    \end{cases}\]
\end{definition}

The choice we made of how to define the reduced characteristic polynomial for empty matroids will be useful to simplify some statements later. Since this is not a standard definition, every time we rely on it, we will indicate that explicitly. We now state a useful lemma that we will need later, in Section~\ref{subsec:self-rec-and-their-geometry}.

\begin{lemma}\label{lemma:sum-reduced-char-poly-contractions}
    Let $\M$ be a loopless matroid on $E$, then
    \[ \sum_{\substack{F\in \mathcal{L}(\M)\\F\neq E}} \overline{\chi}_{\M/F}(x) = 1 + x + x^2 + \cdots + x^{\rk(\M) - 1}.\]
\end{lemma}

\begin{proof}
    Notice that none of the contractions is the empty matroid. In particular, after multiplying both sides by $x-1$ and adding $1 = \chi_{\U_{0,0}}(x)$, proving the above identity is equivalent to proving that \[\sum_{F\in \mathcal{L}(\M)} \chi_{\M/F}(x) = x^{\rk(\M)}.\]
    This is obtained by applying the M\"obius inversion formula (see for example \cite[Proposition~3.7.2]{stanley-ec1}) to the definition of the characteristic polynomial in equation \eqref{eq:definition-char-poly}.
\end{proof}

\subsection{\texorpdfstring{$\gamma$}{gamma}-polynomials and \texorpdfstring{$\gamma$}{gamma}-positivity}

A polynomial $f(x) = \sum_i a_i x^i$ is said to be \emph{symmetric} with center of symmetry $\frac{d}{2}$, if $a_i = a_{d-i}$ for each $i\in \mathbb{Z}$ (where $a_i := 0$ for negative values of $i$). This can be rephrased by saying that $f(x)$ satisfies $x^df(x^{-1})=f(x)$. If the center of symmetry is $\frac{d}{2}$ for $d=\deg f(x)$, the polynomial $f(x)$ will be called  \emph{palindromic}.

It is often useful to express symmetric polynomials in different bases. For example, the family of polynomials $\left\{x^i(1+x+\cdots+x^{d-2i})\right\}_{i=0}^{\lfloor\frac{d}{2}\rfloor}$ is a basis for the space of all symmetric polynomials with center of symmetry $\frac{d}{2}$. It is not difficult to see that $f(x)$ can be written as a nonnegative linear combination in this basis if and only if the sequence $a_0,\ldots, a_d$ is nonnegative, symmetric, and unimodal. A property stronger than unimodality arises when changing the basis to $\left\{x^i(1+x)^{d-2i}\right\}_{i=0}^{\lfloor\frac{d}{2}\rfloor}$.

\begin{proposition}
    If $f(x)\in \mathbb{Z}[x]$ is a polynomial such that $x^df(x^{-1})=f(x)$, then there exist integers $\gamma_0,\ldots,\gamma_{\lfloor \frac{d}{2}\rfloor}$ such that
    \begin{equation}\label{eq:gamma-exp} 
        f(x) = \sum_{i=0}^{\lfloor\frac{d}{2}\rfloor} \gamma_i\, x^i (1+x)^{d-2i}.
    \end{equation}
\end{proposition}

\begin{proof}
    See \cite[Proposition~2.1.1]{gal}.
\end{proof}

\begin{definition}
    Let $f(x) \in \Z[x]$ be a polynomial such that $x^df(x^{-1})=f(x)$. If $\gamma_0,\ldots,\gamma_{\lfloor\frac{d}{2}\rfloor}$ are as in equation \eqref{eq:gamma-exp}, we define the \emph{$\gamma$-polynomial} associated to $f$ by
        \[\gamma(f,x) := \sum_{i=0}^{\lfloor\frac{d}{2}\rfloor} \gamma_i\, x^i.\]
\end{definition}

If $f(x)$ is a palindromic polynomial of degree $d$, we will say that $f(x)$ is \emph{$\gamma$-positive} if all the coefficients of $\gamma(f,x)$ are nonnegative. For a thorough study of $\gamma$-positivity in combinatorics, we refer the reader to the survey \cite{athanasiadis-gamma-positivity} by Athanasiadis. As mentioned before, being $\gamma$-positive is a stronger property than having unimodal coefficients. On the other hand, a well-known fact is that if $f(x)$ has only negative real roots then it is $\gamma$-positive. Let us record these observations in a proposition.

\begin{proposition}
    Let $f$ be a symmetric polynomial with nonnegative coefficients. We have the following strict implications.
    \[\gamma(f,x) \text{ is real-rooted} \iff f(x) \text{ is real-rooted} \Longrightarrow f(x) \text{  is $\gamma$-positive} \Longrightarrow f(x) \text{ is unimodal.}\]
\end{proposition}

The proofs of these implications, as well as counterexamples for the missing ones, can be found in \cite{gal} and \cite{branden} or, alternatively, the reader can look at \cite[Proposition~5.3]{ferroni-nasr-vecchi}.

Observe that a nonzero symmetric polynomial $f(x)\in \mathbb{Z}[x]$ has a unique center of symmetry. In other words, there is exactly one integer $d$ such that $x^df(x^{-1}) = f(x)$. In particular, whenever $f$ is symmetric there is no ambiguity in writing $\gamma(f,x)$, even when the degree of the polynomial $f(x)$ is not specified. We have the following toolbox of basic identities that exhibit the behavior of the assignment $f(x) \mapsto \gamma(f,x)$ under simple operations.

\begin{lemma}\label{lemma:properties-gamma}
    Let $f(x)$ and $g(x)$ be symmetric polynomials. Then, we have:
    \begin{enumerate}[\normalfont(i)]
        \item $\gamma(fg, x) = \gamma(f,x)\cdot\gamma(g,x)$.
        \item $\gamma(xf, x) = x\cdot\gamma(f,x)$.
        \item $\gamma((x+1)f, x) = \gamma(f,x)$.
        \item If $f(x)$ and $g(x)$ have the same center of symmetry, then $\gamma(f+g,x) = \gamma(f,x) + \gamma(g,x)$.
    \end{enumerate}
\end{lemma}

\subsection{Derangements and (binomial) Eulerian polynomials}

We present three families of polynomials that will be of much importance when we deal with Hilbert series of Chow rings and augmented Chow rings of uniform and Boolean matroids.

\subsubsection{Eulerian polynomials} 

One of the most pervasive objects in enumerative combinatorics is the family of  Eulerian polynomials. Given a permutation $\sigma\in\mathfrak{S}_n$, $n\geq 1$, written in one-line notation as $\sigma=\sigma_1\cdots\sigma_n$, the number of \emph{descents} of $\sigma$ is defined as the cardinality of the set $\{i\in [n-1]:\sigma_i>\sigma_{i+1}\}$, and is denoted by $\des(\sigma)$. For $n\geq 1$, we define the $n$-th \emph{Eulerian polynomial} as follows:
    \[ A_n(x) := \sum_{\sigma\in\mathfrak{S}_n} x^{\des(\sigma)}.\] 
We note that this differs by a factor of $x$ from the definition in \cite[p.~33]{stanley-ec1}. Furthermore, we define $A_0(x)=1$. Explicitly, we have:
    \[ A_n(x) = \begin{cases}
        1 & \text{if } n = 0,1,\\
        x+1 & \text{if } n = 2,\\
        x^2+4x+1 & \text{if } n=3,\\
        x^3+11x^2+11x+1 & \text{if } n=4,\\
        x^4+26x^3+66x^2+26x+1 & \text{if } n=5,\\
        \text{etc.} &
    \end{cases}\]
The polynomials $A_n(x)$ are palindromic and $\deg A_n(x) = n-1$. It is a classical result attributed to Frobenius that these polynomials are real-rooted (for a proof, see \cite[Example 7.3]{branden}). The coefficients of the Eulerian polynomials admit several combinatorial interpretations, many of which can be found in \cite[Chapter 1]{stanley-ec1}.

\subsubsection{Derangement polynomials}

A permutation $\sigma\in\mathfrak{S}_n$ is said to be a \emph{derangement} if $\sigma_i\neq i$ for all $i$, i.e. if $\sigma$ has no fixed points. The set of all derangements on $n$ elements is usually denoted by $\mathfrak{D}_n$. For each $n\geq 1$, the $n$-th \emph{derangement polynomial}, denoted $d_n(x)$, is defined by
    \[ d_n(x) := \sum_{\sigma\in\mathfrak{D}_n} x^{\operatorname{exc}(\sigma)}.\]
where $\operatorname{exc}(\sigma):=|\{i\in [n]:\sigma_i>i\}|$ denotes the number of \emph{excedances} of $\sigma$. We extend this definition to $n=0$ by taking $d_0(x):=1$. The first few values of $d_n(x)$ are:
    \[ d_n(x) = \begin{cases}
        1 & \text{if } n = 0,\\
        0 & \text{if } n = 1,\\
        x & \text{if } n = 2,\\
        x^2+x& \text{if } n=3,\\
        x^3+7x^2+x & \text{if } n=4,\\
        x^4+21x^3+21x^2+x & \text{if } n=5,\\
        \text{etc.} &
    \end{cases}\]

We have $\deg d_n(x) = n-1$ for each $n\geq 1$. With only the exception of $d_0(x)$, the polynomials $d_n(x)$ are a multiple of $x$ and are symmetric with center of symmetry $\frac{n}{2}$.  Derangement polynomials are known to be real-rooted (see for example \cite[Theorems 3.5 and 4.1]{gustafsson-solus}), and therefore are in particular $\gamma$-positive.

\subsubsection{Binomial Eulerian polynomials}

A related family of polynomials that will play an important role is that of the \emph{binomial Eulerian polynomials}, which were named this way, e.g. in \cite{shareshian-wachs,athanasiadis-eulerian}. The $n$-th binomial Eulerian polynomial $\widetilde{A}_n(x)$ is defined by
    \[ \widetilde{A}_n(x) := 1 + x \sum_{j=1}^{n} \binom{n}{j} A_j(x).\]
In particular, the first few values of these polynomials are given by:
    \[ \widetilde{A}_n(x) = \begin{cases}
        1 & \text{if } n = 0,\\
        x+1 & \text{if } n = 1,\\
        x^2+3x+1 & \text{if } n=2,\\
        x^3+7x^2+7x+1 & \text{if } n=3,\\
        x^4+15x^3+33x^2+15x+1 & \text{if } n=4,\\
        x^5+31x^4+131x^3+131x^2+31x+1 & \text{if } n=5,\\
        \text{etc.} &
    \end{cases}\]
We note that for every $n\geq 0$, $\deg\widetilde{A}_n(x)=n$. It is a non-obvious fact that these polynomials are palindromic and $\gamma$-positive, see for example \cite[Theorem~11.6]{postnikov-reiner-williams} or \cite[Theorem~1.1]{athanasiadis-eulerian}. Furthermore, they are real-rooted, by \cite[Theorem~3.1]{haglund-zhang} or \cite[Theorem 4.4]{branden-jochemko}.

\section{The Chow ring and the augmented Chow ring}\label{sec:three}

\subsection{The definitions}

In \cite{feichtner-yuzvinsky} Feichtner and Yuzvinsky introduced the notion of Chow ring for an arbitrary finite atomic lattice. Their construction takes as inputs an atomic lattice $\mathcal{L}$ and a so-called \emph{building set} $\mathcal{G}\subseteq\mathcal{L}$, and returns a ring $D(\mathcal{L},\mathcal{G})$ that they refer to as \emph{the Chow ring of $\mathcal{L}$ with respect to $\mathcal{G}$}. For the definition of building set and \emph{irreducible element}, we refer to \cite[Definition~1]{feichtner-yuzvinsky}. Let us denote by $\widehat{0}$ and $\widehat{1}$, respectively, the bottom and top elements of $\mathcal{L}$. There are two distinguished building sets; the first, which is the one we care the most about, is called the \emph{maximal} building set $\mathcal{G}_{\max}$ and is given by $\mathcal{L}\smallsetminus\{\widehat{0}\}$. The second is the \emph{minimal} building set $\mathcal{G}_{\min}$, and it consists of all elements of $\mathcal{L}\smallsetminus\{\widehat{0}\}$ that are irreducible.

The terminology ``Chow ring'' has its roots in the classical notion of Chow ring in algebraic geometry, and particularly in toric geometry. More precisely, if one considers the lattice of intersections of a hyperplane arrangement, i.e. the lattice of flats of a representable matroid, and an arbitrary building set containing the top element, the resulting Chow ring is isomorphic to the Chow ring of the associated wonderful compactification of De~Concini and Procesi \cite{deconcini-procesi}. Since the notion presented by Feichtner and Yuzvinsky is meaningful for arbitrary geometric lattices (as opposed to only ``representable'' ones), we can define a notion of Chow ring for the lattice of flats of an arbitrary (loopless) matroid. What is usually called the ``Chow ring of a  matroid'' in the literature is obtained by taking the maximal building set of its lattice of flats.

By making a slight modification to its presentation, and following the notation of \cite{semismall}, we introduce the Chow ring of a matroid using the following definition.

\begin{definition}
    Let $\M$ be a loopless matroid. The \textit{Chow ring} of $\M$ is the quotient
        \[
        \uCH(\M) = \mathbb{Q}[x_F : F\in\mathcal{L}(\M)\smallsetminus\{\varnothing,E\}]/{(I+J)},
        \]
    where the ideals $I$ and $J$ are defined respectively by
    \begin{align*}
        I &= \left< x_{F_1} x_{F_2} \,:\, F_1,F_2 \in \mathcal{L}(\M)\smallsetminus\{\varnothing,E\} \text{ are incomparable}\right>,\\
        J &= \left< \sum_{F\ni i} x_F - \sum_{F\ni j} x_F \,:\, i,j\in E\right>.
    \end{align*}
\end{definition}

In \cite{semismall} and \cite{braden-huh-matherne-proudfoot-wang}, Braden, Huh, Matherne, Proudfoot, and Wang introduced an ``augmented'' version of the Chow ring of a matroid. 
For the interested reader, we mention that the augmented Chow ring of a matroid $\M$ can be defined in terms of the construction of Feichtner and Yuzvinsky, namely as the Chow ring of the lattice of flats of the free coextension of $\M$ with respect to a certain building set; we refer to \cite[Lemma 5.14]{stellahedral} for more details regarding this perspective. The definition of augmented Chow ring that we will use is the following.

\begin{definition}
     Let $\M$ be a loopless matroid. The \textit{augmented Chow ring} of $\M$ is the quotient
        \[
        \CH(\M) = \mathbb{Q}[x_F,\, y_i : F\in\mathcal{L}(\M)\smallsetminus\{E\} \text{ and } i\in E]/{(I+J)},
        \]
    where the ideals $I$ and $J$ are defined respectively by
    \begin{align*}
        I &= \left< y_i - \sum_{F\notni i} x_F \,:\, i\in E\right>,\\
        J &= \left< x_{F_1} x_{F_2} \,:\, F_1,F_2 \in \mathcal{L}(\M)\smallsetminus\{E\} \text{ are incomparable}\right> + \left< y_i x_F :F\in\mathcal{L}(\M)\smallsetminus\{E\},\, i \notin F\right>.
    \end{align*}
\end{definition}

For a fixed loopless matroid $\M$ of rank $k$, both the Chow ring $\uCH(\M)$ and the augmented Chow ring $\CH(\M)$ are graded rings, each admitting a decomposition of the form
    \begin{equation} \label{eq:grading-chow}
    \uCH(\M) = \bigoplus_{j=0}^{k-1} \uCH^j(\M),\qquad\CH(\M) = \bigoplus_{j=0}^{k} \CH^j(\M).
    \end{equation}
The generating function of the graded dimensions of these rings, i.e. their \textit{Hilbert--Poincar\'e series}, or just \textit{Hilbert series}, are defined respectively by
    \begin{align*} 
    \uH_{\M}(x) &:= \Hilb(\uCH(\M), x) = \sum_{j=0}^{k-1} \dim_{\mathbb{Q}}(\uCH^j(\M)) \, x^j,\\ \H_{\M}(x) &:= \Hilb(\CH(\M), x) = \sum_{j=0}^{k} \dim_{\mathbb{Q}}(\CH^j(\M)) \, x^j.
    \end{align*}
Notice that both $\uH_{\M}(x)$ and $\H_{\M}(x)$ are polynomials, having degree $k-1$ and $k$ respectively. The main result of \cite{adiprasito-huh-katz} asserts the validity of the ``K\"ahler package'' in $\uCH(\M)$. In particular, Poincar\'e duality and the hard Lefschetz property hold. This automatically implies that the coefficients of $\uH_{\M}(x)$ are palindromic and form a unimodal sequence. As a consequence of \cite{semismall}, we have analogous statements for $\CH(\M)$, which therefore guarantee that $\H_{\M}(x)$ is a palindromic polynomial having unimodal coefficients as well.

\begin{example}
    In Section~\ref{subsec:incidence-uniform} we will address the Hilbert series of the Chow rings of uniform matroids $\U_{k,n}$. When a uniform matroid is Boolean, i.e. $k=n$, one has the following formula:
    \[ \uH_{\U_{n,n}}(x) = A_n(x),\]
    where $A_n(x)$ is the $n$-th Eulerian polynomial. A geometric way of establishing the identity $\uH_{\U_{n,n}}(x) = A_n(x)$ relies on the fact that the $h$-vector of the $n$-th \emph{permutohedron} $\underline{\Pi}_n$ is given by the Eulerian numbers and that the Chow ring of $\U_{n,n}$ is isomorphic to the Chow ring of the permutohedral variety. In Hampe's work \cite{hampe}, the Chow ring of $\U_{n,n}$ is also referred to as ``the intersection ring of matroids''.
\end{example}

\begin{example}
    Later, in Section~\ref{subsec:incidence-uniform}, we will address a general formula for the Hilbert series of the augmented Chow ring of a uniform matroid $\U_{k,n}$. For a Boolean matroid, one has the following formula:
    \[ \H_{\U_{n,n}}(x) = \widetilde{A}_n(x),\]
    where $\widetilde{A}_n(x)$ is the $n$-th binomial Eulerian polynomial. The coefficients of these polynomials coincide with the $h$-vector of the $n$-th \emph{stellahedron} $\Pi_n$. This is a consequence of $\CH(\U_{n,n})$ being isomorphic to the Chow ring of the stellahedral variety, see, e.g. \cite{stellahedral} for further details. 
\end{example}

\subsection{Flags of flats and convolutions}

A natural question that arises when studying Hilbert series of (augmented) Chow rings is how to actually compute them. A naive approach is to compute the polynomials by constructing the full ring from the definition, 
but the computational cost of this method is excessively high for all
but the smallest matroids.
For example, a small matroid such as $\M\cong\U_{4,8}$ has $94$ flats, so both the Chow ring and the augmented Chow ring are defined as a quotient of a polynomial ring with nearly $100$ variables. 
% Even using a computer, it quickly becomes infeasible to construct the full ring and then extract the Hilbert series. 

One of the main results of Feichtner and Yuzvinsky, \cite[Proposition~1]{feichtner-yuzvinsky}, simplifies the computational challenge by providing an explicit Gr\"obner basis for the Chow ring of an atomic lattice with respect to an arbitrary building set. When translated into the setting of Hilbert series of Chow rings of matroids, their result yields the following formula.

\begin{proposition}\label{prop:hilbert-chow}
    Let $\M$ be a loopless matroid. The Hilbert--Poincar\'e series of the Chow ring $\uCH(\M)$ is given by
        \[ \uH_{\M}(x) = \sum_{\varnothing = F_0 \subsetneq F_1 \subsetneq \cdots \subsetneq F_m} \prod_{i=1}^m \frac{x ( 1 - x^{\rk(F_i)-\rk(F_{i-1})-1})}{1-x}.\]
    Here the sum is taken over all the nonempty chains of flats starting at the empty set, i.e. $\varnothing = F_0\subsetneq \cdots \subsetneq F_m $ in $\mathcal{L}(\M)$ for every $0\leq m\leq k-1$.\footnote{The chain consisting of only the empty flat yields $m=0$, and the corresponding summand is an empty product, which by convention will be considered as $1$.}
\end{proposition}

For the Chow ring of a matroid, the building set under consideration is the maximal one. In particular, this case is worked out by Feichtner and Yuzvinsky \cite[p.~526]{feichtner-yuzvinsky}.\footnote{We point out that there is a misprint in \cite[p.~526]{feichtner-yuzvinsky}. Namely, in the displayed formula the numerator of one of the fractions is $t(1-t)^{r_i-r_{i-1}-1}$, and it should be $t(1-t^{r_i-r_{i-1}-1})$.} In fact, the way it is stated, this formula holds true for graded atomic lattices with an extra property: the metric $d(A,A')$ that Feichtner and Yuzvinsky introduce in their Definition~4 must coincide with the difference of the ranks of $A$ and $A'$. This property holds for geometric lattices. For an alternative construction of a Gr\"obner basis in the case of matroids which yields a short proof of the above proposition, see alternatively a paper by Backman, Eur, and Simpson \cite[Corollary 3.3.3]{backman-eur-simpson}. Moreover, in \cite[Theorem 3.3.8]{backman-eur-simpson} they provide an appealing interpretation for the coefficients of $\uH_{\M}(x)$ in terms of relative nested quotients.

The next result provides a counterpart result for the augmented Chow ring.

\begin{proposition}\label{prop:hilbert-augmented-chow}
    Let $\M$ be a loopless matroid. The Hilbert--Poincar\'e series of the augmented Chow ring $\CH(\M)$ is given by
        \[ \H_{\M}(x) = 1+\sum_{F_0 \subsetneq F_1 \subsetneq \cdots \subsetneq F_m} \frac{x(1-x^{\rk(F_0)})}{1-x} \prod_{i=1}^m \frac{x(1- x^{\rk(F_i)-\rk(F_{i-1})-1})}{1-x}.\]
    Here the sum is taken over all the nonempty chains of flats, i.e. $F_0\subsetneq \cdots \subsetneq F_m$ in $\mathcal{L}(\M)$ for every $0\leq m\leq k-1$.
\end{proposition}

\begin{proof}
    As mentioned before, one can construct the augmented Chow ring of $\mathcal{L}(\M)$ by considering the lattice of flats of the free coextension of $\M$ and taking a suitable building set on it; this allows one to use the Gr\"obner basis of Feichtner and Yuzvinsky. This computation was carried out by Mastroeni and McCullough in \cite[Section~5.1]{mastroeni-mccullough}. In particular, the basis they construct in \cite[Corollary~5.4]{mastroeni-mccullough} immediately yields our claimed formula. Alternatively, see \cite[Lemma~7.8]{stellahedral} or \cite[Corollary~3.12]{liao}.
\end{proof}

\begin{comment}
    
    A proof of this result can be produced by adapting several steps in the original proof of Proposition~\ref{prop:hilbert-chow}. We require the following preliminary result. 
    
    \begin{lemma}[Corollary 1 in \cite{feichtner-yuzvinsky}]
        Let $\mathcal{L}$ be an atomic lattice and let
        $\mathcal{G}$ be a building set in $\mathcal{L}$. A monomial $\Z$-basis of the Chow ring
        $D(\mathcal{L}, \mathcal{G})$ of $\mathcal{G}$
        is given by
        \begin{equation*}
            \left\{\prod_{A \in \mathcal S} x_A^{m_A} \mathrel{\Big|} \mathcal S \in \mathcal N(\mathcal{L}, \mathcal{G}), m_A < d\left(\bigvee \mathcal S \cap \mathcal{L}_{< A}, A\right)\right\}.
        \end{equation*}
    \end{lemma}

    \begin{proof}[Proof of \Cref{prop:hilbert-augmented-chow}]
        \todo[inline]{LF: Perhaps Matthew can put his proof here. I get confused with the details, but the statement should be completely correct as stated. Notice that we do not require to ``add a coloop'' to our matroid}
        \todo[inline]{working on this in the "goal 3 scratch work" document. - Matthew}
        \todo[inline]{update: it's basically done, but I should make some notational and organizational improvements before merging it into the main document.}
        Combine \Cref{augFYpresentation} and \cite[Corollary 1]{feichtner-yuzvinsky}.
    \end{proof}
\end{comment}

We point out that although the formulas of Propositions \ref{prop:hilbert-chow} and \ref{prop:hilbert-augmented-chow} can be used to compute the Hilbert series of (augmented) Chow rings of small matroids, a drawback that they have is that they require iterating over all the chains of flats in the matroid. The total number of chains of flats of a matroid on $n$ elements can be as large as $\frac{2n!}{\log_2(n)}$ (see sequence \hyperlink{https://oeis.org/A000670}{A000670} in the OEIS \cite{oeis}), so this approach is also considerably slow even for relatively small values of $n$. 

It is of interest to produce alternative or more conceptual ways of computing both  $\uH_{\M}(x)$ and $\H_{\M}(x)$. A remarkable connection between these two invariants that can be used to compute one in terms of the other is given by the following result, which is also the first step towards proving Theorem~\ref{thm:main-recursion-defi-H-and-uH}.

\begin{proposition}\label{prop:recursion-H-in-terms-uH}
    Let $\M$ be a loopless matroid. Then
    \[ \H_{\M}(x) = \sum_{F\in \mathcal{L}(\M)} x^{\rk(F)}\, \uH_{\M/F}(x).\]
\end{proposition}

\begin{proof}
    From the formula of Proposition~\ref{prop:hilbert-chow}, by considering each flat $F\neq \varnothing$ as the term $F_1$ in the chain, we see that
    \begin{align}
        \uH_{\M}(x) &= 1 +  \sum_{\substack{F\in\mathcal{L}(\M)\\F\neq\varnothing}} \frac{x(1-x^{\rk(F)-1})}{1-x} \sum_{F=F_1\subsetneq F_2\subsetneq\cdots\subsetneq F_m} \prod_{i=2}^m \frac{x(1-x^{\rk(F_i)-\rk(F_{i-1})-1})}{1-x} \nonumber\\
        &=1 + \sum_{\substack{F\in\mathcal{L}(\M)\\F\neq\varnothing}} \frac{x(1-x^{\rk(F)-1})}{1-x}\cdot \uH_{\M/F}(x). \label{eq:rec-chow}
    \end{align}
    The summand equal to $1$ comes from considering the chain that consists of only the flat $F_0=\varnothing$ separately. The last equation follows from the fact that the lattice of flats of $\M/F$ is isomorphic to the interval $[F,E]$ in $\mathcal{L}(\M)$. Analogously, by applying the same argument to the formula of Proposition~\ref{prop:hilbert-augmented-chow}, we can fix the flat $F=F_0$ of the chain and write
    \begin{align}
        \H_{\M}(x) &= 1 + \sum_{F\in\mathcal{L}(\M)} \frac{x(1-x^{\rk(F)})}{1-x} \sum_{F\subsetneq F_1\subsetneq F_2\subsetneq\cdots\subsetneq F_m} \prod_{i=1}^m \frac{x(1-x^{\rk(F_i)-\rk(F_{i-1})-1})}{1-x} \nonumber\\
        &=1+\sum_{F\in\mathcal{L}(\M)} \frac{x(1-x^{\rk(F)})}{1-x}\cdot \uH_{\M/F}(x)\nonumber\\
        &= 1+\sum_{\substack{F\in\mathcal{L}(\M)\\F\neq\varnothing}} \frac{x(1-x^{\rk(F)})}{1-x}\cdot \uH_{\M/F}(x). \label{eq:rec-aug-chow}
    \end{align}
    For each integer $r\geq 1$, we have that $\frac{x(1-x^r)}{1-x} = \frac{x(1-x^{r-1})}{1-x} + x^r$. In particular, by combining equations \eqref{eq:rec-chow} and \eqref{eq:rec-aug-chow}, we obtain
    \begin{align*}
        \H_{\M}(x)
        &= 1 + \sum_{\substack{F\in\mathcal{L}(\M)\\F\neq\varnothing}} \frac{x(1-x^{\rk(F)-1})}{1-x}\cdot \uH_{\M/F}(x) + \sum_{\substack{F\in\mathcal{L}(\M)\\F\neq\varnothing}} x^{\rk(F)}\, \uH_{\M/F}(x) \\
        &= \uH_{\M}(x) +  \sum_{\substack{F\in\mathcal{L}(\M)\\F\neq\varnothing}} x^{\rk(F)}\, \uH_{\M/F}(x) \\
        &= \sum_{F\in \mathcal{L}(\M)} x^{\rk(F)}\, \uH_{\M/F}(x).
    \end{align*}
    This proves the desired recursion. 
\end{proof}

\begin{remark}\label{remark:alternative-convolution}
    The careful reader may notice that since $\uH_{\M}(x)$ and $\H_{\M}(x)$ are palindromic and their degrees are $\rk(\M)-1$ and $\rk(\M)$ respectively, it follows from the above recursion that 
    \[ \H_{\M}(x) = x^{\rk(\M)} \H_{\M}(x^{-1}) = x^{\rk(\M)} \sum_{F\in\mathcal{L}(\M)} x^{-\rk(F)} \, \uH_{\M/F}(x^{-1}) = 1 + x \sum_{\substack{F\in\mathcal{L}(\M)\\F\neq E}} \uH_{\M/F}(x).\]
    Clearly, this new recursion also implies the previous one. 
\end{remark}

\begin{remark}
   It is possible to deduce Proposition~\ref{prop:recursion-H-in-terms-uH} in an alternative way. In the work of Braden, Huh, Matherne, Proudfoot, and Wang, \cite[Proposition~1.8]{braden-huh-matherne-proudfoot-wang} states an isomorphism of modules that, when restricting to graded dimensions, yields automatically that
    \[ Z_{\M}(x) = \sum_{F\in\mathcal{L}(\M)} x^{\rk(F)} P_{\M/F}(x).\]
   The proof they provided requires only results about filtrations that are valid for arbitrary pure modules. One may apply a similar reasoning to the Chow ring and the augmented Chow ring. By leveraging \cite[Proposition~2.15 and Lemma~5.7]{braden-huh-matherne-proudfoot-wang}, one deduces that the \emph{stalk} of $\CH(\M)$ at the flat $F$ is isomorphic to $\uCH(\M/F)$. In particular, if $\H(\M)$ denotes the graded M\"obius algebra (defined below, in Section \ref{subsec:ih}), one sees that there is an isomorphism of $\H(\M)$-modules:
    \[ \CH(\M) \cong \bigoplus_{F\in\mathcal{L}(\M)} \uCH(\M/F)[-\rk(F)].\]
    We refer the reader also to Remark~\ref{remark:stalk-empty} below.
\end{remark}

\subsection{An intrinsic definition}\label{subsec:intrinsic-definition}

The goal of this subsection is to give the  proof of Theorem~\ref{thm:main-recursion-defi-H-and-uH}. As was mentioned in the introduction, this theorem highlights the similarities between the polynomials we are studying in this section and the Kazhdan--Lusztig polynomials of matroids. For technical reasons and clarity of the proof, we put fewer requirements in the following statement than in the statement in the introduction.

\begin{theorem}\label{thm:intrinsic-main-body}
    There is a unique way of assigning to each loopless matroid $\M$ a polynomial $\widehat{\uH}_{\M}(x) \in \Z[x]$ satisfying the following properties:
    \begin{enumerate}[\normalfont(i)]
        \item If $\rk(\M) = 0$, then $\widehat{\uH}_\M(x) = 1$. 
        \item If $\rk(\M) > 0$, then $\deg \widehat{\uH}_{\M}(x) < \rk(\M)$ and $x^{\rk(\M)-1}\, \widehat{\uH}_{\M}(x^{-1})=\widehat{\uH}_{\M}(x)$.
        \item For every loopless matroid $\M$, the polynomial
            \[ \widehat{\H}_{\M}(x) := \sum_{F\in\mathcal{L}(\M)} x^{\rk(F)}\, \widehat{\uH}_{\M/F}(x)\]
        is palindromic.
    \end{enumerate} 
    In particular, for every matroid $\M$ we have $\widehat{\uH}_{\M}(x) = \uH_{\M}(x)$ and $\widehat{\H}_{\M}(x) = \H_{\M}(x)$.
\end{theorem}

As it turns out, although the second condition only requires that $\deg \widehat{\uH}_{\M}(x) < \rk(\M)$, the last part of the statement will end up guaranteeing that in fact $\deg \widehat{\uH}_{\M}(x) = \deg \uH_{\M}(x) = \rk(\M) - 1$, and therefore the condition $x^{\rk(\M)-1}\widehat{\uH}_{\M}(x^{-1})=\widehat{\uH}_{\M}(x)$ states that $\widehat{\uH}(x)$ is palindromic.  

The statement of Theorem~\ref{thm:main-recursion-defi-H-and-uH} follows. In that formulation we posed stronger restrictions than in Theorem~\ref{thm:intrinsic-main-body} and, as we know from the conclusion of this, the polynomials resulting from the statement of Theorem~\ref{thm:intrinsic-main-body} satisfy all the conditions of Theorem~\ref{thm:main-recursion-defi-H-and-uH}.

As for the proof, we rely on an elementary symmetric decomposition that is of particular interest in Ehrhart theory due to the work of Stapledon (see also \cite{branden-solus} and \cite{athanasiadis-tzanaki} for related work on real-rootedness of these decompositions). Precisely, we need a slight modification of \cite[Lemma~2.3]{stapledon}.

\begin{lemma}\label{lemma:stapledon2}
    Let $p(x)$ be a polynomial of degree $d$. There exist unique polynomials $a(x)$ of degree $d$ and $b(x)$ of degree at most $d-1$ with the properties that $a(x) = x^d a(x^{-1})$ and $b(x)=x^{d-1} b(x^{-1})$, and that satisfy
        \[ p(x) = a(x) + b(x).\]
\end{lemma}

\begin{proof}
    Let us denote by $p_i$, $a_i$, and $b_i$ the coefficients of $x^i$ in each of $p(x)$, $a(x)$, and $b(x)$. The condition that $\deg a(x) = \deg p(x)$ and $\deg b(x) < \deg p(x)$ implies that $a_d = p_d$ and the condition that $x^d a(x^{-1}) = a(x)$ yields that $a_0 = p_d$ as well. This together with $a(x)+b(x)$ determines $b_0 = p_0 - p_d$, and this in turn determines $b_{d-1}=b_0=p_0-p_d$. Continuing this way, we determine all the coefficients of $a(x)$ and $b(x)$ inductively. Indeed, for each $i$ the coefficients are determined by the equations
    \begin{align*}
        a_i &= p_d + \cdots + p_{d-i} - p_0 - \cdots - p_{i-1},\\
        b_i &= p_0 + \cdots + p_i - p_d - \cdots - p_{d-i}.\qedhere
    \end{align*}
\end{proof}

\begin{proof}[Proof of Theorem~\ref{thm:intrinsic-main-body} and thus of Theorems \ref{thm:main-recursion-defi-H-and-uH} and \ref{thm:correcthpseries}]
    Let us prove the first part of the statement by induction on the size of the ground set of $\M$. We need to establish the uniqueness of $\widehat{\uH}_{\M}(x)$, as $\widehat{\H}_{\M}(x)$ is determined by the former. If $\M$ has cardinality $n=0$, then $\rk(\M) = 0$ and the polynomial $\widehat{\uH}_{\M}(x)$ is uniquely defined and equal to $1$ by the first property. Now, assume the uniqueness has  already been established for matroids with cardinality at most $n-1$, and consider a matroid $\M$ of cardinality $n$. The polynomial
        \[ \mathrm{S}_\M(x) := \sum_{\substack{F\in\mathcal{L}(\M)\\F\neq\varnothing}} x^{\rk(F)}\,\widehat{\uH}_{\M/F}(x),\]
    is determined because all the matroids $\M/F$ for flats $F\neq\varnothing$ have ground sets with cardinality at most $n-1$. Observe that since $F=E$ (the ground set of $\M$) is a nonempty flat of $\M$, in $\mathrm{S}_{\M}(x)$ we have a summand of degree $\rk(\M)$, whereas for $F\subsetneq E$, condition (ii) guarantees that $\deg\left(x^{\rk(F)}\, \widehat{\uH}_{\M/F}(x)\right) \leq \rk(F) + \rk(\M/F) - 1 = \rk(\M) - 1$. In particular $\deg \mathrm{S}_{\M}(x) = \rk(\M)$. 
    
    Now, using Lemma~\ref{lemma:stapledon2}, we can find unique polynomials $a(x)$ and $b(x)$ such that $\deg a(x) = \rk(\M)$, $\deg b(x) \leq \rk(\M) - 1$, the polynomial $a(x)$ is palindromic, the polynomial $b(x)$ satisfies $b(x) = x^{\rk(\M)-1}b(x^{-1})$, and the following property holds:
        \[ \mathrm{S}_{\M}(x) = a(x) + b(x).\]
    In particular, by defining $\widehat{\uH}_{\M}(x) = -b(x)$, which satisfies the requirements of (ii), we obtain that
        \[ \sum_{F\in \mathcal{L}(\M)} x^{\rk(F)} \, \widehat{\uH}_{\M/F}(x) = \widehat{\uH}_{\M}(x) + \mathrm{S}_{\M}(x) = -b(x) + \mathrm{S}_{\M}(x) =  a(x),\]
    which is palindromic, as required. Notice that the uniqueness of the decomposition of Lemma~\ref{lemma:stapledon2} yields the uniqueness for $\widehat{\uH}_{\M}(x)$ as we claimed. Now, since Proposition~\ref{prop:recursion-H-in-terms-uH} guarantees that the Hilbert series of the Chow ring satisfies all these conditions, then the last assertion of the statement follows.
\end{proof}

\subsection{More recursions and their geometry}\label{subsec:self-rec-and-their-geometry}

In the matroid Kazhdan--Lusztig framework (cf. Section~\ref{sec:four}) one has an analogous convolution formula linking the $Z$-polynomial of $\M$ with the Kazhdan--Lusztig polynomial of all the contractions of $\M$, i.e.
    \[ Z_{\M}(x) = \sum_{F\in\mathcal{L}(\M)} x^{\rk(F)}\, P_{\M/F}(x).\]

This recursion can be used to intrinsically define $P_{\M}(x)$ and $Z_{\M}(x)$ for all loopless matroids by proving Theorem~\ref{thm:definition-kl-and-zeta}. However, the original definition of the Kazhdan--Lusztig polynomial of a matroid by Elias, Proudfoot, and Wakefield \cite{elias-proudfoot-wakefield} does not make any reference to $Z_{\M}(x)$. They introduce $P_{\M}(x)$ via a convolution of itself with the characteristic polynomial, i.e.
    \[ x^{\rk(\M)}P_{\M}(x^{-1}) = \sum_{F\in\mathcal{L}(\M)} \chi_{\M|_F}(x)\, P_{\M/F}(x).\] 

It is not initially clear whether an analogue of this expression exists for $\uH_{\M}(x)$---the content of the following result is precisely such a formula. After the proof of this statement, we will discuss some of the geometric motivation behind it.

\begin{theorem}\label{thm:convolution-reduced-char-poly}
    Let $\M$ be a loopless matroid. Then, the Hilbert series of the Chow ring of $\M$ satisfies
    \begin{equation}
        \uH_{\M}(x) = \sum_{\substack{F\in\mathcal{L}(\M)\\F\neq\varnothing}} \overline{\chi}_{\M|_F}(x) \, \uH_{\M/F}(x).\label{eq:convolution-reduced-char-poly}
    \end{equation}
    In particular, the above recursion and the initial condition $\uH_{\U_{0,0}}(x)=1$  uniquely define the map associating to each loopless matroid the Hilbert series of its Chow ring. 
\end{theorem}

\begin{proof}
    To establish this result, we rely on Theorem~\ref{thm:main-recursion-defi-H-and-uH}. This will allow us to show that the right-hand side of \eqref{eq:convolution-reduced-char-poly} indeed coincides with the Hilbert series of $\uCH(\M)$. It suffices to prove that the polynomials
    \[ \widetilde{\uH}_{\M}(x) := 
    \begin{cases} 1 & \text{if $\M$ is empty,}\\ \displaystyle\sum_{\substack{F\in\mathcal{L}(\M)\\F\neq\varnothing}} \overline{\chi}_{\M|_F}(x) \, \uH_{\M/F}(x) & \text{if $\M$ is nonempty}
    \end{cases}\]
    satisfy all the properties of that statement. The first two conditions are immediate to check, and for the last, it will suffice to verify that $\widetilde{\uH}_{\M}(x)$ satisfies the recursion of Remark \ref{remark:alternative-convolution}. We have a chain of equalities:
    \begin{align}
        1 + x \sum_{\substack{F\in \mathcal{L}(\M)\\F\neq E}} \widetilde{\uH}_{\M/F}(x) 
        &=1+x\sum_{\substack{F\in \mathcal{L}(\M)\\F\neq E}} \sum_{\substack{G\in \mathcal{L}(\M/F)\\G\neq \varnothing}}  \overline{\chi}_{(\M/F)|_G}(x)\, \uH_{(\M/F)/G}(x) \nonumber\\
        &=1+x\sum_{\substack{F\in \mathcal{L}(\M)\\F\neq E}} \sum_{\substack{G\in \mathcal{L}(\M)\\G\supsetneq F}} \overline{\chi}_{(\M/F)|_{G\smallsetminus F}}(x)\, \uH_{\M/G}(x)\nonumber\\
        &= 1 +  x \sum_{\substack{G\in \mathcal{L}(\M)\\G\neq \varnothing}} \uH_{\M/G}(x) \sum_{\substack{F\in \mathcal{L}(\M)\\F\subsetneq G}} \overline{\chi}_{(\M/F)|_{G\smallsetminus F}}(x)\nonumber\\
        &=1 +  x \sum_{\substack{G\in \mathcal{L}(\M)\\G\neq \varnothing}} \uH_{\M/G}(x) \sum_{\substack{F\in \mathcal{L}(\M|_G)\\F\subsetneq G}} \overline{\chi}_{(\M|_G)/F}(x)\nonumber\\
        &= 1 +  x \sum_{\substack{G\in \mathcal{L}(\M)\\G\neq \varnothing}} \uH_{\M/G}(x) \cdot \frac{1-x^{\rk(\M|_G)}}{1-x} \label{eq:sum-char-poly}\\
        &= \H_{\M}(x),\nonumber
    \end{align}
    where in \eqref{eq:sum-char-poly} we used Lemma~\ref{lemma:sum-reduced-char-poly-contractions} and in the last step we used a formula we had obtained in equation~\eqref{eq:rec-aug-chow}.
\end{proof}

The preceding proof, which also shows the validity of Theorem~\ref{thm:intro-main0}, resembles the proof of \cite[Proposition~2.3]{proudfoot-xu-young} and relies on a very similar lemma. 

In analogy to the Kazhdan--Lusztig setting, Theorem~\ref{thm:convolution-reduced-char-poly} is not only a different computational shortcut to compute Hilbert series of Chow rings of matroids, but also a shadow of an algebro-geometric fact that we describe now. 

\begin{remark}
    Suppose a matroid $\M$ is realized by a central, essential hyperplane arrangement $\mathcal{A} = \mathcal{A}(\M) = \{H_i \}_{i \in E} \subseteq V$ over $\mathbb{C}$. In this setting, the polynomial $\uH_\M(x)$ is equal to the Hilbert series of the Chow ring of the wonderful variety $\underline{X}_\mathcal{A}$.
    If we denote by $\mathcal{M}(\mathcal{A}) = V \smallsetminus \mathcal{A}$ the complement of the hyperplane arrangement, it is known that $\chi_\M(x)$ (after taking the absolute value of the coefficients) is equal to the Hilbert series of $\mathcal{M}(\mathcal{A})$ and, as a consequence, $\overline{\chi}_\M(x)$ coincides with the Hilbert series of the projectivization $\mathbb{P}(\mathcal{M}(\mathcal{A}))$. 
    Moreover, following the notation in \cite[Section~1.5]{semismall}, for any nonempty proper subset $S \subseteq E$ we define
    \[
    H_S = \bigcap_{i \in S}H_i  \quad \text{and} \quad H_S^\circ = H_S \smallsetminus \bigcup_{S \subsetneq T} H_T.
    \]
    Then $H_S^\circ$ is nonempty if and only if $S$ is a nonempty proper flat of $\M$. This gives us a stratification of $\bigcup_{i \in E} H_i = \bigsqcup_{F}H_F^\circ \subseteq \mathbb{P}(V)$. We recall that the wonderful variety $\underline{X}_\mathcal{A}$ can be realized by blowing up the points in $\mathbb{P}(V)$ corresponding to corank $1$ flats, then blowing up the strict transforms corresponding to corank $2$ flats and so on up until (and including) the rank $1$ flats. This induces a stratification of $\underline{X}_\mathcal{A}$, with the strata labeled by nonempty flats, in the following way: the stratum corresponding to the flat $F$ is given by removing from the exceptional divisor corresponding to $F$ the strict transforms of the linear spaces corresponding to flats smaller than $F$.  
    Now define, as usual, the following two arrangements for any flat $F$:
    \begin{itemize} 
        \item $\mathcal{A}^F = \{H \in \mathcal{A} : H \supseteq F\}$ is an arrangement in $V/F$ called the localization at $F$. We know that $\mathcal{A}(\M|_F) = \mathcal{A}^F$.
        \item $\mathcal{A}_F = \{H \cap F : H \not\supseteq F\} $ is an arrangement in $F$ called the restriction to $F$. We know that $\mathcal{A}(\M/F) = \mathcal{A}_F$.
    \end{itemize}
    It can be shown that the stratum of $\underline{X}_{\mathcal{A}}$ associated to the flat $F$ is isomorphic to $\mathbb{P}(\mathcal{M}(\mathcal{A}^F)) \times \underline{X}_{\mathcal{A}_F}$.
    The formula for the global cohomology of $\underline{X}_\mathcal{A}$ (i.e.  equation~\eqref{eq:convolution-reduced-char-poly} in the realizable case) follows by considering the map on the Grothendieck ring of varieties over $\mathbb{C}$ that takes a smooth projective variety to the Hilbert series of its cohomology.
\end{remark}

At this point it is also of interest to produce a recursion for $\H_{\M}(x)$ in terms of Hilbert series of augmented Chow rings of contractions $\H_{\M/F}(x)$ for $F\in\mathcal{L}(\M)$, but with specific care to not making any references to $\uH_{\M}(x)$. Following the parallelism with the Kazhdan--Lusztig framework, we underline the fact that such type of recursion is certainly complicated for the $Z$-polynomial.\footnote{
    Working carefully in the incidence algebra and using the notion of \emph{inverse Kazhdan--Lusztig polynomial} of a matroid $\M$ (usually denoted $Q_{\M}(x)$) as introduced by Gao and Xie in \cite{gao-xie}, one has the following expression:
    \[ Z_{\M}(x) = \sum_{\substack{F\in\mathcal{L}(\M)\\F\neq\varnothing}} Y_{\M|_F}(x)\, Z_{\M/F}(x),\]
    where for each loopless matroid $\M$ on a ground set $E$ one defines
    \[ Y_{\M}(x) := -\sum_{F\in\mathcal{L}(\M)} (-1)^{\rk(F)} Q_{\M|_F}(x) \mu(F,E) x^{\rk(\M)-\rk(F)}.\]
}

Fortunately, the Hilbert series of the augmented Chow ring is more tractable in this respect.

\begin{theorem}\label{thm:convolution-mobius}
     Let $\M$ be a loopless matroid. Then the Hilbert series of the augmented Chow ring of $\M$ satisfies  
    \begin{equation} 
        \H_{\M}(x) = - \sum_{\substack{F\in\mathcal{L}(\M)\\F\neq\varnothing}} \mu(\varnothing,F) (1+x+\cdots+ x^{\rk(F)}) \, \H_{\M/F}(x).\label{eq:convolution-mobius}
    \end{equation}
    In particular, the above recursion and the initial condition $\H_{\U_{0,0}}(x)=1$  uniquely define the map associating to each loopless matroid the Hilbert series of its augmented Chow ring. 
\end{theorem}

\begin{proof}
    Again, we will rely on Theorem~\ref{thm:main-recursion-defi-H-and-uH}, but this time we will also need Theorem~\ref{thm:convolution-reduced-char-poly}. Starting from the right-hand side of \eqref{eq:convolution-mobius} (without the minus sign) and using the recursion of Theorem~\ref{thm:main-recursion-defi-H-and-uH}, we have
    \begin{align*}
        &\sum_{\substack{F\in\mathcal{L}(\M)\\F\neq\varnothing}} \mu(\varnothing,F) \frac{1-x^{\rk(F)+1}}{1-x}\, \H_{\M/F}(x)\\ &= \sum_{\substack{F\in\mathcal{L}(\M)\\F\neq\varnothing}} \mu(\varnothing,F) \frac{1-x^{\rk(F)+1}}{1-x}\sum_{\substack{G\in\mathcal{L}(\M)\\G\supseteq F}}x^{\rk(G) - \rk(F)} \uH_{\M/G}(x)\\
        &=\sum_{\substack{G\in\mathcal{L}(\M)\\G\neq\varnothing}} \uH_{\M/G}(x)  \sum_{\substack{F\in\mathcal{L}(\M)\\\varnothing\neq F\subseteq G}}\mu(\varnothing, F)\, x^{\rk(G) - \rk(F)}\frac{1-x^{\rk(F)+1}}{1-x},
        \intertext{where in the last step we just interchanged the order of summation. Now, breaking the inner sum into a difference of two sums yields}
        &= \sum_{\substack{G\in\mathcal{L}(\M)\\G\neq\varnothing}} \uH_{\M/G}(x) \frac{1}{1-x}\left( \sum_{\substack{F\in\mathcal{L}(\M)\\\varnothing\neq F\subseteq G}}\mu(\varnothing, F)\, x^{\rk(G) - \rk(F)} - \sum_{\substack{F\in\mathcal{L}(\M)\\\varnothing\neq F\subseteq G}} \mu(\varnothing, F) x^{\rk(G)+1}\right)\\
        &= \sum_{\substack{G\in\mathcal{L}(\M)\\G\neq\varnothing}} \uH_{\M/G}(x) \frac{1}{1-x}\left( \left(\chi_{\M|_G}(x) - x^{\rk(G)}\right) + x^{\rk(G)+1}\right)\\
        &= -\sum_{\substack{G\in\mathcal{L}(\M)\\G\neq\varnothing}} \overline{\chi}_{\M|_G}(x)\, \uH_{\M/G}(x) -\sum_{\substack{G\in\mathcal{L}(\M)\\G\neq\varnothing}} x^{\rk(G)} \uH_{\M/G}(x).\\
        \intertext{Using the recursions of Theorem~\ref{thm:convolution-reduced-char-poly} and Theorem~\ref{thm:main-recursion-defi-H-and-uH}, we obtain}
        &= -\uH_{\M}(x) - \left(\H_{\M}(x) - \uH_{\M}(x)\right)\\
        &= -\H_{\M}(x),
    \end{align*}
    and therefore the proof is complete.
\end{proof}

\subsection{Incidence algebras and the case of uniform matroids}\label{subsec:incidence-uniform}

One of the goals we now pursue is determining $\uH_{\M}(x)$ and $\H_{\M}(x)$ whenever $\M$ is an arbitrary uniform matroid. 

We will first produce more general formulas using the incidence algebra of the lattice of flats $\mathcal{L}(\M)$. These new formulas work for arbitrary matroids, but will be of particular use to produce the first concrete expressions for the Hilbert series of both the Chow ring and the augmented Chow ring of arbitrary uniform matroids. In particular, we will use these formulas to settle a conjecture posed by Hameister, Rao, and Simpson \cite[Conjecture 6.2]{hameister-rao-simpson}, and we will provide a counterpart of it in the augmented case. Afterwards, in Theorem~\ref{thm:uniform-formulas-main}, we will deduce alternative expressions for $\uH_{\U_{k,n}}(x)$ and $\H_{\U_{k,n}}(x)$; these alternative formulas arise via Propositions \ref{prop:hilbert-chow} and \ref{prop:hilbert-augmented-chow}, along with some results on local $h$-vectors of simplicial complexes. 

We recall some notions of poset theory and incidence algebras; for a detailed treatment we refer to \cite[Section 3.6]{stanley-ec1}. Let $\mathcal{L}$ be a locally finite poset, and denote by $\operatorname{Int}(\mathcal{L})$ the set of all  closed intervals of $\mathcal{L}$. (Recall that $\varnothing$ is \emph{not} a closed interval.) The \textit{incidence algebra} of $\mathcal{L}$ over a ring $R$, denoted $\mathcal{I}(\mathcal{L})$, is the $R$-module of all maps $f\colon\operatorname{Int}(\mathcal{L})\to R$, where the multiplication (or ``convolution'') is defined via
    \[ (f\cdot g)(s,u) := \sum_{s\leq t\leq u} f(s,t)g(t,u).\]
Recall that $\mathcal{I}(\mathcal{L})$ is associative and possesses a two-sided identity $\delta$ defined by $\delta(s,t) =1$ when $s=t$ and $\delta(s,t)=0$ when $s<t$. A useful fact is that if $f\cdot g=\delta$ in $\mathcal{I}(\mathcal{L})$, then one also has $g \cdot f=\delta$. For a proof 
 of this fact, see \cite[Proposition~3.6.2]{stanley-ec1}.

In what follows we will be concerned with the case in which $\mathcal{L}=\mathcal{L}(\M)$ is the lattice of flats of a loopless matroid $\M$ on a ground set $E$, and $R$ is the field of rational functions on a variable $x$. Recalling that we have defined the characteristic polynomial of a (loopless) matroid of rank $0$ to be $-1$, the identity of Theorem~\ref{thm:convolution-reduced-char-poly} can be rephrased as follows: 
    \[ -\delta_{\M} := -\delta_{[\varnothing,E]}= \sum_{F\in\mathcal{L}(\M)} \overline{\chi}_{\M|_F}(x)\, \uH_{\M/F}(x) = \sum_{F\in\mathcal{L}(\M)} \overline{\chi}_{[\varnothing, F]} \, \uH_{[F,E]} = (\overline{\chi}\cdot \uH)_{[\varnothing,E]},\]
where $\overline{\chi}_{[\varnothing,F]}$ stands for the evaluation of $\overline{\chi}$ at the interval $[\varnothing, F]$, and analogously for $\uH_{[F,E]}$. In other words, $-\overline{\chi}$ and $\uH$ are inverses in $\mathcal{I}(\mathcal{L}(\M))$ and hence one has
    \[ -\delta_{\M} = (\uH\cdot \overline{\chi})_{[\varnothing, E]} = \sum_{F\in \mathcal{L}(\M)} \uH_{[\varnothing, F]}\, \overline{\chi}_{[F,E]} = \sum_{F\in\mathcal{L}(\M)} \uH_{\M|_F}(x)\, \overline{\chi}_{\M/F}(x).\]

This proves the following alternative convolution formula.

\begin{proposition}\label{prop:hilbert-chow-inverse-incidence}
    Let $\M$ be a loopless matroid on $E$. The following formula holds:
    \[ \uH_{\M}(x) = \sum_{\substack{F\in\mathcal{L}(\M)\\F\neq E}} \uH_{\M|_F}(x)\, \overline{\chi}_{\M/F}(x).\]
\end{proposition}

This result was also achieved using an  algebro-geometric perspective by Jensen, Kutler, and Usatine, and it was instrumental in proving one of their main results: a recursion that the motivic zeta functions of matroids satisfy \cite[Theorem~1.8]{jensen-kutler-usatine}.  Yet another way of deducing Proposition~\ref{prop:hilbert-chow-inverse-incidence}, via tropical geometry, is by using the main result of Amini and Piquerez \cite[Theorem~1.1]{amini-piquerez}; we omit the details here. 

%We want to mention that this result is instrumental to retrieve one of the main theorems of Jensen, Kutler, and Usatine \cite[Theorem~1.8]{jensen-kutler-usatine} (see Section 4 in that paper for details about this assertion); this is used to prove the validity of the recursion that the motivic zeta functions of matroids satisfy. An alternative way of deducing Proposition \ref{prop:hilbert-aug-chow-inverse-incidence}, via tropical geometry, is by using the main result of Amini and Piquerez \cite[Theorem~1.1]{amini-piquerez}; we omit the details here. 

This proposition provides a recursion that is particularly useful to compute $\uH_{\M}(x)$ whenever $\M$ is an arbitrary uniform matroid. The reason for this is that the restriction $\M|_F$ for $F\in\mathcal{L}(\M)$ is always a Boolean matroid whenever $\M$ is uniform and $F$ is a proper flat. One of the motivations for this idea came from the Kazhdan--Lusztig theory of matroids, where the inverse Kazhdan--Lusztig polynomial plays a role to compute the Kazhdan--Lusztig polynomial of arbitrary uniform matroids in \cite[Section 3]{gao-xie}.

\begin{corollary}\label{corollary:uH-uniform}
    The Hilbert series of the Chow ring of a uniform matroid of rank $k$ and cardinality $n$ is given by
    \[ \uH_{\U_{k,n}}(x) = \sum_{j=0}^{k-1} \binom{n}{j} A_{j}(x)\, \overline{\chi}_{\U_{k-j,n-j}}(x).\]
\end{corollary}

The preceding formula can be made explicit because the reduced characteristic polynomial of a uniform matroid is not difficult to compute (see Lemma~\ref{reduced-char-poly} below). Moreover, it essentially resolves a conjecture posed by Hameister, Rao, and Simpson in \cite{hameister-rao-simpson} regarding the face enumeration of the Bergman complex of a matroid (i.e. the order complex of the proper part of the lattice of flats). We reformulate their conjecture here.

\begin{corollary}[{\cite[Conjecture 6.2]{hameister-rao-simpson}}]
    Let us denote by $h_{\Delta(\widehat{\mathcal{L}}(\U_{k,n}))}(x)$ the $h$-polynomial of the Bergman complex of $\U_{k,n}$. Then
    \[ h_{\Delta(\widehat{\mathcal{L}}(\U_{k,n}))}(x) = \sum_{i=1}^k \binom{n-i-1}{k-i}\, \uH_{\U_{i,n}}(x).\]
\end{corollary}

Since the proof involves quite a bit of calculations and  background, we reserve it until Appendix~\ref{appendix}.

Regarding the augmented case, it is evident that we can use the same reasoning as above, together with Theorem~\ref{thm:convolution-mobius} to produce a formula similar to that of Proposition~\ref{prop:hilbert-chow-inverse-incidence} for the Hilbert series of the augmented Chow ring. 

\begin{proposition}\label{prop:hilbert-aug-chow-inverse-incidence}
    Let $\M$ be a loopless matroid on $E$. The following formula holds:
    \[ \H_{\M}(x) = - \sum_{\substack{F\in\mathcal{L}(\M)\\F\neq E}} \H_{\M|_F}(x)\, \mu(F,E) \left(1 + x + \cdots + x^{\rk(\M)-\rk(F)}\right).\]
\end{proposition}

The above proposition yields a formula for the Hilbert series of the augmented Chow ring of arbitrary uniform matroids, via expressing them in terms of binomial Eulerian polynomials. 

\begin{corollary}\label{corollary:H-uniform}
    The Hilbert series of the augmented Chow ring of a uniform matroid of rank $k$ and cardinality $n$ is given by
    \[ \H_{\U_{k,n}}(x) = \sum_{j=0}^{k-1} (-1)^{k-1-j} \binom{n}{j} \binom{n-1-j}{k-1-j} \widetilde{A}_{j}(x)\, (1+x+\cdots+x^{k-j}).\]
\end{corollary}

Although Corollary~\ref{corollary:uH-uniform} and Corollary~\ref{corollary:H-uniform} are explicit expressions, they are both alternating sums, due to basic properties of the M\"obius function and the reduced characteristic polynomial. In the remainder of this subsection we will prove Theorem~\ref{thm:uniform-formulas-main}. That result provides the cleanest way we are aware of for computing the Hilbert series of Chow rings and augmented Chow rings of uniform matroids. The proof will be carried out by leveraging the following lemma, whose proof relies on (and later will extend) a recursion found by Juhnke-Kubitzke, Murai, and Sieg \cite{juhnke-murai-sieg} for derangement polynomials.

\begin{lemma}\label{lem:dn-and-An}
    Let $\M=\U_{n,n}$ be a Boolean matroid on a ground set $E$ of $n\geq 1$ elements. Then, by considering only chains of flats that end at the top element $E$ of $\mathcal{L}(\M)$, we have 
    \begin{align*}
         d_n(x) &= \sum_{\varnothing = F_0 \subsetneq \cdots \subsetneq F_m= E} \prod_{i=1}^m \frac{x ( 1 - x^{\rk(F_i)-\rk(F_{i-1})-1})}{1-x},\\
         xA_n(x) &= \sum_{F_0 \subsetneq \cdots \subsetneq F_m=E} \frac{x(1-x^{\rk(F_0)})}{1-x} \prod_{i=1}^m \frac{x(1- x^{\rk(F_i)-\rk(F_{i-1})-1})}{1-x}.
    \end{align*}
\end{lemma}

\begin{proof}
    Let us prove the first identity. For each $n\geq 1$, denote by $\mathfrak{d}_n(x)$ the sum on the right-hand side. By considering what the penultimate element of the chain ending in $E$ is, we see that the sequence $\mathfrak{d}_n(x)$ satisfies the recurrence
        \[ \mathfrak{d}_n(x) = \sum_{j=0}^{n-2} \binom{n}{j} \mathfrak{d}_j(x) (x+x^2+\cdots+x^{n-j-1}).\]
    In \cite[Corollary 4.2]{juhnke-murai-sieg}, it is proved that this recursion determines the derangement polynomials, and hence we have $\mathfrak{d}_n(x)=d_n(x)$ for each $n\geq 1$, as claimed.
    
    To prove the second identity we rely on the first. Call the right-hand side $\mathfrak{a}_n(x)$. Choosing the set $F_0$ we obtain
    \begin{align*}
        \mathfrak{a}_n(x) &= \sum_{j=1}^n \binom{n}{j} (x+x^2+\cdots+x^j) d_{n-j}(x)\\
        &= \sum_{i=0}^{n-1} \binom{n}{i} d_i(x) (x+x^2+\cdots+x^{n-i}) = xA_n(x),
    \end{align*}
    where the last equality follows from the locality formula for the $h$-polynomial of the barycentric subdivision of the boundary of the $(n-1)$-simplex, i.e. from combining \cite[Proposition~2.4]{stanley-local} and \cite[Theorem~3.2]{stanley-local}.
\end{proof}

\begin{theorem}\label{thm:uniform-formulas-body}
    The Hilbert series of the Chow ring and augmented Chow ring of arbitrary uniform matroids are given by 
    \begin{align*}
        \uH_{\U_{k,n}}(x) &= \enspace \sum_{j=0}^{k-1}\, \binom{n}{j} \, d_j(x) (1+x+\cdots+x^{k-1-j}),\\
        \H_{\U_{k,n}}(x) &= 1 + x\sum_{j=0}^{k-1} \binom{n}{j}\, A_j(x) (1+x+\cdots+x^{k-1-j}).
    \end{align*}
\end{theorem}

\begin{proof}
    Let us apply the formula of Proposition~\ref{prop:hilbert-chow} to the uniform matroid $\U_{k,n}$. Each chain of flats $\varnothing=F_0\subsetneq\cdots\subsetneq F_m$ appearing in the sum either has $F_m\subsetneq E$ or $F_m=E$. Lemma \ref{lem:dn-and-An} tells us that, for each flat $F\subsetneq E$ of rank $j$, the sum over all chains finishing at $F$ yields the polynomial $d_j(x)$. On the other hand,  by fixing the flat $F_{m-1}$,  Lemma~\ref{lem:dn-and-An} also allows us to calculate the sum of all the summands for which $F_m=E$. This yields
    \[ \uH_{\U_{k,n}}(x) = \sum_{j=0}^{k-1} \binom{n}{j}\, d_j(x) + \sum_{j=0}^{k-1} \binom{n}{j}\, d_j(x) \frac{x(1-x^{k-j-1})}{1-x}.\]
    After rearranging, the claimed identity is proved. The formula for the augmented Chow ring follows in a completely analogous way. 
\end{proof}

\begin{remark}
   The expression for $\H_{\U_{k,n}}(x)$ derived in the last statement bears an intriguing resemblance to the $h$-polynomials of the class of polytopes known as \emph{partial permutohedra}, studied recently in \cite{partial-permutohedra}. More precisely, compare our formula with their \cite[Theorem~3.17]{partial-permutohedra}.
\end{remark}

\begin{table}[htb]%\[!htb\]
    \caption{Examples of $\uH_{\M}(x)$ and $\H_{\M}(x)$ for some uniform matroids $\M$}
    \begin{subtable}[t]{.5\textwidth}
        \caption{Examples of $\uH_{\U_{k,k+1}}(x)$}
        \raggedright
            \begin{tabular}{l l l l l l l l}\hline
$k=$      & 1   & 2 & 3 & 4 & 5 & 6 & 7  \\ \hline
 $1$      & 1   & 1  & 1  & 1  & 1  & 1  & 1   \\
 $x$      &     & 1  & 7  & 21  & 51  &  113 & 239  \\
 $x^2$    &     &   &  1   & 21  & 161  & 813  & 3361   \\
 $x^3$    &     &   &     & 1  & 51  & 813  & 7631   \\
 $x^4$    &     &   &     &    &  1  & 113  & 3361    \\
 $x^5$    &     &   &   &   &   &  1&  239  \\
 $x^6$    &     &   &   &   &   &   &   1
\end{tabular}
    \end{subtable}%
   \begin{subtable}[t]{.5\textwidth}
        \raggedleft
        \caption{Examples of $\H_{\U_{k,k+1}}(x)$}
        \begin{tabular}{l l l l l l l l l l l l l l l}\hline
$k=$   & 1   & 2 & 3 & 4 & 5 & 6 & 7   \\ \hline
$1$  & 1 & 1 & 1 & 1 & 1 & 1 & 1 \\
$x$  & 1 & 4 & 11 & 26 & 57 & 120 & 247 \\
$x^2$  &  & 1 & 11 & 66 & 302 & 1191 & 4293 \\
$x^3$  &  &  & 1 & 26 & 302 & 2416 & 15619 \\
$x^4$  &  &  &  & 1 & 57 & 1191 & 15619 \\
$x^5$  &  &  &  &  & 1 & 120 & 4293\\
$x^6$  &  &  &  &  &  & 1 & 247 \\
$x^7$  &  &  &  &  &  &  & 1
\end{tabular}
    \end{subtable}
\vspace{2mm}

       \begin{subtable}[t]{.5\textwidth}
        \caption{Examples of $\uH_{\U_{k,k+2}}(x)$}
        \raggedright
\begin{tabular}{l l l l l l l l l l l l l l l}\hline
$k=$      & 1   & 2 & 3 & 4 & 5 & 6 & 7\\ \hline
$1$  & 1 & 1 & 1 & 1 & 1 & 1 & 1 \\
$x$  &  & 1 & 11 & 36 & 92 & 211 & 457 \\
$x^2$  &  &  & 1 & 36 & 337 & 1877 & 8269 \\
$x^3$  &  &  &  & 1 & 92 & 1877 & 20155 \\
$x^4$  &  &  &  &  & 1 & 211 & 8269 \\
$x^5$  &  &  &  &  &  & 1 & 457 \\
$x^6$  &  &  &  &  &  &  & 1 \\
\end{tabular}
    \end{subtable}%
   \begin{subtable}[t]{.5\textwidth}
        \raggedleft
        \caption{Examples of $\H_{\U_{k,k+2}}(x)$}
\begin{tabular}{l l l l l l l l l l l l l l l}\hline
$k=$      & 1   & 2 & 3 & 4 & 5 & 6 & 7   \\ \hline
$1$  & 1 & 1 & 1 & 1 & 1 & 1 & 1 \\
$x$  & 1 & 5 & 16 & 42 & 99 & 219 & 466 \\
$x^2$  &  & 1 & 16 & 117 & 610 & 2641 & 10204 \\
$x^3$  &  &  & 1 & 42 & 610 & 5637 & 40444 \\
$x^4$  &  &  &  & 1 & 99 & 2641 & 40444 \\
$x^5$  &  &  &  &  & 1 & 219 & 10204 \\
$x^6$  &  &  &  &  &  & 1 & 466 \\
$x^7$  &  &  &  &  &  &  & 1 \\
\end{tabular}
    \end{subtable}
\end{table}

\subsection{Semi-small decompositions and \texorpdfstring{$\gamma$}{gamma}-positivity}\label{subsec:gamma-positivity}

Now we turn our attention back to the general case. In \cite{semismall} Braden, Huh, Matherne, Proudfoot, and Wang found a \textit{semi-small decomposition} for both the Chow ring and the augmented Chow ring of arbitrary loopless matroids. These decompositions can be used to prove the K\"ahler package for both of these rings. Before stating this result, let us introduce some useful notation: whenever $\M$ is a loopless matroid and $i$ is \emph{not} a coloop, we define two special families of flats of $\M$:
    \begin{align}
        \underline{\mathscr{S}}_i = \underline{\mathscr{S}}_i(\M) &= \left\{ F \in \mathcal{L}(\M): \varnothing \subsetneq F\subsetneq E\smallsetminus\{i\} \text{ and } F\cup\{i\}\in \mathcal{L}(\M)\right\},\nonumber\\
        \mathscr{S}_i = \mathscr{S}_i(\M) &=  \left\{ F \in \mathcal{L}(\M):  F\subsetneq E\smallsetminus\{i\} \text{ and } F\cup\{i\}\in \mathcal{L}(\M)\right\}.\label{eq:family_S_i}
    \end{align}

\begin{theorem}[{\cite[Theorems 1.2 and 1.5]{semismall}}]\label{thm:semismall-decompositions}
    Let $\M$ be a loopless matroid and let $i\in E$ be an element that is not a coloop. Then, there is an isomorphism of $\uCH(\M\smallsetminus\{i\})$-modules:
    \[ \uCH(\M) \cong \uCH(\M\smallsetminus\{i\}) \oplus \bigoplus_{F\in \underline{\mathscr{S}}_i} \uCH\left(\M/{\left(F\cup\{i\}\right)}\right)\otimes \uCH(\M|_F)[-1]. \]
    Additionally, there is an isomorphism of $\CH(\M\smallsetminus\{i\})$-modules:
    \[ \CH(\M) \cong \CH(\M\smallsetminus\{i\}) \oplus \bigoplus_{F\in \mathscr{S}_i} \uCH\left(\M/{\left(F\cup\{i\}\right)}\right)\otimes \CH(\M|_F)[-1]. \]
\end{theorem}

We stress the fact that in the second of the above two isomorphisms, i.e. in the case of augmented Chow rings, the terms appearing in the direct sum depend on the Chow ring (as opposed to \emph{augmented}) of certain contractions. 

As a consequence of the above isomorphisms, one obtains a different set of recurrences for the Hilbert series of the Chow rings and augmented Chow rings of matroids in terms of single-element deletions and certain restrictions and contractions.

\begin{corollary}\label{coro:recursion-hilbert-chow}
    Let $\M$ be a loopless matroid and let $i\in E$ be an element that is not a coloop. Then, the Hilbert series of the Chow ring satisfies
    \[ \uH_{\M}(x) = \uH_{\M\smallsetminus\{i\}}(x) + x \sum_{F\in \underline{\mathscr{S}}_i} \uH_{\M/{(F\cup\{i\})}}(x) \cdot \uH_{\M|_F}(x).\]
    Additionally, the Hilbert series of the augmented Chow ring satisfies
    \[ \H_{\M}(x) = \H_{\M\smallsetminus\{i\}}(x) + x \sum_{F\in \mathscr{S}_i} \uH_{\M/{(F\cup\{i\})}}(x) \cdot \H_{\M|_F}(x).\]
\end{corollary}

In contrast to the previous formulas we obtained, the ones in the preceding result are \emph{quadratic} recurrences for $\uH_{\M}(x)$ and $\H_{\M}(x)$, i.e. this is the first time we encounter two polynomials $\uH_{\M_1}(x)$ and $\uH_{\M_2}(x)$ multiplying each other. To the best of our knowledge, there is no ``formal'' way of deducing the above identities from the ones in the previous subsections. A particularly fruitful application of these recursions arising from the semi-small decomposition is that they allow us to formulate a proof of the $\gamma$-positivity for $\uH_{\M}(x)$ and $\H_{\M}(x)$ using a simple inductive argument.

%We make a short digression about the above formulas by comparing them with some of our previous recursions. On one hand, they are similar to the formulas we have established so far, in the sense that they allow one to compute the Hilbert series of the Chow rings and augmented Chow rings in terms of various restrictions and contractions. On the other hand, they are of a different nature, as they provide a \emph{quadratic} recurrence for $\uH_{\M}(x)$ and $\H_{\M}(x)$, i.e. this is the first time we encounter two polynomials $\uH_{\M_1}(x)$ and $\uH_{\M_2}(x)$ multiplying each other. To the best of our knowledge, there is no ``formal'' way of deducing the above identities from the ones in the previous subsections. As we continue our study, the validity of these recursions will make apparent to the reader the strength of the semi-small decompositions. A particularly fruitful instance is that this allows us to formulate a proof of the $\gamma$-positivity for $\uH_{\M}(x)$ and $\H_{\M}(x)$ using a simple inductive argument.

\begin{theorem}\label{thm:gamma-positivity-uH-H}
    Let $\M$ be a loopless matroid. The polynomials $\uH_{\M}(x)$ and $\H_{\M}(x)$ are $\gamma$-positive.
\end{theorem}

\begin{proof}
    We proceed by induction on the size of the ground set of $\M$. If the matroid $\M$ has a ground set of cardinality $1$, then $\M\cong \U_{1,1}$. In this case, $\uH_{\M}(x) = 1$ and $\H_{\M}(x)=x+1$. The associated $\gamma$-polynomials are $\gamma(\uH_{\M},x) = 1$, and $\gamma(\H_{\M},x) = 1$, and hence they are $\gamma$-positive. 
    
    Assuming that we have proved the validity of the statement for all matroids with cardinality at most $n-1$, let us consider a matroid $\M$ having cardinality $n$. If $\M$ is a Boolean matroid, i.e. $\M \cong \U_{n,n}$ for some $n\geq 1$, then $\uH_{\M}(x) = A_n(x)$, the Eulerian polynomial, whereas $\H_{\M}(x) = \widetilde{A}_n(x)$, the binomial Eulerian polynomial. As we mentioned before, both of these families of polynomials are known to be real-rooted and hence $\gamma$-positive.
    
    If $\M$ is not Boolean, then there is at least one element $i\in E$ that is not a coloop. Using Corollary~\ref{coro:recursion-hilbert-chow} in combination with Lemma~\ref{lemma:properties-gamma}, we obtain the following two recurrences:
    \begin{align*}
        \gamma(\uH_{\M},x) &= \gamma(\uH_{\M\smallsetminus\{i\}},x) + x \sum_{F\in \underline{\mathscr{S}}_i(\M)} \gamma(\uH_{\M/{(F\cup\{i\})}},x) \cdot \gamma(\uH_{\M|_F},x),\\
        \gamma(\H_{\M},x) &= \gamma(\H_{\M\smallsetminus\{i\}},x) + x \sum_{F\in \mathscr{S}_i(\M)} \gamma(\uH_{\M/{(F\cup\{i\})}},x) \cdot \gamma(\H_{\M|_F},x).
    \end{align*}
    The induction hypothesis guarantees that each of the summands on the right-hand side has nonnegative coefficients. The proof is now complete.
\end{proof}

%\textcolor{dn}{in the proof above and/or in the remark below, we should cite the original sources of the $\gamma$-positivity of the Eulerian and binomial Eulerian polynomials, because we are really trying to sell that those were hard and this just falls out!}

\begin{remark}\label{remark:with-coloop}
    The preceding proof relies on the $\gamma$-positivity of the Eulerian and the binomial Eulerian polynomials. Although these two results are now well-known, since their proofs are not straightforward (see \cite[Section~11]{postnikov-reiner-williams}), a reasonable question that the reader might ask is whether it is possible to circumvent the base cases of Boolean matroids in the induction, or at least give a self-contained proof including this case. The answer is yes; in fact, the second case of the semi-small decompositions of \cite[Theorem~1.2 and Theorem~1.5]{semismall} consider the case in which the element to delete from the matroid is a coloop. By computing the graded dimensions, one obtains formulas for $\uH_{\M\oplus\U_{1,1}}(x)$ and $\H_{\M\oplus\U_{1,1}}(x)$ when deleting the coloop corresponding to the ground set of the direct summand $\U_{1,1}$:
        \begin{align*}
        \uH_{\M\oplus\U_{1,1}}(x) &= (1+x)\,\uH_{\M}(x) + x\sum_{\substack{F\in\mathcal{L}(\M)\\\varnothing\neq F\neq E}} \uH_{\M/F}(x)\, \uH_{\M|_F}(x),\\
        \H_{\M\oplus\U_{1,1}}(x) &= (1+x)\,\H_{\M}(x) + x\sum_{\substack{F\in\mathcal{L}(\M)\\F\neq E}} \uH_{\M/F}(x) \,\H_{\M|_F}(x).
        \end{align*}
    Therefore, by Lemma~\ref{lemma:properties-gamma}, at the level of $\gamma$-polynomials one has
    \begin{align*}
        \gamma(\uH_{\M\oplus\U_{1,1}},x) &= \gamma(\uH_{\M},x) + x\sum_{\substack{F\in\mathcal{L}(\M)\\\varnothing\neq F\neq E}} \gamma(\uH_{\M/F},x)\, \gamma(\uH_{\M|_F},x),\\
        \gamma(\H_{\M\oplus\U_{1,1}},x) &= \gamma(\H_{\M},x) + x\sum_{\substack{F\in\mathcal{L}(\M)\\F\neq E}} \gamma(\uH_{\M/F},x) \,\gamma(\H_{\M|_F},x).
        \end{align*}
    In particular, since all the restrictions and contractions in a Boolean matroid are again Boolean, reasoning inductively one proves that the $\gamma$-polynomial of $\uH_{\M}(x)$ and $\H_{\M}(x)$ have nonnegative coefficients for all Boolean matroids. This gives an independent proof of the $\gamma$-positivity of the families of Eulerian and binomial Eulerian polynomials.
\end{remark}

The $\gamma$-positivity for Hilbert series of Chow rings of matroids was also observed independently by Botong Wang in private communication with the authors. Continuing with our digression about $\gamma$-positivity really being a consequence of Theorem~\ref{thm:semismall-decompositions}, we comment about what may happen if one pretends to extend this property to other posets.

\begin{remark}
    Observe that for an arbitrary finite bounded graded poset $P$, there is no formal obstruction in using Theorem \ref{thm:main-recursion-defi-H-and-uH} as the definition of two polynomial invariants $\uH_{P}(x)$ and $\H_P(x)$; the proof of that statement did not rely on the fact that we were dealing with a geometric lattice.\footnote{Analogously, it is possible to define a Kazhdan--Lusztig and $Z$-polynomial for an arbitrary finite bounded graded poset; see \cite[Example~2.13]{proudfoot-kls}.} However, in such generality the real-rootedness and the $\gamma$-positivity fail. For instance, the poset depicted on the left in Figure \ref{fig:counterexample-generalposet} has
    \begin{align*}
        \uH_P(x) &= x^4+7x^3+11x^2+7x+1,\\ 
        \H_P(x) &= x^5+8x^4+18x^3+18x^2+8x+1.
    \end{align*}
    We observe that $\uH_{P}(x)$ and $\H_P(x)$ are not real-rooted since they are not even $\gamma$-positive. In this case we have $\gamma(\uH_P,x) = \gamma(\H_P,x)=-x^2+3x+1$. 

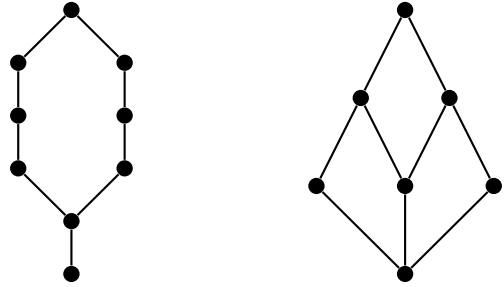
\begin{figure}[ht]
    \centering
	\begin{tikzpicture}  
	[scale=0.7,auto=center,every node/.style={circle,scale=0.8, fill=black, inner sep=2.7pt}] 
	\tikzstyle{edges} = [thick];
	
	\node[] (a1) at (0,0) {};  
	\node[] (a2) at (0,1)  {};  
	\node[] (a3) at (-1,2) {};
	\node[] (a4) at (1,2) {};
	\node[] (a5) at (-1,3)  {};  
	\node[] (a6) at (1,3)  {};  
	\node[] (a7) at (-1,4)  {};  
	\node[] (a8) at (1,4)  {};  
	\node[] (a9) at (0,5) {};
	
	\draw[edges] (a1) -- (a2);  
	\draw[edges] (a2) -- (a3);  
	\draw[edges] (a2) -- (a4);  
	\draw[edges] (a3) -- (a5);
	\draw[edges] (a4) -- (a6);
	\draw[edges] (a5) -- (a7);
	\draw[edges] (a6) -- (a8);
	\draw[edges] (a7) -- (a9);
	\draw[edges] (a8) -- (a9);
	\end{tikzpicture}\qquad\qquad\qquad\begin{tikzpicture}  
	[scale=0.7,auto=center,every node/.style={circle,scale=0.8, fill=black, inner sep=2.7pt}] 
	\tikzstyle{edges} = [thick];
	
	\node[] (a1) at (0,0) {};  
	\node[] (a3) at (0,5/3)  {};  
	\node[] (a2) at (-5/3,5/3) {};
	\node[] (a4) at (5/3,5/3) {};
	\node[] (a5) at (-5/6,10/3)  {};  
	\node[] (a6) at (5/6,10/3)  {};  
	\node[] (a7) at (0,5)  {};
	
	\draw[edges] (a1) -- (a2);  
	\draw[edges] (a1) -- (a3);  
	\draw[edges] (a1) -- (a4);  
	\draw[edges] (a2) -- (a5);
	\draw[edges] (a3) -- (a5);
	\draw[edges] (a3) -- (a6);
	\draw[edges] (a4) -- (a6);
	\draw[edges] (a5) -- (a7);
	\draw[edges] (a6) -- (a7);
	\end{tikzpicture}\caption{Two posets $P$ and $Q$.}\label{fig:counterexample-generalposet}
\end{figure}

    On the other hand, in \cite[p.~535]{feichtner-yuzvinsky}, Feichtner and Yuzvinsky considered the atomic lattice $Q$ depicted on the right in Figure \ref{fig:counterexample-generalposet}. They computed the Hilbert series of the Chow ring $D(Q,\mathcal{G}_{\max})$ with respect to the maximal building set of $Q$, and they showed that it equals $1+3x$. Notice this is not even palindromic.
\end{remark}

%hhowever, the authors do not know whether it is possible to efficiently check whether two matroids are isomorphic (or whether they have the same Hilbert series). Recent progress in this direction was made by Rao and Sarma in \cite{rao-sarma}. LF: 

\begin{remark}
    The reduced characteristic polynomial is a Tutte--Grothendieck invariant of a matroid in the sense of \cite{brylawski-oxley} (essentially, the Tutte polynomial of the matroid determines it). A natural question is whether the Hilbert series of the Chow ring has the same property. As we have seen, the map $\M\mapsto \uH_{\M}(x)$ is (up to a sign) the inverse of the map $\M\mapsto \overline{\chi}_{\M}(x)$ in the incidence algebra of the lattice of flats, so this has some plausibility. In spite of that, we can find two matroids $\M_1$ and $\M_2$ of rank $4$ on $7$ elements having the same Tutte polynomial but whose Chow rings have different Hilbert series. Precisely, consider the matroids $\M_1^*$ and $\M_2^*$ depicted in Figure \ref{fig:matroids-tutte-counterexample} (the reason for depicting the duals instead of the original matroids is that they have rank $3$):
    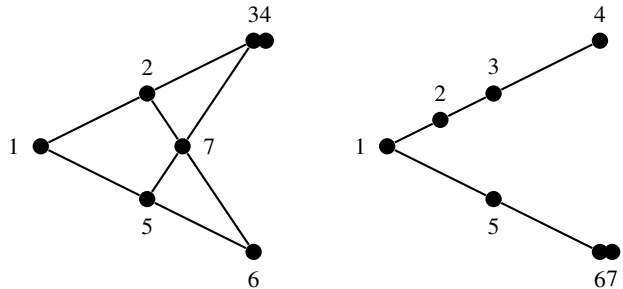
\begin{figure}[ht]
    \centering
	\begin{tikzpicture}  
	[scale=0.7,auto=center,every node/.style={circle,scale=0.8, fill=black, inner sep=2.7pt}] 
	\tikzstyle{edges} = [thick];
	
	\node[label=left:$1$] (a1) at (0,0) {};  
	\node[label=above:$2$] (a2) at (4/2,2/2)  {};  
	\node[label=above:$3$] (a3) at (8/2,4/2) {};
	\node[label=above:$4$] (a4) at (4.23,2) {};
	\node[label=below:$5$] (a5) at (4/2,-2/2)  {};  
	\node[label=below:$6$] (a6) at (8/2,-4/2)  {};    
	\node[label=right:$7$] (a7) at (5.33/2,0) {};
	
	\draw[edges] (a1) -- (a2);  
	\draw[edges] (a2) -- (a3);  
	\draw[edges] (a1) -- (a5);  
	\draw[edges] (a5) -- (a6);
	\draw[edges] (a7) -- (a2);
	\draw[edges] (a7) -- (a3);
	\draw[edges] (a7) -- (a5);
	\draw[edges] (a7) -- (a6);
	\end{tikzpicture}\qquad
	\begin{tikzpicture}  
	[scale=0.7,auto=center,every node/.style={circle,scale=0.8, fill=black, inner sep=2.7pt}] 
	\tikzstyle{edges} = [thick];
	
	\node[label=left:$1$] (a1) at (0,0) {};  
	\node[label=above:$2$] (a2) at (2/2,1/2)  {};
	\node[label=above:$3$] (a3) at (4/2,2/2)  {};  
	\node[label=above:$4$] (a4) at (8/2,4/2) {};
	\node[label=below:$5$] (a5) at (4/2,-2/2)  {};  
	\node[label=below:$6$] (a6) at (8/2,-4/2)  {};   
	\node[label=below:$7$] (a7) at (4.23,-4/2)  {};
	\draw[edges] (a1) -- (a2); 
	\draw[edges] (a2) -- (a3);  
	\draw[edges] (a3) -- (a4);  
	\draw[edges] (a1) -- (a5);  
	\draw[edges] (a5) -- (a6);
	\end{tikzpicture} \caption{The duals of the matroids $\M_1$ and $\M_2$}\label{fig:matroids-tutte-counterexample}
\end{figure}

The matroids $\M_1$ and $\M_2$ have the same Tutte polynomial:
    \[ T_{\M_1}(x,y) = T_{\M_2}(x,y) = x^{4} + 3 x^{3} + 2 x^{2} y + x y^{2} + y^{3} + 4 x^{2} + 5 x y + 3 y^{2} + 2 x + 2 y.\]
However, we have
    \begin{align*} 
        \uH_{\M_1}(x) &= x^3 + 30x^2 + 30x + 1,\\
        \uH_{\M_2}(x) &= x^3 + 31x^2 + 31x + 1.
    \end{align*}
Moreover, the same example shows that the Hilbert series of the augmented Chow ring is not a Tutte-Grothendieck invariant, because:
    \begin{align*} 
        \H_{\M_1}(x) &= x^4 + 37x^3 + 98x^2 + 37x + 1,\\
        \H_{\M_2}(x) &= x^4 + 38x^3 + 102x^2 + 38x + 1.
    \end{align*}
We also mention explicitly the fact that neither $\uH_{\M}(x)$ nor $\H_{\M}(x)$ determines the other.
\end{remark}

\subsection{Dominance of uniform matroids}\label{subsec:dominance-uniform}

Whenever $A(x)$ and $B(x)$ are polynomials with real coefficients, we will write $A(x) \preceq B(x)$ if the difference $B(x) - A(x)$ has nonnegative coefficients. In other words, $A(x)\preceq B(x)$ is equivalent to stating that $A(x)$ is coefficient-wise smaller than $B(x)$. An intriguing conjecture in the theory of Kazhdan--Lusztig polynomials of matroids, attributed to Gedeon (unpublished) and in the equivariant case to Proudfoot (see \cite[Conjecture~1.1]{lee-nasr-radcliffe} and \cite[Conjecture~1.6]{karn-proudfoot-nasr-vecchi}), asserts the following.

\begin{conjecture}
    Let $\M$ be a loopless matroid of rank $k$ on a ground set of size $n$. The following inequalities hold:
    \begin{align*}
        P_{\M}(x) &\preceq P_{\U_{k,n}}(x),\\
        Z_{\M}(x) &\preceq Z_{\U_{k,n}}(x).
    \end{align*}
    In other words, uniform matroids maximize coefficient-wise the Kazhdan--Lusztig and $Z$-polynomials among all matroids with fixed rank and size.
\end{conjecture}

This motivates us to formulate an analogous conjecture for $\uH_{\M}(x)$ and $\H_{\M}(x)$. In fact, in this alternative setting we will be able to prove it, and that constitutes precisely the content of Theorem~\ref{thm:dominance-uniform-main}. The main ingredients that make the proof possible are the formulas obtained in Propositions \ref{prop:hilbert-chow} and \ref{prop:hilbert-augmented-chow}. Before stating the main result of this subsection, we start with a useful combinatorial lemma.

\begin{lemma}\label{lemma:chains-flats-injection-uniform}
    Let $\M$ be a loopless matroid on a ground set $E$ having size $n$ and rank $k$. To each chain $F_0\subsetneq \cdots \subsetneq F_m$ of flats of $\M$ we can associate injectively a chain $G_0\subsetneq \cdots \subsetneq G_m$ of flats of $\U_{k,n}$ in such a way that $\rk_{\M}(F_j) = \rk_{\U_{k,n}}(G_j)$ for each $j=0,\ldots,m$.
\end{lemma}

\begin{proof}
    Let us assume that the ground set of $\M$ is the set of integers $E=\{1,\ldots,n\}$; the natural total order of the ground set induces a total order for the subsets of $E=[n]$ with fixed cardinality given by comparing lexicographically any pair of sets. Fix a chain of flats $F_0\subsetneq \cdots \subsetneq F_m$ of $\M$. Let us call $\rk_{\M}(F_j) = r_j$ for each $j=0,\ldots,m$. Among all the independent subsets of $F_0$ that have rank $r_0$, consider the lexicographically minimum set $I_0$. Since $I_0$ is independent in $\M$, it has cardinality at most $k$, and hence it is independent in $\U_{k,n}$ as well. We define the flat $G_0$ as the closure of $I_0$ in $\U_{k,n}$. Notice that 
        \[\rk_{\U_{k,n}}(G_0) = \rk_{\U_{k,n}}(I_0) = |I_0| = \rk_{\M}(I_0) = \rk_{\M}(F_0),\]
    as $G_0=E$ or $G_0=I_0$ according to whether $F_0=E$ or $F_0\subsetneq E$. 
    Assume we have already constructed $G_0\subsetneq \cdots \subsetneq G_s$ for $s\geq 0$. To construct the flat $G_{s+1}$ we proceed as follows. First, among all the independent sets of $\M$ contained in $F_{s+1}$ that have rank $r_{s+1}$ and contain $I_s$, consider the lexicographically minimum set $I_{s+1}$, and now define $G_{s+1}$ as the closure of $I_{s+1}$ in $\U_{k,n}$. Since $I_{s+1}$ was independent in $\M$ it is independent in $\U_{k,n}$ as well, and hence $\rk_{\M}(G_{s+1}) = \rk_{\M}(I_{s+1}) = |I_{s+1}| = r_{s+1}$. Since $I_{s+1}\supsetneq I_s$, the monotonicity of the closure operator in $\M$ guarantees that $G_{s+1}\supsetneq G_s$. Observe that the whole construction is injective, because each flat $F_i$ of the original chain in $\M$ can be recovered by taking the closure of $G_i$ in $\M$. 
\end{proof}

\begin{theorem}\label{thm:uniform-dominates}
    Let $\M$ be a loopless matroid of rank $k$ on a ground set of size $n$. The following inequalities hold:
    \begin{align*}
        \uH_{\M}(x) &\preceq \uH_{\U_{k,n}}(x),\\
        \H_{\M}(x) &\preceq \H_{\U_{k,n}}(x).
    \end{align*}
    In other words, uniform matroids maximize coefficient-wise the Hilbert series of Chow rings and augmented Chow rings among all matroids with fixed rank and size.
\end{theorem}
 
\begin{proof}
    Consider $\uH_{\M}(x)$ and $\H_{\M}(x)$. The formulas of Proposition~\ref{prop:hilbert-chow} and Proposition~\ref{prop:hilbert-augmented-chow} express them as sums over a set of certain chain of flats of $\M$ a polynomial with nonnegative coefficients. Using the map of Lemma~\ref{lemma:chains-flats-injection-uniform} we can associate injectively a chain of flats in $\U_{k,n}$ in which the flats have the same ranks correspondingly. In other words, each summand appearing in the expressions of $\uH_{\M}(x)$ (resp. $\H_{\M}(x)$) appears in $\uH_{\U_{k,n}}(x)$ too (resp. in $\H_{\U_{k,n}}(x)$). This proves the desired inequalities.   
\end{proof}

The following generalization of Lemma~\ref{lemma:chains-flats-injection-uniform}, conjectured in a previous version of this paper, is proved in \cite{elias-miyata-proudfoot-vecchi}.
\begin{conjecture}[Confirmed in \cite{elias-miyata-proudfoot-vecchi}]\label{conj:weak-maps-monotonicity}
    Consider two matroids $\M$ and $\N$ on the same ground set, and assume that there exists a rank-preserving weak map $\N\to \M$.\footnote{By asserting that there is a rank-preserving weak map $\N\to \M$, we mean that $\M$ and $\N$ are matroids on the same ground set, having the same rank, and with the property that all the independent sets of $\M$ are independent in $\N$.} Denote by $\Delta(\mathcal{L}(\M))$ and $\Delta(\mathcal{L}(\N))$ the set of all the chains of flats in $\M$ and $\N$, respectively. There exists a map $\varphi\colon\Delta(\mathcal{L}(\M))\to \Delta(\mathcal{L}(\N))$ such that:
    \begin{enumerate}[(i)]
        \item $\varphi$ is injective.
        \item $\varphi$ maps the only empty chain of flats of $\M$ to the only empty chain of flats of $\N$.
        \item For each $m\geq 0$, the image of a chain $F_0\subsetneq\cdots\subsetneq F_m$ is a chain $G_0\subsetneq\cdots\subsetneq G_m$ with the property that $\rk_{\M}(F_i)=\rk_{\N}(G_i)$ for each $0\leq i\leq m$. 
    \end{enumerate}
\end{conjecture}
Replacing Lemma~\ref{lemma:chains-flats-injection-uniform} with Conjecture \ref{conj:weak-maps-monotonicity} in the proof of Theorem \ref{thm:uniform-dominates} shows that Hilbert series of Chow rings and augmented Chow rings are both monotonic under weak maps. Observe that it is not true in general that there is an order-preserving map from the family of flats of $\N$ to the family of flats of $\M$ (see the example and the digression in \cite[p. 259]{lucas}).
An appealing feature of this line of attack is that, in contrast to other techniques to prove that a valuation is monotone under weak maps, the content of Conjecture~\ref{conj:weak-maps-monotonicity} is a purely combinatorial statement.

\section{The intersection cohomology module}\label{sec:four}

\subsection{Kazhdan--Lusztig and \texorpdfstring{$Z$}{Z}-polynomials}

To any Coxeter group $W$ one can associate its Bruhat poset. The classical Kazhdan--Lusztig theory \cite{kazhdan-lusztig} associates to each interval in this poset a polynomial encoding fundamental representation-theoretic information. A more general procedure that attaches polynomials to intervals in posets (under certain assumptions) was devised by Stanley \cite{stanley-local}, thus giving name to the Kazhdan--Lusztig--Stanley polynomials of posets \cite{proudfoot-kls}. In particular, if the poset is the lattice of flats of a loopless matroid $\M$, one can define such polynomials.

Although the following families of polynomials were introduced with a slightly different statement in \cite{elias-proudfoot-wakefield} and \cite{proudfoot-xu-young}, by following \cite[Theorem~2.2]{braden-vysogorets} we define:

\begin{definition}
    There is a unique way of assigning a polynomial to each loopless matroid $\M$, say $\M \mapsto P_{\M}(x)\in \mathbb{Z}[x]$ in such a way that the following conditions hold:
    \begin{enumerate}[(i)]
        \item If $\rk(\M) = 0$, then $P_{\M}(x) = 1$.
        \item If $\rk(\M) > 0$, then $\deg P_{\M}(x) < \frac{\rk(\M)}{2}$.
        \item The polynomial $Z_{\M}(x)$ defined by
            \[ Z_{\M}(x) = \sum_{F\in\mathcal{L}(\M)} x^{\rk(F)}\cdot P_{\M/F}(x),\]
        is palindromic and has degree $\rk(\M)$.
    \end{enumerate}
\end{definition}

The polynomials $P_{\M}(x)$ and $Z_{\M}(x)$ arising from the above definition are called, respectively, the \textit{Kazhdan--Lusztig polynomial} and the \textit{$Z$-polynomial} of the matroid $\M$. In analogy to the case of $\M\mapsto \uH_{\M}(x)$ and $\M\mapsto \H_{\M}(x)$, we can extend this definition to also cover matroids with loops. We define $P_{\M}(x):=0$ and $Z_{\M}(x):= Z_{\M\smallsetminus\{\text{loops}\}}(x)$ whenever $\M$ has loops.

A fundamental fact, which also resembles the case of Kazhdan--Lusztig polynomials of intervals in a Bruhat poset \cite{elias-williamson}, is that the polynomials $P_{\M}(x)$ always have nonnegative coefficients. In fact, the following is the second main result of \cite{braden-huh-matherne-proudfoot-wang}.

\begin{theorem}[{\cite[Theorem~1.2]{braden-huh-matherne-proudfoot-wang}}]
    For every matroid $\M$, we have the following:
    \begin{enumerate}[\normalfont(i)]
        \item The polynomial $P_{\M}(x)$ has nonnegative coefficients.
        \item The polynomial $Z_{\M}(x)$ has unimodal coefficients.
    \end{enumerate}
\end{theorem}

Both of these statements follow from interpreting the coefficients of the Kazhdan--Lusztig and $Z$-polynomials as the graded dimensions of certain modules. In particular, for the $Z$-polynomial, as we will explain in the next subsection, the module in question is called the ``intersection cohomology'' of the matroid $\M$. Again, the validity of the hard Lefschetz property for this module is what guarantees the unimodality of the coefficients of the $Z$-polynomial, which we will later extend to the stronger property of being $\gamma$-positive.

\subsection{Intersection cohomology of matroids}\label{subsec:ih}

The \textit{graded M\"obius algebra} of $\M$ is the graded vector space
\[
\H(\M) = \bigoplus_{F \in \cL(\M)} \Q y_F,
\]
where $y_F$ is placed in degree $\rk(F)$.  It is made into a graded algebra via the multiplication
\[
y_F \cdot y_G = \begin{cases}y_{F \vee G} & \text{if $\rk(F) + \rk(G) = \rk(F \vee G)$},\\ 0 & \text{if $\rk(F) + \rk(G) > \rk (F \vee G)$}.\end{cases} 
\]

We note that $\H(\M)$ is a graded subalgebra of $\CH(\M)$ \cite[Proposition~2.18]{semismall}; thus, we may view $\CH(\M)$ as a graded $\H(\M)$-module.

\begin{definition}\label{def:ih}
    Let $\M$ be a matroid. The \textit{intersection cohomology module} of $\M$, denoted by $\IH(\M)$, is the unique (up to isomorphism) indecomposable graded $\H(\M)$-module direct summand of $\CH(\M)$ that contains $\H(\M)$.\footnote{In \cite[Definition 3.2]{braden-huh-matherne-proudfoot-wang}, a construction of $\IH(\M)$ is given as an explicit $\H(\M)$-submodule of $\CH(\M)$.  However, for the purposes of this paper, it will be sufficient to only know $\IH(\M)$ up to isomorphism because we are mainly concerned with its Poincar\'e polynomial.}  The \textit{stalk at the empty flat}\footnote{We point to \cite[Section~5]{braden-huh-matherne-proudfoot-wang} for more about the terminology ``stalk" in this context.} of $\IH(\M)$ is the graded vector space $\IH(\M)_\varnothing := \IH(\M) \otimes_{\H(\M)} \Q$, where $\Q$ is the one-dimensional graded $\H(\M)$-module placed in degree zero.
\end{definition}

\begin{theorem}[{\cite[Theorem~1.9]{braden-huh-matherne-proudfoot-wang}}]Let $\M$ be a loopless matroid.
\begin{itemize}
    \item The Kazhdan--Lusztig polynomial of $\M$ coincides with the Hilbert series of the stalk at the empty flat of $\IH(\M)$.  In other words,
        \[P_{\M}(x) = \sum_{j\ge0} \dim_{\Q}(\IH^j(\M)_\varnothing) \, x^j.\]
    
    \item The $Z$-polynomial of $\M$ coincides with the Hilbert series of the intersection cohomology module of $\M$. In other words,
    \[Z_{\M}(x) = \sum_{j\ge0} \dim_{\Q}(\IH^j(\M)) \, x^j.\]
\end{itemize}
\end{theorem}

\begin{remark}\label{remark:stalk-empty}
   The similarity between the recurrence linking $P_{\M}(x)$ to $Z_{\M}(x)$  with the one linking $\uH_{\M}(x)$ to $\H_{\M}(x)$ is also hinted by the fact that $\CH(\M)_{\varnothing} := \CH(\M)\otimes_{\H(\M)}\mathbb{Q}$, and that the latter is isomorphic to $\uCH(\M)$ (see \cite[Remark~1.4]{semismall}). In other words, in terms of stalks \cite[Section~5]{braden-huh-matherne-proudfoot-wang}, the Chow ring $\uCH(\M)$ is the stalk at the empty flat of the augmented Chow ring $\CH(\M)$.
\end{remark}

\subsection{\texorpdfstring{$\gamma$}{gamma}-positivity}

Now our main goal is to prove that the $Z$-polynomial of a matroid is always a $\gamma$-positive polynomial. This was conjectured in \cite[Conjecture 5.6]{ferroni-nasr-vecchi}, and known to be valid in a number of cases, see \cite[Theorem~1.9]{ferroni-nasr-vecchi}. 

In contrast to the Hilbert series of both the Chow ring and the augmented Chow ring, we now lack a version of the semi-small decompositions of Theorem~\ref{thm:semismall-decompositions} for the intersection cohomology module. However, we can use the following result of Braden and Vysogorets, which can be stated only using the basic combinatorial theory of the Kazhdan--Lusztig and $Z$-polynomials of matroids. Before formulating it, we introduce the following notation:
    \[ \tau(\M) := 
    \begin{cases}
        [x^{\frac{\rk(\M)-1}{2}}] P_{\M}(x) & \text{if $\rk(\M)$ is odd,}\\
        0 & \text{if $\rk(\M)$ is even.}
    \end{cases}
    \]

\begin{theorem}[{\cite[Theorem~2.8]{braden-vysogorets}}]
    Let $\M$ be a loopless matroid of rank $k$, and let $i\in E$ be an element of the ground set that is not a coloop. Then,
    \begin{align*}
    P_{\M}(x) &= P_{\M\smallsetminus\{i\}}(x) - xP_{\M/\{i\}}(x) + \sum_{F\in\mathscr{S}_i} \tau\left(\M/{(F\cup\{i\})}\right)\, x^{\frac{k - \rk(F)}{2}}\, P_{\M|_F}(x),\\
     Z_{\M}(x) &= Z_{\M\smallsetminus\{i\}}(x) + \sum_{F\in\mathscr{S}_i} \tau\left(\M/{(F\cup\{i\})}\right)\, x^{\frac{k - \rk(F)}{2}}\, Z_{\M|_F}(x).
    \end{align*}
\end{theorem}

Recall that the family $\mathscr{S}_i$ of flats was introduced in equation \eqref{eq:family_S_i}. Observe that a~priori this result does not witness any decompositions at the level of (stalks of) intersection cohomology modules. However, this provides a counterpart for Corollary~\ref{coro:recursion-hilbert-chow} that suffices for our purposes. Also, note that for $F\in \mathscr{S}_i$, we have $\tau(\M/(F\cup\{i\}))=0$ whenever $k-\rk(F)$ is odd.

\begin{theorem}\label{thm:z-is-gamma-positive}
    For every matroid $\M$, the polynomial $Z_{\M}(x)$ is $\gamma$-positive.
\end{theorem}

\begin{proof}
    It suffices to prove the statement only for loopless matroids. We proceed by induction on the size of the ground set of $\M$. If the matroid $\M$ is empty, the rank of $\M$ is zero and thus the $Z$-polynomial is $Z_{\M}(x) = 1$, and the associated $\gamma$-polynomial is $\gamma(Z_{\M},x) = 1$.
    
    Assuming that we have proved the validity of the statement for all matroids with ground sets of cardinality at most $n-1$, let us consider a matroid $\M$ having ground set of cardinality $n$. If $\M$ is a Boolean matroid, i.e. $\M \cong \U_{n,n}$ for some $n\geq 1$, then $Z_{\M}(x) = (x+1)^n$. In this case, one obtains $\gamma(Z_{\M}, x) = 1$, which has nonnegative coefficients.
    
    If $\M$ is not Boolean, then there is at least one element $i\in E$ that is not a coloop. Using the result by Braden and Vysogorets, we obtain the following recurrence for the $\gamma$-polynomial:
    \[ \gamma(Z_{\M},x) = \gamma(Z_{\M\smallsetminus\{i\}},x) + \sum_{F\in\mathscr{S}_i} \tau\left(\M/{(F\cup\{i\})}\right)\, x^{\frac{k - \rk(F)}{2}}\, \gamma(Z_{\M|_F},x).\]
    The induction hypothesis guarantees that each of the summands on the right-hand side has nonnegative coefficients. The proof is complete.
\end{proof}

\begin{remark}
    We stress the fact that we are using that the $\tau$-invariant is always a nonnegative integer. This fact is highly non-trivial and follows from the nonnegativity of the coefficients of the Kazhdan--Lusztig polynomials of matroids. We know of no proof of the nonnegativity of $\tau(\M)$ that does not rely on that.
\end{remark}

\begin{remark}
    Although the Kazhdan--Lusztig polynomial is not palindromic in general, a reasonable question that the reader might ask is whether it is \emph{non-symmetric $\gamma$-positive} in the sense of \cite[Section 5.1]{athanasiadis-gamma-positivity}. In particular, one could ask whether $P_{\M}(x)$ is always right or left $\gamma$-positive. Unfortunately it is not the case. As Athanasiadis points out, being right or left $\gamma$-positive implies unimodality, and the peak of the coefficients is attained in the middle terms. However, observe that
    \[ P_{\U_{15,16}}(x) = 1430x^7 + 32032x^6 + \boxed{91728}x^5 + 76440x^4 + 23100x^3 + 2640x^2 + 104x + 1\]
    and the peak is not in the middle terms (which correspond to degrees $3$ and $4$). In fact, experimentation suggests that the peak of $P_{\U_{k,n}}(x)$ is always attained approximately at the coefficient of degree $\lfloor\frac{k}{3}\rfloor$.
\end{remark}

\begin{remark}
    It is natural to ask whether the equivariant $Z$-polynomial (see \cite[Section~6]{proudfoot-xu-young}) is ``equivariant $\gamma$-positive'', in the sense of \cite[Section~5.2]{athanasiadis-gamma-positivity}. The answer to this question is negative. In fact, it is not difficult to find explicit counterexamples. One such instance is the matroid $\U_{2,2}$ with the action induced by its full automorphism group. 
\end{remark}

\section{Concluding remarks}\label{sec:five}

The purpose of this section is to discuss several problems, results, and remarks regarding the real-rootedness of the polynomials we addressed in this article. 

\subsection{Haglund--Zhang polynomials}\label{subsec:haglund-zhang}

In this subsection, we will establish the real-rootedness of $\H_\M(x)$ when $\M = \U_{k,n}$ is an arbitrary uniform matroid, therefore proving Theorem~\ref{thm:intro-hz}. We will show that $\H_{\U_{k,n}}(x)$ is an example of a class of polynomials introduced and proven to be real-rooted by Haglund and Zhang in \cite{haglund-zhang}.

To any sequence $\mathbf{s} = (s_1,\ldots,s_n) \in \Z_{>0}^n$, Haglund and Zhang associate a generalized binomial Eulerian polynomial $\widetilde{E}_n^{\mathbf{s}}(x)$ in the following way.  First, define the set 
    \[
    \cI_n^{\mathbf{s}} = \{\mathbf{e} = (e_1,\ldots, e_n) \in \Z^n : 0 \le e_i < s_i \text{ and } 0 \le i \le n\},
    \]
where we set $e_0 = e_{n+1} = 0$ and $s_0 = s_{n+1} = 1$.  Furthermore, we say that $i \in [0,n]$ is an ascent of $\mathbf{e} \in \cI_n^{\mathbf{s}}$ if $\frac{e_i}{s_i} < \frac{e_{i+1}}{s_{i+1}}$, and that it is a collision if $\frac{e_i}{s_i} = \frac{e_{i+1}}{s_{i+1}}$.  We write $\mathrm{asc}(\mathbf{e})$ and $\mathrm{col}(\mathbf{e})$ for the number of ascents and collisions of $\mathbf{e}$, respectively.  Now define the polynomial
\[
\widetilde{E}_n^{\mathbf{s}}(x) := \sum_{\mathbf{e} \in \cI_n^{\mathbf{s}}} (1+x)^{\mathrm{col}(\mathbf{e})}x^{\mathrm{asc}(\mathbf{e})}.
\]

The main result of Haglund and Zhang \cite[Theorem~1.1]{haglund-zhang} proves that all such polynomials are real-rooted.

\begin{theorem}\label{thm:hz}
    Let $\mathbf{s} = (s_1,\ldots,s_n) \in \Z_{>0}^n$.  Then $\widetilde{E}_n^{\mathbf{s}}(x)$ is real-rooted.
\end{theorem}

We mention explicitly that one of the motivations of Haglund and Zhang to define their polynomials originates in the work of Savage and Visontai \cite{savage-visontai} and Gustafsson and Solus \cite{gustafsson-solus}, in which they define similar real-rooted polynomials which are indexed by vectors of positive integers. 

Based on computational evidence, we first conjectured and then proved that, for convenient choices of the vector $\mathbf{s}$, one can obtain the Hilbert series of the augmented Chow ring of arbitrary uniform matroids.

\begin{theorem}\label{thm:unifrealrooted}
    For $\mathbf{s} = (n - k + 2, n - k + 3, \dots, n)$, we have
    \begin{equation*}
      \widetilde{E}^{\mathbf s}_{k-1}(x) = \H_{\U_{k, n}}(x).
  \end{equation*}
  In particular, Theorem~\ref{thm:intro-hz} holds, i.e. the Hilbert series of the augmented Chow ring of a uniform matroid is always a real-rooted polynomial.
\end{theorem}

Observe that the preceding statement is an extension of the real-rootedness of the binomial Eulerian polynomials. When the uniform matroid is the Boolean matroid $\U_{n,n}$ by taking $\mathbf{s}=(2,3,\ldots,n)$ one has $\widetilde{E}^{\mathbf{s}}_{n-1}(x) = \widetilde{A}_n(x) = \H_{\U_{n,n}}(x)$. This particular case was precisely the content of another result of Haglund and Zhang \cite[Theorem~3.1]{haglund-zhang}. The proof of Theorem~\ref{thm:unifrealrooted} is postponed to Appendix \ref{appendixb} because it involves many intricate calculations.

\begin{remark}
   Given that the polynomials $\H_{\U_{k, n}}(x)$ are related to 
   the generalized binomial Eulerian polynomials
   studied in \cite{haglund-zhang}, it is natural to ask whether the Hilbert--Poincar\'e polynomials
   of augmented Chow rings of arbitrary matroids arise in this way. The answer is no: if $\M=\U_{3, 4} \oplus \U_{1, 1}$, we can compute $\H_\M(x) = 1 + 23x + 55x^2 + 23x^3 + x^4$, and an exhaustive computer
   search shows that there is no $\mathbf{s} \in \Z_{>0}^{\rk(\M) - 1}$ with $\widetilde{E}^{\mathbf s}_{\rk(\M) - 1}(x) = \H_{\M}(x)$.
\end{remark}

\begin{remark}
   We do not know of any analogues of Theorem~\ref{thm:unifrealrooted} for the usual Chow ring. Recall that $\uH_{\U_{n-1,n}}(x) = \frac{1}{x} d_n(x)$, and hence it is reasonable to search among known generalizations of the derangement polynomials. Although the work of Gustafsson and Solus \cite{gustafsson-solus} provides one such generalization via Ehrhart local $h^*$-vectors, we were not able to produce $\uH_{\U_{n-2,n}}(x)$ as a particular case of their polynomials.
   In general, the polynomials $\uH_{\U_{n-1,n}}(x)$ do not arise as instances of the polynomials
   studied in \cite{haglund-zhang}: if $\M = \U_{5,6}$, then $\uH_{\M}(x) = 1 + 51x + 161x^2 + 51x^3 + x^4$, and
   an exhaustive computer search shows that there is no $\mathbf{s} \in \Z_{>0}^{\rk(\M) - 2}$ with $\widetilde{E}^{\mathbf s}_{\rk(\M) - 2}(x) = \uH_{\M}(x)$.
\end{remark}

\subsection{Braid matroids}

The \emph{braid matroid} $\mathsf{K}_n$ is defined as the graphic matroid associated to a complete graph on $n$ vertices. Equivalently, this is the matroid associated to the Coxeter arrangement of type $\mathrm{A}_{n-1}$. Since every flat of $\mathsf{K}_n$ is a complete subgraph or a vertex disjoint union of complete graphs, one may conclude that the flats of $\mathsf{K}_n$ are in bijection with the partitions of the set $[n]$. Hence, the lattice $\mathcal{L}(\mathsf{K}_n)$ is also sometimes referred to as the \emph{partition lattice} of $n$.

One of the most intriguing problems in the Kazhdan--Lusztig theory of matroids has been that of understanding combinatorially the Kazhdan--Lusztig polynomial and the $Z$-polynomial of $\mathsf{K}_n$. Recently, Ferroni and Larson \cite{ferroni-larson} found a strikingly concrete description of the coefficients of $P_{\mathsf{K}_n}(x)$ and $Z_{\mathsf{K}_n}(x)$. These coefficients enumerate labelled matroids that arise from direct sums of series-parallel networks. It is tempting to ask whether a counterpart of that beautiful description exists for the Chow and augmented Chow polynomials of $\mathsf{K}_n$.

\begin{problem}
    Study the Chow polynomial and the augmented Chow polynomial of braid matroids.
\end{problem}

More specifically, it would be of interest to produce recursions, explicit formulas, combinatorial interpretations of the coefficients, and a proof of the real-rootedness for the polynomials $\uH_{\mathsf{K}_n}(x)$ and $\H_{\mathsf{K}_n}(x)$.

We observe that a related object, the Chow ring of the lattice of flats of $\mathsf{K}_n$ using the minimal building set (as opposed to the maximal one), i.e. $D(\mathcal{L}(\mathsf{K}_n),\mathcal{G}_{\min})$, cf. \cite[Example 1]{feichtner-yuzvinsky}, yields the (rational) cohomology ring of $\overline{\mathcal{M}}_{0,n-1}$, the Deligne--Mumford compactification
of the moduli space of complex projective lines with $n-1$ marked points. This space and its cohomology are of significant importance in algebraic geometry, field theory, and the theory of operads; for more information about this, we refer the reader, e.g. to \cite{ginzburg-kapranov,kontsevich,etingof-henriques-kamnitzer-rains,dotsenko}.

Turning back towards real-rootedness, we point out that the reduced characteristic polynomial of the braid matroid $\mathsf{K}_n$ satisfies $\overline{\chi}_{\mathsf{K}_n}(x) = (x-2)\cdot (x-3)\cdots (x-n+1)$, and hence is real-rooted. Also, the $h$-polynomial of the Bergman complex of $\mathsf{K}_n$, i.e. $h_{\Delta(\widehat{\mathcal{L}}(\mathsf{K}_n))}(x)$ (see Appendix \ref{appendix} for a more precise definition) is a real-rooted polynomial \cite[Proposition~4.2]{athanasiadis-kalampogia}. 

In \cite[p.~527]{feichtner-yuzvinsky} Feichtner and Yuzvinsky used Proposition~\ref{prop:hilbert-chow} to slightly simplify the resulting expression for $\uH_{\mathsf{K}_n}(x)$. Their formula can be rewritten as follows in terms of Stirling numbers of the second kind $\tstirlingtwo{a}{b}$, i.e. the number of partitions of the set $[a]$ into $b$ blocks.

\begin{proposition}
    The Hilbert series of the Chow ring of the braid matroid $\mathsf{K}_n$ is given by:
    \[ \uH_{\mathsf{K}_n}(x) = \sum_{R\subseteq [n]} \prod_{i=1}^{|R|} \frac{x(1-x^{r_i-r_{i-1}-1})}{1-x} \stirlingtwo{n-r_{i-1}}{n-r_i}.\]
    In the above sum, $r_0:=0$ and the elements of $R$ in increasing order are $r_1 < r_2 < \cdots$.
\end{proposition}

Using this formula, we have verified that $\uH_{\mathsf{K}_n}(x)$ is real-rooted for all $n\leq 30$. 

\subsection{More observations regarding real-rootedness}

One of the standard techniques to prove that a family of polynomials is real-rooted is that of \emph{interlacing sequences}. We refer to \cite[Section 7]{branden} for more details. In particular, a phenomenon observed in this setting that we pose here as a conjecture is the following.

\begin{conjecture}
    For every matroid $\M$, the polynomial $\uH_{\M}(x)$ interlaces $\H_{\M}(x)$.
\end{conjecture}

Notice that it is not possible to formulate an analogous conjecture in the Kazhdan--Lusztig setting since the counterpart polynomials, i.e. $P_{\M}(x)$ and $Z_{\M}(x)$, have degrees at most $\lfloor\frac{\rk(\M)-1}{2}\rfloor$ and exactly $\rk(\M)$ respectively.

Going back to a further technique, first proved by Aissen, Schoenberg, and Whitney \cite{aissen-schoenberg-whitney}, to establish the real-rootedness of a degree $d$ polynomial $p(x):=\sum_{j=0}^d p_j\, x^j$, one may consider the \emph{Toeplitz matrix}, $P:=(p_{j-i})_{0\leq i,j\leq d}$, where by definition $p_i:=0$ if $i<0$. It is known by the main result of \cite{aissen-schoenberg-whitney} that all the roots of $p(x)$ are real and negative if and only if $P$ is a totally-positive matrix, i.e. all the minors of $P$ are nonnegative. In the case the polynomial $p(x)$ is $\uH_{\M}(x)$, $\H_{\M}(x)$, or $Z_{\M}(x)$, the following are natural questions that one could ask:
\begin{itemize}
    \item Is it possible to provide a combinatorial interpretation for the minors of the Toeplitz matrices?
    \item Since in $\uCH(\M)$, $\CH(\M)$, and $\IH(\M)$ the Hodge--Riemann relations hold in arbitrary degree, can one relate them with the minors of the Toeplitz matrices of $\uH_{\M}(x)$, $\H_{\M}(x)$, or $Z_{\M}(x)$?
\end{itemize}

A positive answer to the second of these questions would provide a combinatorial application of the Hodge--Riemann relations in degree greater than one; this is a question posed by Huh in \cite{huh-icm18}.

\subsection{Chow rings and Koszulness}

Let $A$ be a graded algebra over $\mathbb{Q}$ and consider a minimal free graded $A$-resolution of $\mathbb{Q}$,
    \[ \cdots \stackrel{\phi_3}{\longrightarrow} A^{b_2}\stackrel{\phi_2}{\longrightarrow} A^{b_1}\stackrel{\phi_1}{\longrightarrow} A \longrightarrow \mathbb{Q}. \]
In other words, all the nonzero entries of the matrices $\phi_i$ are homogeneous and have positive degree. Given that $\phi_i \otimes \mathbb{Q}$ is identically zero in a minimal resolution, we have
\[ \operatorname{Tor}_i^A(\mathbb{Q},\mathbb{Q}) \cong A^{b_i} \otimes\mathbb{Q} \cong \mathbb{Q}^{b_i} \cong \operatorname{Ext}_A^i(\mathbb{Q},\mathbb{Q}).\]
The \emph{Poincar\'e series} of $A$, denoted $\operatorname{Poin}(A,x)$, is the generating function of the graded dimensions of $\operatorname{Tor}_i^A(\mathbb{Q},\mathbb{Q})$. In other words,
    \[ \Poin(A,x) = \sum_{i=0}^{\infty} \dim_{\mathbb{Q}} \Tor_i^A(\mathbb{Q},\mathbb{Q})\, x^i = \sum_{i=0}^{\infty} \dim_{\mathbb{Q}} \Ext_A^i(\mathbb{Q},\mathbb{Q})\, x^i.\] 

Shifting the degrees such that all $\phi_i$ become maps of degree zero, we produce a grading on $\Tor_i^A(\mathbb{Q},\mathbb{Q}) = \bigoplus_{j\geq 0} \left(\Tor_i^A(\mathbb{Q},\mathbb{Q})\right)^j$; analogously, we conclude that there is a grading $\Ext^i_A(\mathbb{Q},\mathbb{Q}) = \bigoplus_{j\geq 0} \left(\Ext^A_i(\mathbb{Q},\mathbb{Q})\right)^j$.

We say that $A$ is a \emph{Koszul algebra} if $\left(\Tor_i^A(\mathbb{Q},\mathbb{Q})\right)^{j} = 0$ for each $i\neq j$ or, equivalently, if $\left(\Ext^i_A(\mathbb{Q},\mathbb{Q})\right)^{j} = 0$ for each $i\neq j$. A basic reference on Koszulness is the survey \cite{froberg} by Fr\"oberg, whose notation and terminology we follow.

In \cite[Conjecture~2]{dotsenko} Dotsenko conjectured that the Chow ring $\uCH(\M)$ is a Koszul algebra for every matroid $\M$. This, along with its version for augmented Chow rings, was recently proved by Mastroeni and McCullough.

\begin{theorem}[\cite{mastroeni-mccullough}]
    The Chow ring and the augmented Chow ring of a loopless matroid $\M$ are Koszul.
\end{theorem}

Unsurprisingly, the property of being Koszul imposes heavy restrictions on the Hilbert series of the graded algebra $A$. In fact, the following identity holds:
    \[ \Hilb(A, x) \cdot \Poin(A,-x) = 1.\]

Particularizing this property for both Chow rings and augmented Chow rings yields that for every matroid $\M$,
    \[ \frac{1}{\uH_{\M}(-x)}\enspace \text{ and }\enspace \frac{1}{\H_{\M}(-x)}\enspace \text{ are series with nonnegative coefficients}.\]
As a consequence, one concludes the validity of several inequalities satisfied by the coefficients of $\H_{\M}(x)$ and $\uH_{\M}(x)$. For example, for a matroid $\M$ of rank $5$, the Hilbert series of its Chow ring has the form $\uH_{\M}(x) = 1+ax+bx^2+ax^3+x^4$ with $a\leq b$; a straightforward computation reveals that in this case:
    \[ \frac{1}{\uH_{\M}(-x)} = 1 + ax + (a^2-b) x^2 + (a^3-2ab+a) x^3 + \cdots\]
In particular, the nonnegativity of the quadratic term yields $a^2\geq b$, which proves that $\uH_{\M}(x)$ has log-concave coefficients; this property also supports the real-rootedness conjectures because it is well-known that being real-rooted is a stronger property than having log-concave coefficients.

In the monograph by Polishchuk and Positselski \cite[Chapter~7]{polishchuk-positselski}, one finds the following result, which essentially states a list of inequalities that the graded dimensions of a Koszul algebra always satisfy. 

\begin{theorem}[{\cite[Chapter~7, Theorem~2.1]{polishchuk-positselski}}]\label{thm:polischuk}
    Let $A=\bigoplus_{i\geq 0} A^i$ be a Koszul algebra, and let $a_i = \dim_{\mathbb{Q}} A^i$ for each $i\geq 0$. Consider the infinite matrix $X = (a_{j-i})_{i,j=0}^{\infty}$ with the convention that $a_i = 0$ whenever $i<0$. For each finite list $I=(i_1,\ldots,i_s)$ of positive integers, the $s\times s$ square submatrix $X_I$ of $X$ obtained by choosing the entries on the rows $(0,i_1,i_1+i_2,\ldots, i_{1}+\cdots+i_{s-1})$ and the columns $(i_1,i_1+i_2,\ldots,i_1+\cdots+i_s)$ satisfies:
        \[ \det(X_I) \geq 0.\]
\end{theorem}

Notice that, in the case of $\uH_{\M}(x)$ and $\H_{\M}(x)$, the above result asserts that \emph{certain} minors of their Toeplitz matrices are nonnegative, whereas Conjecture~\ref{conj:real-rootedness} is equivalent to stating that \emph{all of them} are nonnegative (again, via the main result of \cite{aissen-schoenberg-whitney}). Observe that, as Polishchuk and Positselski mention in \cite[p.~137]{polishchuk-positselski}, the Koszulness of $A$ in general is not strong enough to conclude the total positivity of the matrix $X$ in Theorem~\ref{thm:polischuk}.

Also, in \cite[Section~4]{reiner-welker} Reiner and Welker make several general observations that yield the real-rootedness of $\uH_{\M}(x)$ and $\H_{\M}(x)$ in a number of cases. For example, as a consequence of \cite[Proposition~4.13]{reiner-welker}, being Koszul guarantees that $\H_{\M}(x)$ and $\uH_{\M}(x)$ have at least one real root. Additionally, since both the augmented Chow ring and the Chow ring are Koszul, Artinian, Gorenstein algebras, and their Hilbert series are $\gamma$-positive polynomials (in particular, the notion of \emph{Charney--Davis quantity} defined in \cite{reiner-welker} yields a nonnegative number), by using \cite[Corollary~4.3 and Corollary~4.14]{reiner-welker} we conclude the following.

\begin{theorem}
    Let $\M$ be a matroid of rank $k\leq 5$. Then $\uH_{\M}(x)$ is real-rooted. If $k\leq 4$, then $\H_{\M}(x)$ is real-rooted.
\end{theorem}

\begin{remark}
   It is natural to also ask if the graded M\"obius algebra is Koszul. The relevance of this is clear, as the Hilbert--Poincar\'e series of $\H(\M)$ has as coefficients the number of flats of each rank of $\M$. In other words,
        \[ \Hilb(\H(\M), x) = \sum_{F\in \mathcal{L}(\M)} x^{\rk(F)} = \sum_{j=0}^k W_k\, x^k,\]
    where $W_j = |\{F\in\mathcal{L}(\M): \rk(F) = j\}|$ are the Whitney numbers of the second kind. Although \cite[Theorem~1.1]{braden-huh-matherne-proudfoot-wang} proves that these coefficients are top-heavy, i.e. $W_i \leq W_j$ whenever $i\leq j \leq k-i$, an outstanding conjecture posed by Rota \cite{rota} asserts that they are unimodal. The Koszulness of $\H(\M)$ a priori would provide interesting inequalities between the numbers $W_i$. Unfortunately, it does not hold in general; for example, the uniform matroid $\U_{3,5}$ satisfies
        \[ \Hilb(\H(\M), x) = 1 + 5x + 10x^2 + x^3.\]
    However,
        \[ \frac{1}{\Hilb(\H(\M),-x)} = 1 + 5x + 15x^2 + 26x^3 - 15x^4 - 320x^5 - \cdots,\]
    which yields the impossibility for Koszulness since there are negative coefficients.
\end{remark}

\appendix

\section{}\label{appendix}

In this appendix we show how Corollary~\ref{corollary:uH-uniform} can be used to prove the conjecture posed by Hameister, Rao, and Simpson in \cite[Conjecture 6.2]{hameister-rao-simpson}.

For a matroid $\M$ we denote by $\Delta(\widehat{\mathcal{L}}(\M))$ the order complex of the proper part of the lattice of flats of $\M$; this complex is also known as the \emph{Bergman complex} and has been studied for example in \cite{ardila-klivans}. This is a simplicial complex whose simplices correspond to chains of proper nonempty flats of $\M$. One can consider its \emph{$f$-polynomial}, defined by $f_{\Delta(\widehat{\mathcal{L}}(\M))}(x) := \sum_{i=0}^d f_{i-1}\,x^{d-i}$ where each $f_i$ counts the number of $i$-dimensional faces and $d:=\dim \Delta(\widehat{\mathcal{L}}(\M))=\rk(\M)-2$. Both the $f$- and the $h$-polynomial, which is defined by $h_{\Delta(\widehat{\mathcal{L}}(\M))}(x) := f_{\Delta(\widehat{\mathcal{L}}(\M))}(x-1)$, have nonnegative coefficients (the first is clear; for the second we refer to \cite[Section 7.6]{bjorner-matroids}); each of them is conjectured to have only real roots \cite[Conjecture~1.2]{athanasiadis-kalampogia}.

Hameister, Rao, and Simpson observed that the following equality holds for several small cases of $k$ and $n$:
    \[ h_{\Delta(\widehat{\mathcal{L}}(\U_{k,n}))}(x) = \sum_{i=1}^k \binom{n-i-1}{k-i}\, \uH_{\U_{i,n}}(x),\]
and conjectured that this holds for all $k$ and $n$.

The work of Brenti and Welker \cite[Theorem~1]{brenti-welker} provides an explicit formula for the polynomials on the left-hand side in terms of Eulerian polynomials,\footnote{Using Brenti and Welker's terminology, the displayed formula is explained by considering the simplicial complex given by the $(k-1)$-skeleton of an $n$-simplex, and using that the barycentric subdivision yields the order complex of the lattice of flats of the matroid $\U_{k,n}$ without the top element.} concretely,
    \[ h_{\Delta(\widehat{\mathcal{L}}(\U_{k,n}))}(x) = \sum_{j=0}^{k-1} \binom{n}{j} A_{j}(x)\, (x-1)^{k-1-j}.\]
By applying the principle of inclusion-exclusion and the preceding formula, the conjecture of Hameister, Rao, and Simpson is asserting that 
    \begin{align}
        \uH_{\U_{k,n}}(x) &= \sum_{i=1}^k (-1)^{k-i} \binom{n-i-1}{k-i}\, h_{\Delta(\widehat{\mathcal{L}}(\U_{i,n}))}(x)\nonumber\\
        &= \sum_{i=1}^{k} \sum_{j=0}^{i-1} (-1)^{k-i} \binom{n-i-1}{k-i}  \binom{n}{j} A_{j}(x)\, (x-1)^{i-1-j}.\label{eq:hameister-rao-simspon-to-prove}
    \end{align}
    
In what follows, we show how one can manipulate the right-hand side of Corollary~\ref{corollary:uH-uniform} to prove this equality.

\begin{lemma}\label{reduced-char-poly}
    The reduced characteristic polynomial of the uniform matroid $\U_{k,n}$ is
    \[ \overline{\chi}_{\U_{k,n}}(x) = \sum_{j=0}^{k-1} (-1)^j \binom{n-1}{j}\, x^{k-1-j}. \]
\end{lemma}

\begin{lemma}\label{lemma:riordan}
    The following identity of binomial coefficients holds:
    \[ \binom{n+p+q+1}{n} = \sum_{j=0}^n \binom{p+j}{j}\binom{q+n-j}{n-j}.\]
\end{lemma}

A proof of the above identity can be found in Riordan's book \cite[p.~148]{riordan}; alternatively one can prove it just by induction. We are ready to prove the main result of this appendix.

\begin{theorem}
    The following identity holds:
    \[ h_{\Delta(\widehat{\mathcal{L}}(\U_{k,n}))}(x) = \sum_{i=1}^k \binom{n-i-1}{k-i}\, \uH_{\U_{i,n}}(x).\]
\end{theorem}

\begin{proof}
    As we have indicated before, proving the above equality is equivalent to proving the validity of equation \eqref{eq:hameister-rao-simspon-to-prove}.  Let us denote by $(\star)$ the right-hand side of that equation.
    We can expand the term $(x-1)^{k-1-j}$ to obtain
    \begin{align}
        (\star) &=  \sum_{i=1}^k \sum_{j=0}^{i-1} \sum_{\ell = 0}^{i-j-1} (-1)^{k-i+i-j-1-\ell} \binom{n-i-1}{k-i} \binom{n}{j} \binom{i-1-j}{\ell} A_j(x)\, x^\ell,\nonumber
        \intertext{which after interchanging the first two sums becomes}
        &=  \sum_{j=0}^{k-1} \sum_{i=j+1}^{k} \sum_{\ell = 0}^{i-j-1} (-1)^{k-j-\ell-1} \binom{n-i-1}{k-i} \binom{n}{j} \binom{i-1-j}{\ell} A_j(x)\, x^\ell, \nonumber\intertext{and after interchanging the order of the second and third sum and relabelling,}\
        &=  \sum_{i=0}^{k-1} \sum_{\ell = 0}^{k-i-1} \sum_{j = \ell +i+1}^{k} (-1)^{k-i-\ell-1} \binom{n-j-1}{k-j} \binom{n}{i} \binom{j-1-i}{\ell} A_i(x)\, x^\ell.\label{eq:star-equals-sum}
    \end{align}
    On the other hand, by combining Corollary~\ref{corollary:uH-uniform} and Lemma~\ref{reduced-char-poly}, the Hilbert series of the Chow ring of $\U_{k,n}$ is given by
    \begin{align}
        \uH_{\U_{k,n}}(x) = & \sum_{i=0}^{k-1} \sum_{j=0}^{k-i-1}(-1)^j\binom{n}{i} \binom{n-i-1}{j}A_i(x)\, x^{k-i-j-1},\nonumber \intertext{and after reindexing the second sum with $\ell = k-i-j-1$,}
        = & \sum_{i=0}^{k-1} \sum_{\ell = 0}^{k-i-1}(-1)^{k-i-\ell-1}\binom{n}{i} \binom{n-i-1}{k-i-\ell -1} A_i(x)\, x^\ell.\label{eq:uHukn-equals-sum}
    \end{align}
    By Lemma~\ref{lemma:riordan}, after reparameterizing, we have the following equality:
    \[
    \sum_{j= \ell+i+1}^{k} \binom{n-j-1}{k-j} \binom{j-i-1}{\ell} = \binom{n-i-1}{k-i-\ell-1},
    \]
    which allows us to conclude that the expressions of equations \eqref{eq:star-equals-sum} and \eqref{eq:uHukn-equals-sum} are equal, and hence $(\star) = \uH_{\U_{k,n}}(x)$, and the proof is complete.
\end{proof}

\section{}\label{appendixb}

\noindent Let us fix integers $k$ and $n$ such that $2\leq k \leq n$. We can consider the vector of consecutive integers $\mathbf{s}:=(n-k+2,\ldots,n)\in\mathbb{Z}^{k-1}$ and define the set $\mathcal{I}_{k,n}:=\mathcal{I}_{k-1}^{\mathbf{s}}$. In other words, we have 
    \[\mathcal{I}_{k,n} :=  \left\{(e_1,\ldots,e_{k-1})\in\mathbb{Z}^{k-1} : 0\leq e_i < s_i \text{ for $1\leq i\leq k-1$}\right\},\]
where additionally we use the conventions $e_0=e_k=0$ and $s_0=s_k=1$. 

We are interested in proving that the polynomial
    \begin{equation}\label{eq:def-ekn} 
    E_{k,n}(x) := \sum_{\sigma\in \mathcal{I}_{k,n}} (1+x)^{\col(\sigma)}x^{\asc(\sigma)}
    \end{equation} 
is precisely the Hilbert series of the augmented Chow ring of the uniform matroid $\U_{k,n}$. Observe that even though the vector $\mathbf{s}=(n-k+2,\ldots,n)$ is not well-defined whenever $k=0$ or $k=1$, we can make sense of the definition of the polynomial $E_{k,n}(x)$ for those values of $k$, just by setting $E_{1,n}(x):=x+1$ and $E_{0,n}(x):=1$. In particular, note that $E_{1,n}(x)$ is consistent with equation \eqref{eq:def-ekn} by interpreting that there is only one ``empty'' vector in $\mathcal{I}_{1,n}$ leading to a single collision and no descents after adding the left zero coordinate $e_0=0$ and the right zero coordinate $e_1=0$.

The first step towards proving the result of our interest, consists of showing first that the polynomials $E_{k,n}(x)$ satisfy the following recursion.

\begin{lemma}\label{lemma:recurrence-hz}
    The polynomials $E_{k,n}(x)$ satisfy the following recurrence:
    \[ E_{k,n}(x) = E_{k-1,n-1}(x) + x \sum_{j=0}^{k-1} \binom{n-1}{j}\,  A_{j}(x)\, E_{k-1-j,n-1-j}(x).\]
    This, along with the initial conditions $E_{0,n}(x) = 1$ for all $n\geq 0$, determines them uniquely.
\end{lemma}

\begin{proof}
    Since the elements of $\mathbf{s}$ are consecutive integers, a position $i\in [0,k-1]$ is a collision if and only if $e_i=e_{i+1}=0$. Similarly, $i\in [0,k-1]$ is an ascent if and only if $e_i < e_{i+1}$. Notice that each element $\mathbf{e}\in\mathcal{I}_{k,n}$ can be thought of as an element of $\mathcal{I}_{n,n}$ by adding zeros to the left. For instance, $(2,2,3)\in \mathcal{I}_{4,7}$ can be embedded into $\mathcal{I}_{7,7}$ as $(0,0,0,2,2,3)$.
    In particular, following the bijection $\Theta\colon\mathfrak{S}_n\to\mathcal{I}_{n,n}$ of the proof of \cite[Theorem~3.1]{haglund-zhang}, which is defined by $\pi \mapsto (t_{n-1},\ldots,t_1)$ where $t_i=\#\{j>i:\pi_j < \pi_i\}$, we have that the preimage of $\mathcal{I}_{k,n}\hookrightarrow\mathcal{I}_{n,n}$ under $\Theta$ are precisely the permutations $\sigma\in\mathfrak{S}_n$ such that $\sigma_n>\sigma_{n-1}>\cdots>\sigma_{k}$. Let us denote this set of permutations $\mathfrak{S}_{n,k}$, for each $k$ and $n$. In particular, again reasoning as in the proof of \cite[Theorem~3.1]{haglund-zhang}, we obtain
        \[ E_{k,n}(x) = \sum_{\mathbf{e}\in\mathcal{I}_{k,n}} (1+x)^{\col(\mathbf{e})} x^{\asc(\mathbf{e})} = \sum_{\sigma \in \mathfrak{S}_{n,k}} (1+x)^{\operatorname{bad}(\sigma)}x^{\operatorname{des}(\sigma)},\]
    where $\bad(\sigma)=\{ i\in [n] : \sigma_{i-1}<\sigma_i \text{ and } \sigma_i < \sigma_j \text{ for all } j > i\}$, with the convention that $\sigma_0=0$ and $\des(\sigma)=\{i\in [n-1]:\sigma_i>\sigma_{i+1}\}$. 
    
    Notice that if $\sigma=\sigma_1\cdots\sigma_n\in\mathfrak{S}_{n,k}$ has the property that $\sigma_1=1$, then $\overline{\sigma}:=\sigma_2\cdots\sigma_{n}$ can be thought of as an element of $\mathfrak{S}_{n-1,k-1}$, and $\des(\sigma)=\des(\overline{\sigma})$, but $\bad(\sigma)=\bad(\overline{\sigma})+1$. On the other hand, if $\sigma_j=1$ for $j > 1$, then the condition that the last $n-k$ elements of the permutation are in increasing order forces $2\leq j \leq k$. There are $\binom{n-1}{j-1}$ ways of choosing the elements $\sigma_1\cdots\sigma_{j-1}$ and, for every possible choice, this part will not contain any bad elements; at position $j-1$ we have a descent because $\sigma_{j-1}>\sigma_j=1$, and the possible permutations $\sigma_{j+1}\cdots\sigma_n$ are in bijection with the elements of $\mathfrak{S}_{n-j,k-j}$. Everything considered, we have
    \[ E_{k,n}(x) = (1+x)E_{k-1,n-1}(x) + x \sum_{j=2}^{k} \binom{n-1}{j-1}\,  A_{j-1}(x)\, E_{k-j,n-j}(x),\]
    and, after reindexing the sum to be from $j=1$ to $k-1$ and then rearranging, we obtain the desired recursion.
\end{proof}

A proof of the next result can be found, for example, in \cite[Theorem~1.5]{petersen}.

\begin{lemma}\label{lemma:recursion-quadratic-eulerian}
    The Eulerian polynomials satisfy the following recurrence:
    \[ A_{n+1}(x) = A_n(x) +x \sum_{j=0}^{n-1} \binom{n}{j}\, A_j(x)\, A_{n-j}(x).\] 
\end{lemma}

Now we have all the ingredients to prove the main result of this appendix.

\begin{theorem}
    The polynomials $E_{k,n}(x)$ are given by
    \[E_{k,n}(x) = 1 + x\sum_{j=0}^{k-1} \binom{n}{j}\, A_j(x) (1+x+\cdots+x^{k-1-j}).\]
    In particular, they coincide with the Hilbert series of the augmented Chow ring of $\U_{k,n}$.
\end{theorem}

\begin{proof}
    Clearly, the polynomials on the right match the base cases $E_{0,n}(x)=1$ for all $n\geq 0$. Hence, it suffices to show that they satisfy the recurrence of Lemma~\ref{lemma:recurrence-hz}. Let us focus only on the sum appearing in that recursion; later we will multiply by $x$ and add the expression corresponding to $E_{k-1,n-1}(x)$. We have:
    \begin{align*}
        &\sum_{j=0}^{k-1}\binom{n-1}{j}\, A_j(x)\, \left( 1 + x\sum_{i=0}^{k-2-j} \binom{n-1-j}{i}\, A_i(x)\, (1+\cdots+x^{k-2-j-i})\right)\\
        &=\sum_{j=0}^{k-1}\binom{n-1}{j} A_j(x) + x\sum_{j=0}^{k-1}\sum_{i=0}^{k-2-j}\binom{n-1}{j}\binom{n-1-j}{i}\, A_j(x)\,A_i(x)\, (1+\cdots+x^{k-2-j-i}),\intertext{noticing that $\binom{n-1}{j}\binom{n-1-j}{i}=\binom{n-1}{i+j}\binom{i+j}{j}$, and making the change of variables $r=i+j$,}
        &=\sum_{j=0}^{k-1}\binom{n-1}{j} A_j(x) + x\sum_{j=0}^{k-1}\sum_{r=j}^{k-2}\binom{n-1}{r}\binom{r}{j}\, A_j(x)\,A_{r-j}(x)\, (1+\cdots+x^{k-2-r}),\intertext{interchanging the order of summation,}
        &=\sum_{j=0}^{k-1}\binom{n-1}{j} A_j(x) + \sum_{r=0}^{k-2}\binom{n-1}{r}\, (1+\cdots+x^{k-2-r})\,x\sum_{j=0}^{r}\binom{r}{j}\, A_j(x)\,A_{r-j}(x),\intertext{using Lemma~\ref{lemma:recursion-quadratic-eulerian},}
        &=\sum_{j=0}^{k-1}\binom{n-1}{j} A_j(x) + \sum_{r=0}^{k-2}\binom{n-1}{r}\, (1+\cdots+x^{k-2-r})\left(A_{r+1}(x) + (x-1)A_r(x)\right),\intertext{splitting the second sum and using that $(1+\cdots+x^{k-2-r})(x-1)=x^{k-1-r}-1$,}
        &=\sum_{j=0}^{k-1}\binom{n-1}{j} A_j(x) + \sum_{r=0}^{k-2}\binom{n-1}{r}A_{r+1}(x) (1+\cdots+x^{k-2-r})+\sum_{r=0}^{k-2}\binom{n-1}{r}\,(x^{k-1-r}-1)A_r(x),\intertext{cancelling terms in common between the first and the third sums above,}
        &=\binom{n-1}{k-1}A_{k-1}(x) + \sum_{r=0}^{k-2}\binom{n-1}{r}A_{r+1}(x) (1+\cdots+x^{k-2-r})+\sum_{r=0}^{k-2}\binom{n-1}{r}x^{k-1-r}A_r(x),\intertext{grouping the first term and the last sum,}
        &= \sum_{r=0}^{k-2}\binom{n-1}{r}A_{r+1}(x) (1+\cdots+x^{k-2-r})+\sum_{r=0}^{k-1}\binom{n-1}{r}x^{k-1-r}A_r(x),\intertext{reindexing the first sum to start at $r=1$,}
        &=\sum_{r=1}^{k-1}\binom{n-1}{r-1}A_{r}(x) (1+\cdots+x^{k-1-r})+\sum_{r=0}^{k-1}\binom{n-1}{r}x^{k-1-r}A_r(x).
    \end{align*}
    Now, to conclude the proof we multiply this expression by $x$ and add the term corresponding to $E_{k-1,n-1}(x)$ to obtain
    \begin{align*}
        &1+x\sum_{j=0}^{k-2}\binom{n-1}{j}\, A_j(x)\, (1+\cdots+x^{k-2-j})\\ 
        &\enspace + x\left(\sum_{j=1}^{k-1}\binom{n-1}{j-1}A_{j}(x) (1+\cdots+x^{k-1-j})+\sum_{j=0}^{k-1}\binom{n-1}{j}x^{k-1-j}A_j(x)\right), \intertext{in the second summation above (the first of the second line), we can isolate the term corresponding to $x^{k-1-j}$}  
        &=1+x\sum_{j=0}^{k-2}\binom{n-1}{j}\, A_j(x)\, (1+\cdots+x^{k-2-j})\\
        &\; + x\left(\sum_{j=1}^{k-2}\binom{n-1}{j-1}A_{j}(x) (1+\cdots+x^{k-2-j}) +\sum_{j=1}^{k-1}\binom{n-1}{j-1}A_j(x) x^{k-1-j}+\sum_{j=0}^{k-1}\binom{n-1}{j}x^{k-1-j}A_j(x)\right), \intertext{we separate the $j=0$ term from the first and the fourth sum,}
        &=1+x\left((1+\cdots+x^{k-2}) + x^{k-1} \right)\\
        &\enspace + x\left( \sum_{j=1}^{k-2}\binom{n-1}{j}\, A_j(x)\, (1+\cdots+x^{k-2-j}) + \sum_{j=1}^{k-2}\binom{n-1}{j-1}A_{j}(x) (1+\cdots+x^{k-2-j})\right)\\ 
        & \enspace +x \left(\sum_{j=1}^{k-1}\binom{n-1}{j-1}A_j(x) x^{k-1-j}+\sum_{j=1}^{k-1}\binom{n-1}{j}x^{k-1-j}A_j(x)\right), \intertext{we use Pascal's identity $\binom{n-1}{j}+\binom{n-1}{j-1}=\binom{n}{j}$ with the first and the second sum and with the third and the fourth sum,} 
        &= 1+ x\left((1+\cdots+x^{k-1})+ \sum_{j=1}^{k-2} \binom{n}{j} A_j(x) (1+\cdots+x^{k-2-j})+ \sum_{j=1}^{k-1} \binom{n}{j} A_j(x) x^{k-1-j}  \right),\intertext{finally, we can group the two sums, and add the case $j=0$, to obtain the desired expression,}
        &= 1+ x \sum_{j=0}^{k-1} \binom{n}{j} A_j(x) (1+\cdots+x^{k-1-j}).\qedhere
    \end{align*}
\end{proof}

\section*{Acknowledgments}

The authors want to thank Christos Athanasiadis, Tom Braden, Matt Larson, Nicholas Proudfoot, and Botong Wang for several insightful mathematical discussions and for sharing ideas and thoughts that improved this article in many significant ways. They also thank two anonymous referees for very helpful suggestions and comments, and the organizers of ``Geometry meets combinatorics in Bielefeld'' for setting up the conference in which the authors met each other and where this collaboration started.

\bibliographystyle{amsalpha}
\bibliography{bibliography}

\newcommand{\etalchar}[1]{$^{#1}$}
\providecommand{\bysame}{\leavevmode\hbox to3em{\hrulefill}\thinspace}
\providecommand{\MR}{\relax\ifhmode\unskip\space\fi MR }
\providecommand{\MRhref}[2]{%
  \href{http://www.ams.org/mathscinet-getitem?mr=#1}{#2}
}
\providecommand{\href}[2]{#2}
\begin{thebibliography}{BHM{\etalchar{+}}22b}

\bibitem[AHK18]{adiprasito-huh-katz}
Karim Adiprasito, June Huh, and Eric Katz, \emph{Hodge theory for combinatorial geometries}, Ann. of Math. (2) \textbf{188} (2018), no.~2, 381--452. \MR{3862944}

\bibitem[AK06]{ardila-klivans}
Federico Ardila and Caroline~J. Klivans, \emph{The {B}ergman complex of a matroid and phylogenetic trees}, J. Combin. Theory Ser. B \textbf{96} (2006), no.~1, 38--49. \MR{2185977}

\bibitem[AKE23]{athanasiadis-kalampogia}
Christos~A. Athanasiadis and Katerina Kalampogia-Evangelinou, \emph{Chain enumeration, partition lattices and polynomials with only real roots}, Comb. Theory \textbf{3} (2023), no.~1, Paper No. 12, 21. \MR{4565299}

\bibitem[AP20]{amini-piquerez}
Omid {Amini} and Matthieu {Piquerez}, \emph{{Hodge theory for tropical varieties}}, arXiv e-prints (2020), arXiv:2007.07826.

\bibitem[ASW52]{aissen-schoenberg-whitney}
Michael Aissen, I.~J. Schoenberg, and A.~M. Whitney, \emph{On the generating functions of totally positive sequences. {I}}, J. Analyse Math. \textbf{2} (1952), 93--103. \MR{53174}

\bibitem[AT21]{athanasiadis-tzanaki}
Christos~A. Athanasiadis and Eleni Tzanaki, \emph{Symmetric decompositions, triangulations and real-rootedness}, Mathematika \textbf{67} (2021), no.~4, 840--859. \MR{4304414}

\bibitem[Ath18]{athanasiadis-gamma-positivity}
Christos~A. Athanasiadis, \emph{Gamma-positivity in combinatorics and geometry}, S\'{e}m. Lothar. Combin. \textbf{77} ([2016--2018]), Art. B77i, 64. \MR{3878174}

\bibitem[Ath20]{athanasiadis-eulerian}
\bysame, \emph{Binomial {E}ulerian polynomials for colored permutations}, J. Combin. Theory Ser. A \textbf{173} (2020), 105214, 38. \MR{4056091}

\bibitem[BCC{\etalchar{+}}22]{partial-permutohedra}
Roger~E. {Behrend}, Federico {Castillo}, Anastasia {Chavez}, Alexander {Diaz-Lopez}, Laura {Escobar}, Pamela~E. {Harris}, and Erik {Insko}, \emph{{Partial permutohedra}}, arXiv e-prints (2022), arXiv:2207.14253.

\bibitem[BES22]{backman-eur-simpson}
Spencer {Backman}, Christopher {Eur}, and Connor {Simpson}, \emph{{Simplicial generation of Chow rings of matroids}}, J. Eur. Math. Soc. (JEMS) (2022), to appear.

\bibitem[BHM{\etalchar{+}}22a]{semismall}
Tom Braden, June Huh, Jacob~P. Matherne, Nicholas Proudfoot, and Botong Wang, \emph{A semi-small decomposition of the {C}how ring of a matroid}, Adv. Math. \textbf{409} (2022), Paper No. 108646. \MR{4477425}

\bibitem[BHM{\etalchar{+}}22b]{braden-huh-matherne-proudfoot-wang}
Tom {Braden}, June {Huh}, Jacob~P. {Matherne}, Nicholas {Proudfoot}, and Botong {Wang}, \emph{{Singular Hodge theory for combinatorial geometries}}, arXiv e-prints (2022), arXiv:2010.06088.

\bibitem[BJ22]{branden-jochemko}
Petter Br\"{a}nd\'{e}n and Katharina Jochemko, \emph{The {E}ulerian transformation}, Trans. Amer. Math. Soc. \textbf{375} (2022), no.~3, 1917--1931. \MR{4378084}

\bibitem[Bj{\"o}92]{bjorner-matroids}
Anders Bj{\"o}rner, \emph{The homology and shellability of matroids and geometric lattices}, Matroid applications, Encyclopedia Math. Appl., vol.~40, Cambridge Univ. Press, Cambridge, 1992, pp.~226--283. \MR{1165544}

\bibitem[BO92]{brylawski-oxley}
Thomas Brylawski and James Oxley, \emph{The {T}utte polynomial and its applications}, Matroid applications, Encyclopedia Math. Appl., vol.~40, Cambridge Univ. Press, Cambridge, 1992, pp.~123--225. \MR{1165543}

\bibitem[Br{\"a}15]{branden}
Petter Br{\"a}nd{\'e}n, \emph{Unimodality, log-concavity, real-rootedness and beyond}, Handbook of enumerative combinatorics, Discrete Math. Appl. (Boca Raton), CRC Press, Boca Raton, FL, 2015, pp.~437--483. \MR{3409348}

\bibitem[BS21]{branden-solus}
Petter Br\"{a}nd\'{e}n and Liam Solus, \emph{Symmetric decompositions and real-rootedness}, Int. Math. Res. Not. IMRN (2021), no.~10, 7764--7798. \MR{4259159}

\bibitem[BV20]{braden-vysogorets}
Tom Braden and Artem Vysogorets, \emph{Kazhdan-{L}usztig polynomials of matroids under deletion}, Electron. J. Combin. \textbf{27} (2020), no.~1, Paper No. 1.17, 17. \MR{4061059}

\bibitem[BW08]{brenti-welker}
Francesco Brenti and Volkmar Welker, \emph{{$f$}-vectors of barycentric subdivisions}, Math. Z. \textbf{259} (2008), no.~4, 849--865. \MR{2403744}

\bibitem[DCP95]{deconcini-procesi}
Corrado De~Concini and Claudio Procesi, \emph{Wonderful models of subspace arrangements}, Selecta Math. (N.S.) \textbf{1} (1995), no.~3, 459--494. \MR{1366622}

\bibitem[Dot22]{dotsenko}
Vladimir Dotsenko, \emph{Homotopy invariants for {$\overline{\mathcal{M}}_{0,n}$} via {K}oszul duality}, Invent. Math. \textbf{228} (2022), no.~1, 77--106. \MR{4392457}

\bibitem[EHKR10]{etingof-henriques-kamnitzer-rains}
Pavel Etingof, Andr\'{e} Henriques, Joel Kamnitzer, and Eric~M. Rains, \emph{The cohomology ring of the real locus of the moduli space of stable curves of genus 0 with marked points}, Ann. of Math. (2) \textbf{171} (2010), no.~2, 731--777. \MR{2630055}

\bibitem[EHL23]{stellahedral}
Christopher Eur, June Huh, and Matt Larson, \emph{Stellahedral geometry of matroids}, Forum Math. Pi \textbf{11} (2023), Paper No. e24. \MR{4653766}

\bibitem[EMPV24]{elias-miyata-proudfoot-vecchi}
Ben {Elias}, Dane {Miyata}, Nicholas {Proudfoot}, and Lorenzo {Vecchi}, \emph{{Categorical valuative invariants of polyhedra and matroids}}, arXiv e-prints (2024), arXiv:2401.06869.

\bibitem[EPW16]{elias-proudfoot-wakefield}
Ben Elias, Nicholas Proudfoot, and Max Wakefield, \emph{The {K}azhdan-{L}usztig polynomial of a matroid}, Adv. Math. \textbf{299} (2016), 36--70. \MR{3519463}

\bibitem[EW14]{elias-williamson}
Ben Elias and Geordie Williamson, \emph{The {H}odge theory of {S}oergel bimodules}, Ann. of Math. (2) \textbf{180} (2014), no.~3, 1089--1136. \MR{3245013}

\bibitem[FL24]{ferroni-larson}
Luis Ferroni and Matt Larson, \emph{Kazhdan--{L}usztig polynomials of braid matroids}, Comm. Amer. Math. Soc. \textbf{4} (2024), no.~02, 64--79. \MR{4697925}

\bibitem[FNV23]{ferroni-nasr-vecchi}
Luis Ferroni, George~D. Nasr, and Lorenzo Vecchi, \emph{Stressed {H}yperplanes and {K}azhdan--{L}usztig {G}amma-{P}ositivity for {M}atroids}, Int. Math. Res. Not. IMRN (2023), no.~24, 20883--20942. \MR{4681276}

\bibitem[Fr{\"{o}}99]{froberg}
Ralf Fr{\"{o}}berg, \emph{Koszul algebras}, Advances in commutative ring theory ({F}ez, 1997), Lecture Notes in Pure and Appl. Math., vol. 205, Dekker, New York, 1999, pp.~337--350. \MR{1767430}

\bibitem[FS22]{ferroni-schroter}
Luis {Ferroni} and Benjamin {Schr{\"o}ter}, \emph{{Valuative invariants for large classes of matroids}}, arXiv e-prints (2022), arXiv:2208.04893v3.

\bibitem[FY04]{feichtner-yuzvinsky}
Eva~Maria Feichtner and Sergey Yuzvinsky, \emph{Chow rings of toric varieties defined by atomic lattices}, Invent. Math. \textbf{155} (2004), no.~3, 515--536. \MR{2038195}

\bibitem[Gal05]{gal}
Swiatoslaw~R. Gal, \emph{Real root conjecture fails for five- and higher-dimensional spheres}, Discrete Comput. Geom. \textbf{34} (2005), no.~2, 269--284. \MR{2155722}

\bibitem[GK94]{ginzburg-kapranov}
Victor Ginzburg and Mikhail Kapranov, \emph{Koszul duality for operads}, Duke Math. J. \textbf{76} (1994), no.~1, 203--272. \MR{1301191}

\bibitem[GLX{\etalchar{+}}21]{gao-uniform}
Alice L.~L. Gao, Linyuan Lu, Matthew H.~Y. Xie, Arthur L.~B. Yang, and Philip~B. Zhang, \emph{The {K}azhdan-{L}usztig polynomials of uniform matroids}, Adv. in Appl. Math. \textbf{122} (2021), Paper No. 102117, 24. \MR{4160476}

\bibitem[GPY17]{gedeon-proudfoot-young-survey}
Katie Gedeon, Nicholas Proudfoot, and Benjamin Young, \emph{Kazhdan-{L}usztig polynomials of matroids: a survey of results and conjectures}, S\'{e}m. Lothar. Combin. \textbf{78B} (2017), Art. 80, 12. \MR{3678662}

\bibitem[GS20]{gustafsson-solus}
Nils Gustafsson and Liam Solus, \emph{Derangements, {E}hrhart theory, and local {$h$}-polynomials}, Adv. Math. \textbf{369} (2020), 107169, 35. \MR{4091894}

\bibitem[GX21]{gao-xie}
Alice L.~L. Gao and Matthew H.~Y. Xie, \emph{The inverse {K}azhdan-{L}usztig polynomial of a matroid}, J. Combin. Theory Ser. B \textbf{151} (2021), 375--392. \MR{4294228}

\bibitem[Ham17]{hampe}
Simon Hampe, \emph{The intersection ring of matroids}, J. Combin. Theory Ser. B \textbf{122} (2017), 578--614. \MR{3575220}

\bibitem[Han21]{binhan}
Bin Han, \emph{Gamma-positivity of derangement polynomials and binomial {E}ulerian polynomials for colored permutations}, J. Combin. Theory Ser. A \textbf{182} (2021), Paper No. 105459, 22. \MR{4243364}

\bibitem[HRS21]{hameister-rao-simpson}
Thomas Hameister, Sujit Rao, and Connor Simpson, \emph{Chow rings of vector space matroids}, J. Comb. \textbf{12} (2021), no.~1, 55--83. \MR{4195584}

\bibitem[Huh18]{huh-icm18}
June Huh, \emph{Combinatorial applications of the {H}odge-{R}iemann relations}, Proceedings of the {I}nternational {C}ongress of {M}athematicians---{R}io de {J}aneiro 2018. {V}ol. {IV}. {I}nvited lectures, World Sci. Publ., Hackensack, NJ, 2018, pp.~3093--3111. \MR{3966524}

\bibitem[HZ19]{haglund-zhang}
James Haglund and Philip~B. Zhang, \emph{Real-rootedness of variations of {E}ulerian polynomials}, Adv. in Appl. Math. \textbf{109} (2019), 38--54. \MR{3954084}

\bibitem[JKMS19]{juhnke-murai-sieg}
Martina Juhnke-Kubitzke, Satoshi Murai, and Richard Sieg, \emph{Local {$h$}-vectors of quasi-geometric and barycentric subdivisions}, Discrete Comput. Geom. \textbf{61} (2019), no.~2, 364--379. \MR{3903794}

\bibitem[JKU21]{jensen-kutler-usatine}
David Jensen, Max Kutler, and Jeremy Usatine, \emph{The motivic zeta functions of a matroid}, J. Lond. Math. Soc. (2) \textbf{103} (2021), no.~2, 604--632. \MR{4230913}

\bibitem[KL79]{kazhdan-lusztig}
David Kazhdan and George Lusztig, \emph{Representations of {C}oxeter groups and {H}ecke algebras}, Invent. Math. \textbf{53} (1979), no.~2, 165--184. \MR{560412}

\bibitem[KM96]{kontsevich}
M.~Kontsevich and Yuri Manin, \emph{Quantum cohomology of a product}, Invent. Math. \textbf{124} (1996), no.~1-3, 313--339, With an appendix by R. Kaufmann. \MR{1369420}

\bibitem[KNPV23]{karn-proudfoot-nasr-vecchi}
Trevor Karn, George~D. Nasr, Nicholas Proudfoot, and Lorenzo Vecchi, \emph{Equivariant {Kazhdan}-{Lusztig} theory of paving matroids}, Algebr. Comb. \textbf{6} (2023), no.~3, 677--688.

\bibitem[{Lia}22]{liao}
Hsin-Chieh {Liao}, \emph{{Stembridge codes and Chow rings}}, arXiv e-prints (2022), arXiv:2212.05362.

\bibitem[LNR21]{lee-nasr-radcliffe}
Kyungyong Lee, George~D. Nasr, and Jamie Radcliffe, \emph{A {C}ombinatorial {F}ormula for {K}azhdan-{L}usztig {P}olynomials of {S}parse {P}aving {M}atroids}, Electron. J. Combin. \textbf{28} (2021), no.~4, Paper No. 4.44--. \MR{4395926}

\bibitem[Luc75]{lucas}
Dean Lucas, \emph{Weak maps of combinatorial geometries}, Trans. Amer. Math. Soc. \textbf{206} (1975), 247--279. \MR{371693}

\bibitem[MM23]{mastroeni-mccullough}
Matthew {Mastroeni} and Jason {McCullough}, \emph{{Chow rings of matroids are Koszul}}, Math. Ann. \textbf{387} (2023), no.~3--4, 1819--1859.

\bibitem[Oxl11]{oxley}
James Oxley, \emph{Matroid theory}, second ed., Oxford Graduate Texts in Mathematics, vol.~21, Oxford University Press, Oxford, 2011. \MR{2849819}

\bibitem[Pet15]{petersen}
T.~Kyle Petersen, \emph{Eulerian numbers}, Birkh\"{a}user Advanced Texts: Basler Lehrb\"{u}cher. [Birkh\"{a}user Advanced Texts: Basel Textbooks], Birkh\"{a}user/Springer, New York, 2015, With a foreword by Richard Stanley. \MR{3408615}

\bibitem[PP05]{polishchuk-positselski}
Alexander Polishchuk and Leonid Positselski, \emph{Quadratic algebras}, University Lecture Series, vol.~37, American Mathematical Society, Providence, RI, 2005. \MR{2177131}

\bibitem[Pro18]{proudfoot-kls}
Nicholas Proudfoot, \emph{The algebraic geometry of {K}azhdan-{L}usztig-{S}tanley polynomials}, EMS Surv. Math. Sci. \textbf{5} (2018), no.~1-2, 99--127. \MR{3880222}

\bibitem[PRW08]{postnikov-reiner-williams}
Alex Postnikov, Victor Reiner, and Lauren Williams, \emph{Faces of generalized permutohedra}, Doc. Math. \textbf{13} (2008), 207--273. \MR{2520477}

\bibitem[PXY18]{proudfoot-xu-young}
Nicholas Proudfoot, Yuan Xu, and Ben Young, \emph{The {$Z$}-polynomial of a matroid}, Electron. J. Combin. \textbf{25} (2018), no.~1, Paper No. 1.26, 21. \MR{3785005}

\bibitem[Rio79]{riordan}
John Riordan, \emph{Combinatorial identities}, Robert E. Krieger Publishing Co., Huntington, N.Y., 1979, Reprint of the 1968 original. \MR{554488}

\bibitem[Rot71]{rota}
Gian-Carlo Rota, \emph{Combinatorial theory, old and new}, Actes du {C}ongr\`es {I}nternational des {M}ath\'{e}maticiens ({N}ice, 1970), {T}ome 3, 1971, pp.~229--233. \MR{0505646}

\bibitem[RW05]{reiner-welker}
Victor Reiner and Volkmar Welker, \emph{On the {C}harney-{D}avis and {N}eggers-{S}tanley conjectures}, J. Combin. Theory Ser. A \textbf{109} (2005), no.~2, 247--280. \MR{2121026}

\bibitem[Slo18]{oeis}
Neil J.~A. Sloane, \emph{The on-line encyclopedia of integer sequences}, Notices Amer. Math. Soc. \textbf{65} (2018), no.~9, 1062--1074. \MR{3822822}

\bibitem[Sta92]{stanley-local}
Richard~P. Stanley, \emph{Subdivisions and local {$h$}-vectors}, J. Amer. Math. Soc. \textbf{5} (1992), no.~4, 805--851. \MR{1157293}

\bibitem[Sta09]{stapledon}
Alan Stapledon, \emph{Inequalities and {E}hrhart {$\delta$}-vectors}, Trans. Amer. Math. Soc. \textbf{361} (2009), no.~10, 5615--5626. \MR{2515826}

\bibitem[Sta12]{stanley-ec1}
Richard~P. Stanley, \emph{Enumerative combinatorics. {V}olume 1}, second ed., Cambridge Studies in Advanced Mathematics, vol.~49, Cambridge University Press, Cambridge, 2012. \MR{2868112}

\bibitem[Ste21]{stevens-bachelor}
Matthew Stevens, \emph{Real-rootedness conjectures in matroid theory}, 2021, Thesis (Bachelor)--Stanford University, \url{https://mcs0042.github.io/bachelor.pdf}.

\bibitem[SV15]{savage-visontai}
Carla~D. Savage and Mirk\'{o} Visontai, \emph{The {$\mathbf{s}$}-{E}ulerian polynomials have only real roots}, Trans. Amer. Math. Soc. \textbf{367} (2015), no.~2, 1441--1466. \MR{3280050}

\bibitem[SW20]{shareshian-wachs}
John Shareshian and Michelle~L. Wachs, \emph{Gamma-positivity of variations of {E}ulerian polynomials}, J. Comb. \textbf{11} (2020), no.~1, 1--33. \MR{4015851}

\end{thebibliography}

\end{document}